\documentclass[12pt]{amsart}
\usepackage{amssymb}
\usepackage{tikz}
\usepackage[pdfencoding=auto,psdextra]{hyperref}
\usepackage{textgreek}

\oddsidemargin \evensidemargin
\oddsidemargin \evensidemargin
\definecolor{amethyst}{rgb}{0.5, 0.3, 0.7}
\definecolor{kellygreen}{rgb}{0.1, 0.55, 0.05}
\definecolor{darkorange}{rgb}{0.75, 0.2, 0.0}
\definecolor{teal}{rgb}{0, 0.45, 0.45}

   \title{A Shuffle Theorem for Paths Under Any Line}

   \author{J. Blasiak}
   \author{M. Haiman}
   \author{J. Morse}
   \author{A. Pun}
   \author{G. H. Seelinger}

    \address[Blasiak]{
    Dept.\ of Mathematics\\
    Drexel University\\
    Philadelphia, PA}
    \email{jblasiak@gmail.com}

   \address[Haiman]{Dept.\ of Mathematics\\
            University of California\\
            Berkeley, CA}
   \email{mhaiman@math.berkeley.edu}

   \address[Morse]{
   Dept.\ of Mathematics\\
            University of Virginia\\
            Charlottesville, VA}
   \email{morsej@virginia.edu}

   \address[Pun] {Dept.\ of Mathematics\\
            University of Virginia\\
            Charlottesville, VA}
   \email{ayp6e@virginia.edu}

   \address[Seelinger]{Dept.\ of Mathematics\\
   University of Virginia\\
   Charlottesville, VA}
   \email{ghs9ae@virginia.edu}
   
   \thanks{Authors were supported by NSF Grants DMS-1855784 (J.~B.)
    and DMS-1855804 (J.~M. and G.~S.).}

\newtheorem{thm}{Theorem}[subsection]
\newtheorem{lemma}[thm]{Lemma}
\newtheorem{prop}[thm]{Proposition}
\newtheorem{cor}[thm]{Corollary}
\newtheorem{conj}[thm]{Conjecture}

\theoremstyle{definition}
\newtheorem{defn}[thm]{Definition}

\theoremstyle{remark}
\newtheorem{example}[thm]{Example}

\newtheorem{remark}[thm]{Remark}

\newcommand{\NN}{{\mathbb N}}
\newcommand{\QQ}{{\mathbb Q}}

\newcommand{\RR}{{\mathbb R}}
\newcommand{\ZZ}{{\mathbb Z}}

\newcommand{\kk}{{\mathbf k}}
\newcommand{\aA}{{\mathbf a}}
\newcommand{\bb}{{\mathbf b}}

\newcommand{\Hbold}{{\mathbf H}}
\newcommand{\Sbold}{{\mathbf S}}
\newcommand{\sigmabold}{{\boldsymbol \sigma }}

\newcommand{\Acal}{{\mathcal A}}
\newcommand{\Ecal}{{\mathcal E}}

\newcommand{\Gcal}{{\mathcal G}}
\newcommand{\Hcat}{{\mathcal H}}  
\newcommand{\Lcal}{{\mathcal L}}

\newcommand{\ctild}{\tilde{c}}
\newcommand{\Htild}{\tilde{H}}

\DeclareMathOperator{\conv}{conv}

\DeclareMathOperator{\dinv}{dinv}
\DeclareMathOperator{\inv}{inv}
\DeclareMathOperator{\pol}{pol}
\DeclareMathOperator{\supp}{supp}

\DeclareMathOperator{\Inv}{Inv}
\DeclareMathOperator{\GL}{GL}

\DeclareMathOperator{\SSYT}{SSYT}

\DeclareMathOperator{\Stab}{Stab}

\newcommand{\chek}{^{\vee }}

\newcommand{\defeq}{\overset{\text{{\em def}}}{=}}

\newcommand{\fp}[1]{{\boldsymbol \{} #1 {\boldsymbol \}}}

\begin{document}


\begin{abstract}
We generalize the shuffle theorem and its $(km,kn)$ version, as
conjectured by Haglund et al.\ and Bergeron et al., and proven by
Carlsson and Mellit, and Mellit, respectively. In our version
the $(km,kn)$ Dyck paths on the combinatorial side are replaced by
lattice paths lying under a line segment whose $x$ and $y$ intercepts
need not be integers, and the algebraic side is given either by a
Schiffmann algebra operator formula or an equivalent explicit raising
operator formula.

We derive our combinatorial identity as the polynomial truncation of
an identity of infinite series of $\GL_{l}$ characters, expressed in
terms of infinite series versions of LLT polynomials.  The series
identity in question follows from a Cauchy identity for non-symmetric
Hall-Littlewood polynomials.
\end{abstract}

\maketitle

\section{Introduction}
\label{s:intro}

\subsection{Overview} \label{ss:intro} The {\em shuffle theorem},
conjectured by Haglund et al.\ \cite{HaHaLoReUl05} and proven by
Carlsson and Mellit \cite{CarlMell18}, is a combinatorial formula for
the symmetric polynomial $\nabla e_{k}$ as a sum over LLT polynomials
indexed by Dyck paths---that is, lattice paths from $(0,k)$ to $(k,0)$
that lie weakly below the line segment connecting these two points.
Here $e_{k}$ is the $k$-th elementary symmetric function, and $\nabla
$ is an operator on symmetric functions with coefficients in $\QQ
(q,t)$ that arises in the theory of Macdonald polynomials
\cite{BeGaHaTe99}.

The polynomial $\nabla e_{k}$ is significant because it describes the
character of the ring $R_{k}$ of diagonal coinvariants for the
symmetric group $S_{k}$ \cite[Proposition 3.5]{Haiman02}.  The
character of $R_{k}$ had been conjectured in the early 1990s to have
surprising connections with the enumeration of combinatorial objects
such as trees and Dyck paths---for instance, its dimension is equal to
the number $(k+1)^{k-1}$ of rooted trees on $k+1$ labelled vertices,
and the multiplicity of the sign character is equal to the Catalan
number $C_{k}$.  A summary of these early conjectures, contributed by
a number of people, can be found in \cite{Haiman94}.  More precisely,
$\nabla e_{k}$ describes the character of $R_{k}$ as a doubly graded
$S_{k}$ module.  The double grading in $R_{k}$ thus gives rise to
$(q,t)$-analogs of the numbers that enumerate the associated
combinatorial objects.  The conjectures connect specializations of
these $(q,t)$-analogs with previously known $q$-analogs in
combinatorics.

Using results in \cite{GarsHaim96}, the whole suite of earlier
combinatorial conjectures follows from the character formula $\nabla
e_{k}$ and the shuffle theorem.

Along with the formula for $\nabla e_{k}$ given by the shuffle
theorem, Haglund et al.\ conjectured a combinatorial formula for
$\nabla ^{m} e_{k}$ as a sum over LLT polynomials indexed by lattice
paths below the line segment from $(0,k)$ to $(km,0)$.  Extending
this, Bergeron et al.\ \cite{BeGaSeXi16} conjectured, and Mellit
\cite{Mellit16} proved, an identity giving the symmetric polynomial
$e_{k}[-M X^{m,n}]\cdot 1$ as a sum over LLT polynomials indexed by
lattice paths below the segment from $(0,kn)$ to $(km,0)$, for any
pair of positive integers expressed in the form $(km,kn)$ with $m,n$
coprime.  Here we have written $e_{k}[-M X^{m,n}]$ for the operator on
symmetric functions given by a certain element of the elliptic Hall
algebra $\Ecal $ of Burban and Schiffmann \cite{BurbSchi12}, such that
for $n=1$ the symmetric polynomial $e_{k}[-M X^{m,n}]\cdot 1$ reduces
to $\nabla ^{m} e_{k}$.  This notation is explained in \S
\ref{s:schiffmann}.

In this paper we prove an even more general version of the shuffle
theorem, involving a sum over LLT polynomials indexed by lattice paths
lying under the line segment between arbitrary points $(0,s)$ and
$(r,0)$ on the positive real axes, and reducing to the shuffle theorem
of Bergeron et al.\ and Mellit when $(r,s) = (km,kn)$ are integers.

Our generalized shuffle theorem (Theorem \ref{thm:main-G}) is an
identity
\begin{equation}\label{e:main-G-pre}
D_{\bb }\cdot 1 = \sum _{\lambda } t^{a(\lambda )}\, q^{\dinv
_{p}(\lambda )}\, \omega (\Gcal _{\nu (\lambda )}(X; q^{-1})),
\end{equation}
whose ingredients we now summarize briefly, deferring full details to
later parts of the paper.

The sum on the right hand side of \eqref{e:main-G-pre} is over lattice
paths $\lambda $ from $(0,\lfloor s \rfloor)$ to $(\lfloor r \rfloor,
0)$ that lie below the line segment from $(0,s)$ to $(r,0)$.  The
quantity $a(\lambda )$ is the number of lattice squares enclosed
between the path $\lambda $ and the highest such path $\delta $.  We
set $p = s/r$ and define $\dinv _{p}(\lambda )$ to be the number
of `$p$-balanced' hooks in the (French style) Young diagram enclosed
by $\lambda $ and the $x$ and $y$ axes.  A hook is $p$-balanced if a
line of slope $-p$ passes through the segments at the top of its leg
and the end of its arm (Definition \ref{def:dinv} and Figure
\ref{fig:balanced-hook}).

In the remaining factor $\omega (\Gcal _{\nu (\lambda )}(X;q^{-1}))$,
the function $\Gcal _{\nu }(X;q)$ is an `attacking inversions' LLT
polynomial (Definition \ref{def:G-nu}),
and $\omega (h_{\mu }) = e_{\mu }$ is the standard involution on
symmetric functions.  The index $\nu (\lambda )$ is a tuple of one-row
skew Young diagrams of the same lengths as runs of contiguous south
steps in $\lambda $, arranged so that the reading order on boxes of
$\nu (\lambda )$ corresponds to the ordering on south steps in
$\lambda $ given by decreasing values of $y+px$ at the upper endpoint
of each step.

The operator $D_{\bb } = D_{b_{1},\ldots,b_{l}}$ on the left hand side
of \eqref{e:main-G-pre} is one of a family of special elements defined
by Negut \cite{Negut14} in the Schiffmann algebra $\Ecal $.  Letting
$\delta $ again denote the highest path under the given line segment,
the index $\bb $ is defined by taking $b_{i}$ to be the number of
south steps in $\delta $ on the line $x = i-1$.

To recover the previously known cases of the theorem, we take $s = kn$
and $r$ slightly larger than $km$, so $p = n/m-\epsilon $ for a small
$\epsilon >0$.  The segment from $(0,s)$ to $(r,0)$ has the same
lattice paths below it as the segment from $(0,kn)$ to $(km,0)$, and
$\dinv _{p}(\lambda )$ reduces to the version of $\dinv (\lambda )$ in
the original conjectures.  The element $D_{\bb }$ associated to the
highest path $\delta $ below the segment from $(0,kn)$ to $(km,0)$ is
equal to $e_{k}[-M X^{m,n}]$.  Hence, \eqref{e:main-G-pre} reduces to
the $(km,kn)$ shuffle theorem.

\subsection{Preview of the proof}
\label{ss:proof-sketch}

We prove our generalized shuffle theorem by a remarkably simple
method, which we now outline to help orient the reader in following
the details.
In \S \ref{s:schiffmann} we will see that the left hand side of
\eqref{e:main-G-pre}, after applying $\omega $ and evaluating in $l =
\lfloor r \rfloor + 1$ variables $x_{1},\ldots,x_{l}$, becomes the
polynomial part
\begin{equation}\label{e:Db-Hb-pre}
\omega (D_{\bb } \cdot 1)(x_{1},\ldots,x_{l}) = \Hcat _{\bb }(x)_{\pol
}
\end{equation}
of an explicit infinite series of $\GL _{l}$ characters
\begin{equation}\label{e:Hb-pre}
\Hcat _{\bb }(x) =\sum _{w\in S_{l}} w \left(
\frac{x_{1}^{b_{1}}\cdots x_{l}^{b_{l}}\, \prod _{i+1<j}(1 - q\, t\,
x_{i}/x_{j})}{\prod _{i<j}\bigl((1-x_{j}/x_{i})(1 - q\, x_{i}/x_{j})(1
- t\, x_{i}/x_{j}) \bigr)} \right).
\end{equation}
When $\nu $ is a tuple of one-row skew shapes $(\beta _{i})/(\alpha
_{i})$, the LLT polynomial $\Gcal _{\nu }(x;q^{-1})$ is equal, up to a
factor of the form $q^{d}$, to the polynomial part of an infinite
$\GL_{l}$ character series
\begin{equation}\label{e:L=G-pre}
q^{d} \, \Gcal _{\nu }(x;q^{-1}) = \Lcal _{\beta /\alpha }(x;q)_{\pol
}
\end{equation}
defined by Grojnowski and the second author \cite{GrojHaim07}.  In
Theorem \ref{thm:main-L} we establish an identity of infinite series
\begin{equation}\label{e:stable-RHS}
\Hcat _{\bb }(x) = \sum _{a_{1},\ldots,a_{l-1}\geq 0} t^{|\aA |} \Lcal
^{\sigma }_{((b_{l},\ldots,b_{1})+(0;\aA ))/(\aA ;0)}(x;q),
\end{equation}
where $\Lcal ^{\sigma }_{\beta /\alpha }(x;q)$ is a `twisted' variant
of $\Lcal _{\beta /\alpha }(x;q)$ (see \S \ref{s:LLT}).  Then
\eqref{e:main-G-pre} follows once we see that the polynomial part of
the right hand side of \eqref{e:stable-RHS} is equal to the right hand
side of \eqref{e:main-G-pre} with the $\omega $ omitted.

In fact, \eqref{e:L=G-pre} holds when $\alpha _{i}\leq \beta _{i}$ for
all $i$, and otherwise $\Lcal _{\beta /\alpha }(x;q)_{\pol } = 0$.
When we take the polynomial part in \eqref{e:stable-RHS}, this leaves
a non-vanishing term $t^{|\aA |} q^{d} \Gcal _{\nu (\lambda
)}(x;q^{-1})$ for each path $\lambda $ under the given line segment,
with $\aA $ giving the number of lattice squares in each column
between $\lambda $ and the highest path $\delta $, so $t^{|\aA |} =
t^{a(\lambda )}$.  The factor $q^{d}$ from \eqref{e:L=G-pre} turns out
to be precisely $q^{\dinv _{p}(\lambda )}$, yielding
\eqref{e:main-G-pre}.

Finally, the infinite series identity \eqref{e:stable-RHS} from
Theorem \ref{thm:main-L} is essentially a corollary to a Cauchy
identity for non-symmetric Hall-Littlewood polynomials, Theorem
\ref{thm:Cauchy}.  This Cauchy formula is quite general and can be
applied in other situations, some of which we will take up elsewhere.

\subsection{Further remarks}
\label{ss:further-remarks}

(i) The conjectures in \cite{BeGaSeXi16, HaHaLoReUl05} and proofs in
\cite{CarlMell18, Mellit16} use a version of $\dinv (\lambda )$ that
coincides with $\dinv _{p}(\lambda )$ for $p = n/m-\epsilon $.
Alternatively, one can tilt the segment from $(0,kn)$ to $(km,0)$ in
the other direction, taking $r = km$ and $s$ slightly larger than
$kn$, to get a version of the original conjectures with a variant of
$\dinv (\lambda )$ that coincides with $\dinv _{p}(\lambda )$ for $p =
n/m+\epsilon $.  Our theorem implies this version as well.

\smallskip

(ii) Haglund, Zabrocki and the third author \cite{HagMoZa12}
formulated a `compositional' generalization of the original shuffle
conjecture, in which the sum over Dyck paths is decomposed into
partial sums over paths touching the diagonal at specified points.
Bergeron et al.\ also gave a compositional form of their $(km,kn)$
shuffle conjecture in \cite{BeGaSeXi16}.  The proofs by Carlsson and
Mellit \cite{CarlMell18} and Mellit \cite{Mellit16} include the
compositional forms of the conjectures, and indeed this seems to be
essential to their methods.

Our results here do not cover the compositional shuffle 
theorems.
For the generalization to paths under any line, it is not yet clear
whether something like a compositional extension is possible or what
form it might take.

\smallskip

(iii) By \cite[Proposition 5.3.1]{HaHaLoReUl05}, the LLT polynomials
$\Gcal _{\nu (\lambda )}(x;q)$ in \eqref{e:main-G-pre} are $q$ Schur
positive, meaning that their coefficients in the basis of Schur
functions belong to $\NN [q]$.  The right hand side of
\eqref{e:main-G-pre} is therefore $q,t$ Schur positive.  In the cases
corresponding to the $(km,kn)$ shuffle theorem for $k = 1$, this can
also be seen from the representation theoretic interpretation of the
right hand side given by Hikita \cite{Hikita14}.

Identity \eqref{e:main-G-pre} therefore implies that the expression
$D_{\bb }\cdot 1$ on the left hand side is $q,t$ Schur positive.  In
the cases where the left hand side coincides with $\nabla^{m} e_{k}$,
this can be explained using the representation theoretic
interpretations of $\nabla e_{k}$ in \cite{Haiman02} and $\nabla ^{m}
e_{k}$ in \cite[Proposition 6.1.1]{HaHaLoReUl05}.  In \S
\ref{s:positivity} we conjecture a more general condition for $D_{\bb
}\cdot 1$ to be $q,t$ Schur positive.

\smallskip

(iv) The algebraic left hand side of \eqref{e:main-G-pre} is
manifestly symmetric in $q$ and $t$.  Hence, the combinatorial right
hand side is also symmetric in $q$ and $t$.  No direct combinatorial
proof of this symmetry is currently known.

\section{Symmetric functions and \texorpdfstring{$\GL_{l}$}{GL\_l}
characters}
\label{s:symmetric}

This section serves to fix notation and terminology for standard
notions concerning symmetric functions and characters of the general
linear groups $\GL _{l}$.

\subsection{Symmetric functions}
\label{ss:symmetric}

Integer partitions are written $\lambda = (\lambda _{1}\geq \cdots
\geq \lambda _{l})$, possibly with trailing zeroes.  We set $|\lambda
| = \lambda _{1}+\cdots +\lambda _{l}$ and let $\ell(\lambda )$ be the
number of non-zero parts.
We also define
\begin{equation}\label{e:n-lambda}
n(\lambda ) = \sum _{i} (i-1)\lambda _{i}.
\end{equation}
The transpose of $\lambda $ is denoted $\lambda ^{*}$.

The partitions of a given integer $n$ are partially ordered by
\begin{equation}\label{e:dominance}
\lambda \leq \mu \quad \text{if}\quad 
\lambda _{1}+\cdots +\lambda _{k} \leq \mu _{1}+\cdots +\mu _{k}
\text{ for all $k$},
\end{equation}
where the sums include trailing zeroes for $k>\ell(\lambda )$ or
$k>\ell(\mu )$.

The (French style) Young
or Ferrers
diagram of a partition $\lambda $ is the set of lattice points
$\{(i,j)\mid 1\leq j\leq \ell(\lambda ),\; 1\leq i \leq \lambda _{j}
\}$.  We often identify $\lambda$ and its diagram with the set of
lattice squares, or {\em boxes}, with northeast corner at a point
$(i,j)\in \lambda$.
A {\em skew diagram} (or {\em skew shape}) $\lambda /\mu $ is
the difference between the diagram of a partition $\lambda $ and that
of a partition $\mu \subseteq \lambda $ contained in it.
The {\em diagram generator} of $\lambda $ is the polynomial
\begin{equation}\label{e:B-lambda}
B_{\lambda }(q,t) = \sum _{(i,j)\in \lambda} q^{i-1}\, t^{j-1}.
\end{equation}

Let $\Lambda = \Lambda _{\kk }(X)$ be the algebra of symmetric
functions in an infinite alphabet of variables $X =
x_{1},x_{2},\ldots$, with coefficients in the field $\kk = \QQ (q,t)$.
We follow Macdonald's notation \cite{Macdonald95} for various graded
bases of $\Lambda $, such as the elementary symmetric functions
$e_{\lambda } = e_{\lambda _{1}}\cdots e_{\lambda _{k}}$, complete
homogeneous symmetric functions $h_{\lambda } = h_{\lambda _{1}}\cdots
h_{\lambda _{k}}$, power-sums $p_{\lambda } = p_{\lambda _{1}}\cdots
p_{\lambda _{k}}$, monomial symmetric functions $m_{\lambda }$ and
Schur functions $s_{\lambda }$.  The involutory $\kk $-algebra
automorphism $\omega \colon \Lambda \rightarrow \Lambda $ mentioned in
the introduction may be defined by any of the formulas
\begin{equation}\label{e:omega}
\omega \, e_{k} = h_{k},\quad \omega \, h_{k} = e_{k},\quad \omega \,
p_{k} = (-1)^{k-1} p_{k},\quad \omega \, s_{\lambda } = s_{\lambda
^{*}}.
\end{equation}
We also need the symmetric bilinear inner product $\langle -, -
\rangle$ on $\Lambda $ defined by any of
\begin{equation}\label{e:inner-product}
\langle s_{\lambda },s_{\mu } \rangle = \delta _{\lambda ,\mu },\qquad
\langle h_{\lambda },m_{\mu } \rangle = \delta _{\lambda ,\mu },\qquad
\langle p_{\lambda },p_{\mu } \rangle = z_{\lambda }\, \delta
_{\lambda ,\mu },
\end{equation}
where $z_{\lambda } = \prod _{i} r_{i}!\, i^{r_{i}}$ if $\lambda =
(1^{r_{1}},2^{r_{2}},\ldots)$.

We write $f^{\bullet }$ for the operator of multiplication by a
function $f$.  Otherwise, the custom of writing $f$ for both the
operator and the function would make it hard to distinguish between
operator expressions such as $(\omega f)^{\bullet }$ and $\omega \cdot
f^{\bullet }$.  When $f$ is a symmetric function, we write $f^{\perp
}$ for the $\langle -, - \rangle$ adjoint of $f^{\bullet }$.

\subsection{}
\label{ss:plethysm}

We briefly recall the device of {\em plethystic evaluation}.  If $A$
is an expression in terms of indeterminates, such as a polynomial,
rational function, or formal series, we define $p_{k}[A]$ to be the
result of substituting $a^{k}$ for every indeterminate $a$ occurring
in $A$.  We define $f[A]$ for any $f\in \Lambda $ by substituting
$p_{k}[A]$ for $p_{k}$ in the expression for $f$ as a polynomial in
the power-sums $p_{k}$, so that $f\mapsto f[A]$ is a homomorphism.

The variables $q, t$ from our ground field $\kk $ count as
indeterminates.

As a simple example, the plethystic evaluation $f[x_{1}+\cdots
+x_{l}]$ is just the ordinary evaluation $f(x_{1},\ldots,x_{l})$,
since $p_{k}[x_{1}+\cdots +x_{l}] = x_{1}^{k}+\cdots +x_{l}^{k}$.
This also works in infinitely many variables.

When $X = x_{1},x_{2},\ldots$ is the name of an infinite alphabet of
variables, we use $f(X)$, with round brackets, as an abbreviation for
$f(x_{1},x_{2},\ldots)\in \Lambda (X)$.  In this situation we also
make the convention that {\em when $X$ appears inside plethystic
brackets, it means $X = x_{1}+x_{2}+\cdots $.}  With this convention,
$f[X]$ is another way of writing $f(X)$.  

As a second example and caution to the reader, the formula in
\eqref{e:omega} for $\omega \, p_{k}$ implies the identity $\omega
f(X) = f[-z X]|_{z=-1}$.  Note that $f[-z X]|_{z=-1}$ does not reduce
to $f(X)$, as it might at first appear, since specializing the
indeterminate $z$ to a number does not commute with plethystic
evaluation.

Plethystic evaluation of a symmetric infinite series is allowed if the
result converges as a formal series.  The series
\begin{equation}\label{e:Omega}
\Omega = 1 + \sum _{k>0} h_{k} = \exp \sum _{k>0} \frac{p_{k}}{k},
\quad \text{or}\quad \Omega [a_{1}+a_{2}+\cdots -b_{1}-b_{2}-\cdots ]
= \frac{\prod_{i} (1-b_{i})}{\prod_{i} (1-a_{i})}
\end{equation}
is particularly useful.  The classical Cauchy identity can be written
using this notation as
\begin{equation}\label{e:classical-Cauchy}
\Omega [XY] = \sum _{\lambda }s_{\lambda }[X]s_{\lambda }[Y].
\end{equation}
Taking the inner product with $f(X)$ in \eqref{e:classical-Cauchy}
yields $f[A] = \langle \Omega [AX], f(X) \rangle$, which implies
\begin{equation}\label{e:Cauchy-II}
\langle \Omega [AX]\Omega [BX],f(X) \rangle = f[A+B] = \langle \Omega
[BX],f[X+A] \rangle.
\end{equation}
As $B$ is arbitrary, $\Omega [BX]$ is in effect a general symmetric
function, so \eqref{e:Cauchy-II} implies
\begin{equation}\label{e:Omega-perp}
\Omega [AX]^{\perp }f(X) = f[X+A].
\end{equation}
Note that although $\Omega [AX]^{\perp } = \sum _{k} h_{k}[AX]^{\perp
}$ is an infinite series, it converges formally as an operator applied
to any $f\in \Lambda (X)$, since $h_{k}[AX]^{\perp }$ has degree $-k$,
and so kills $f$ for $k \gg 0$.

Identifying $\Lambda $ with a polynomial ring in the power-sums
$p_{k}$, we have
\begin{equation}\label{e:pk-perp}
p_{k}^{\perp } = k \frac{\partial }{\partial p_{k}}.
\end{equation}
In fact, taking $A = z$ and $f = p_{k}$ in \eqref{e:Omega-perp} shows
that $\exp ( \sum (p_{k}^{\perp } z^{k})/k)$ is the operator that
substitutes $p_{k}+z^{k}$ for $p_{k}$  in any polynomial
$f(p_{1},p_{2},\ldots)$.  This operator can also be written $\exp
(\sum z^{k} \frac{\partial }{\partial p_{k}})$, giving
\eqref{e:pk-perp}.

Another consequence of \eqref{e:Omega-perp} is the operator identity
\begin{equation}\label{e:Omega-perp-Omega}
\Omega [AX]^{\perp }\, \Omega [BX]^{\bullet } = \Omega [AB]\, \Omega
[BX]^{\bullet }\, \Omega [AX]^{\perp }
\end{equation}
with notation $\Omega [BX]^{\bullet }$ as in \S \ref{ss:symmetric}.

\subsection{\texorpdfstring{$\GL_{l}$}{GL\_l} characters}
\label{ss:GL-characters}

The weight lattice of $\GL _{l}$ is $X = \ZZ ^{l}$, with Weyl group $W
= S_{l}$ permuting the coordinates.  Letting $\varepsilon
_{1},\ldots,\varepsilon _{l}$ be unit vectors, the positive roots are
$\varepsilon _{i} - \varepsilon _{j}$ for $i<j$, with simple roots
$\alpha _{i} = \varepsilon _{i} - \varepsilon _{i+1}$ for
$i=1,\ldots,l-1$.  The standard pairing on $\ZZ ^{l}$ in which the
$\varepsilon _{i}$ are orthonormal identifies the dual lattice $X^{*}$
with $X$.  Under this identification, the coroots coincide with the
roots, and the simple coroots $\alpha _{i}\chek $ with the simple
roots.  A weight $\lambda \in \ZZ ^{l}$ is dominant if $\lambda
_{1}\geq \cdots \geq \lambda _{l}$; a weight is regular (has trivial
stabilizer in $S_{l}$) if $\lambda_{1},\ldots,\lambda_{l}$ are
distinct.

A {\em polynomial weight} is a
dominant weight $\lambda $ such that $\lambda _{l}\geq 0$.  In other
words, polynomial weights of $\GL _{l}$ are integer partitions of
length at most $l$.

The algebra of virtual $\GL _{l}$ characters $(\kk X)^{W}$ can be
identified with the algebra of symmetric Laurent polynomials $\kk
[x_{1}^{\pm 1},\ldots,x_{l}^{\pm 1}]^{S_{l}}$.  If $\lambda $ is a
polynomial weight, the irreducible character $\chi _{\lambda }$ is
equal to the Schur function $s_{\lambda }(x_{1},\ldots,x_{l})$.  Given
a virtual $\GL _{l}$ character $f(x)= f(x_1,\dots,x_l)  = \sum _{\lambda }c_{\lambda
}\chi _{\lambda }$, 
we denote the partial sum over polynomial weights
$\lambda $ by $f(x)_{\pol }$.  Thus, $f(x)_{\pol }$ is a symmetric
polynomial in $l$ variables.  We also use this notation for infinite
formal sums $f(x)$ of irreducible $\GL _{l}$ characters, in which case
$f(x)_{\pol }$ is a symmetric formal power series.

The Weyl symmetrization operator for $\GL _{l}$ is
\begin{equation}\label{e:Weyl-symmetrization}
\sigmabold \, f(x_{1},\ldots,x_{l}) = \sum _{w\in S_{l}}
w\left(\frac{f(x)}{\prod _{i<j}(1-x_{j}/x_{i})} \right).
\end{equation}
For dominant weights $\lambda$,
the Weyl character formula can be written $\chi _{\lambda } = \sigmabold
(x^{\lambda })$.

Fix a weight $\rho$ such that $\langle \alpha _{i}\chek ,\rho \rangle
= 1$ for each simple coroot $\alpha _{i}\chek $, e.g., $\rho =
(l-1,\ldots,1,0)$.  Although $\rho $ is only unique up to adding a
constant vector, all formulas in which $\rho $ appears will be
independent of the choice.  Let $\mu $ be any weight, not necessarily
dominant.  If $\mu +\rho $ is not regular, then $\sigmabold (x^{\mu }) = 0$.
Otherwise, if $w\in S_{l}$ is the unique permutation
such that $w(\mu +\rho )=\lambda +\rho $ for $\lambda $ dominant,
\begin{equation}\label{e:Weyl-general}
\sigmabold (x^{\mu }) = (-1)^{\ell(w)}\chi _{\lambda }.
\end{equation}

The following identities are useful for working with the Weyl
symmetrization operator.

\begin{lemma}\label{lem:Cauchy-Weyl}
For any $\GL _{l}$ weights $\lambda ,\mu $ and Laurent polynomial
$\phi(x)=\phi (x_{1},\ldots ,x_{l})$, we have
\begin{gather} \label{e:Cauchy-Weyl1}
\overline{\chi_\lambda}\, \prod _{i<j}(1-x_{i}/x_{j}) = \sum _{w\in
S_{l}} (-1)^{\ell(w)}x^{-w(\lambda +\rho ) +\rho },\\
 \label{e:Cauchy-Weyl2}
\langle \chi _{\lambda } \rangle \, \sigmabold(\phi (x)) = \langle
x^{0} \rangle \, \overline{ \chi _{\lambda }} \, \phi (x)\prod _{i<j}
(1-x_{i}/x_{j}),\\
\label{e:Cauchy-Weyl3}
\sigmabold (x^{\mu })_{\pol } = \langle z^{-\mu } \rangle \Omega
[\overline{Z}X] \, \prod _{i<j}(1-z_{i}/z_{j})
\end{gather}
in alphabets $X = x_{1}+\cdots + x_{l}$ and $Z = z_{1}+\cdots +z_{l}$,
where $\overline{Z} = z_{1}^{-1}+\cdots +z_{l}^{-1}$.
\end{lemma}

\begin{proof}
Identity \eqref{e:Cauchy-Weyl1} follows directly from the Weyl
character formula.  To prove \eqref{e:Cauchy-Weyl2}, by linearity it
suffices to verify the formula when $\phi(x)= x^\mu$ is any Laurent
monomial.  Then using \eqref{e:Cauchy-Weyl1}, the right side becomes
$\langle x^{-\mu} \rangle \, \overline{ \chi _{\lambda }} \prod _{i<j}
(1-x_{i}/x_{j}) = \langle x^{-\mu} \rangle \sum _{w\in S_{l}}
(-1)^{\ell(w)}x^{-w(\lambda +\rho ) +\rho }.$ This is $(-1)^{\ell(w)}$
if $\mu+\rho =w(\lambda +\rho )$, or zero if there is no such $w$,
which agrees with $\langle \chi _{\lambda } \rangle \,
\sigmabold(x^\mu)$.

The last identity is then proved from the Cauchy identity
\eqref{e:classical-Cauchy} followed by \eqref{e:Cauchy-Weyl2} (applied
with $z$ in place of $x$ on the right):
\begin{multline}
\langle z^{-\mu } \rangle \Omega [\overline{Z}X] \, \prod
_{i<j}(1-z_{i}/z_{j}) = \sum_{\lambda}  s_\lambda(X) \cdot \langle
z^{-\mu } \rangle s_\lambda[\overline{Z}] \prod _{i<j}(1-z_{i}/z_{j})
\\
= \sum_{\lambda} s_\lambda(X) \cdot \langle \chi _{\lambda } \rangle
\sigmabold(x^\mu) = \sigmabold (x^{\mu })_{\pol }. \qed 
\end{multline}
\let\qed\relax
\end{proof}

\subsection{Hall-Littlewood symmetrization}
\label{ss:Hall-Littlewood}

Given a Laurent polynomial $\phi (x_{1},\ldots,x_{l})$, we define
\begin{equation}\label{e:Hq}
\Hbold _{q}(\phi (x)) = \sigmabold \Bigl( \frac{\phi (x)}{\prod
_{i<j}(1-q\, x_{i}/x_{j})} \Bigr) = \sum _{w\in S_{l}}
w\left(\frac{\phi (x)}{\prod _{i<j}((1-x_{j}/x_{i})(1-q\,
x_{i}/x_{j}))} \right).
\end{equation}
Here and in similar raising operator formulas elsewhere, the factors
$1/(1-q\, x_{i}/x_{j})$ are to be understood as geometric series,
making $\Hbold _{q}(\phi (x))$ an infinite formal sum of irreducible
$\GL _{l}$ characters with coefficients in $\kk $.  Since $1/(1-q\,
x_{i}/x_{j})$ is a power series in $q$, if we expand the coefficients
of $\phi (x)$ as formal Laurent series in $q$, then 
$\Hbold _{q}(\phi(x))$
becomes a formal Laurent series in $q$ over virtual $\GL_{l}$
characters.  This is how the last formula in \eqref{e:Hq} should be
interpreted.

\begin{remark}\label{rem:HL-pols}
The dual Hall-Littlewood polynomials, defined by $H_{\mu }(X;q) = \sum
_{\lambda }K_{\lambda ,\mu }(q) s_{\lambda }$, where $K_{\lambda ,\mu
}(q)$ are the $q$-Kostka coefficients, are given in $l$ variables by
$H_{\mu }(x_{1},\ldots,x_{l};q) = \Hbold _{q}(x^{\mu })_{\pol }$.
This explains our terminology.
\end{remark}

\section{The Schiffmann algebra}
\label{s:schiffmann}

\subsection{}
\label{ss:schiffmann}

We recall here some definitions and results from
\cite{BurbSchi12,FeigTsym11,Negut14,Schiffmann12,SchiVass13}
concerning the elliptic Hall algebra $\Ecal $ of Burban and Schiffmann
\cite{BurbSchi12} (or {\em Schiffmann algebra}, for short) and its
action on the algebra of symmetric functions $\Lambda $, constructed
by Feigin and Tsymbauliak \cite{FeigTsym11} and Schiffmann and
Vasserot \cite{SchiVass13}.

From a certain point of view, this material is unnecessary: two of the
three quantities equated by \eqref{e:main-G-pre} and \eqref{e:Db-Hb-pre} are 
defined without reference to the Schiffmann algebra, and we
could take ``Shuffle Theorem'' to mean the
identity between these two, namely
\begin{equation}\label{e:main-G-Hb-pre}
\Hcat _{\bb }(x)_{\pol} = \sum _{\lambda } t^{a(\lambda )}\, q^{\dinv
_{p}(\lambda )}\, \Gcal _{\nu (\lambda )}(x_{1},\ldots,x_{l}; q^{-1}),
\end{equation}
with $\Hcat _{\bb }(x)$ as in \eqref{e:Hb-pre}.  This identity still
has the virtue of equating the combinatorial right hand side,
involving Dyck paths and LLT polynomials, with a simple algebraic left
hand side that is manifestly symmetric in $q$ and $t$.  Furthermore,
our proof of \eqref{e:main-G-pre} in Theorem \ref{thm:main-G} proceeds
by combining \eqref{e:Db-Hb-pre} with a proof of
\eqref{e:main-G-Hb-pre} (via Theorem \ref{thm:main-L}) that makes no
use of the Schiffmann algebra.

What we need the Schiffmann algebra for is to provide the link between
our shuffle theorem and the classical and $(km,kn)$ versions.  Indeed,
the very definition of the algebraic side in the $(km,kn)$ shuffle
theorem is $e_{k}[-MX^{m,n}]\cdot 1$ for a certain operator
$e_{k}[-MX^{m,n}]\in \Ecal $, while the classical shuffle theorems
refer implicitly to the same operators through the identity $\nabla
^{m} e_{k} = e_{k}[-MX^{m,1}]\cdot 1$.

In this section, we review what is needed to relate the symmetric
functions $\nabla ^{m} e_{k}$ and $(\Hcat _{\bb })_{\pol }$ to the
action of the elements $e_{k}[-MX^{m,n}]$ and $D_{\bb }$ on $\Lambda
$.  For ease of reference, we have collected most of the statements that
will be used elsewhere in the paper in \S \ref{ss:schiffmann-summary},
after the technical development in \S \S
\ref{ss:E}--\ref{ss:Negut-elements}.

\subsection{The algebra \texorpdfstring{$\Ecal $}{E}}
\label{ss:E}

Let $\kk =\QQ (q,t)$.  The Schiffmann algebra $\Ecal $ is generated by
subalgebras $\Lambda (X^{m,n})$ isomorphic to the algebra $\Lambda
_{\kk }$ of symmetric functions, one for each pair of coprime integers
$(m,n)$, and a central Laurent polynomial subalgebra $\kk [c_{1}^{\pm
1}, c_{2}^{\pm 1}]$, subject to some defining relations which we will
not list in full here, but only invoke a few of them as needed.

For purposes of comparison with
\cite{BurbSchi12,Schiffmann12,SchiVass13}, our notation (on the left
hand side of each formula) is related as follows to that in
\cite[Definition 6.4]{BurbSchi12} (on the right hand side).  Note that
our indices $(m,n)\in \ZZ ^{2}$ correspond to transposed indices
$(n,m)$ in \cite{BurbSchi12}.
\begin{equation}\label{e:E-notation}
\begin{gathered}
q = \sigma ^{-1},\quad t = \bar{\sigma }^{-1},\\
c_{1}^{m}c_{2}^{n} = \kappa _{n,m}^{-2},\\
\omega p_{k}(X^{m,n}) = \kappa _{kn,km}^{\varepsilon _{n,m}} u_{kn,km},\\
e_{k}[-\widehat{M} X^{m,n}] = \kappa _{kn,km}^{\varepsilon _{n,m}}
\theta _{kn,km},
\end{gathered}
\end{equation}
where $\varepsilon _{n,m}$, which is equal to $(1-\epsilon _{n,m})/2$
in the notation of \cite{BurbSchi12}, is given by
\begin{equation}\label{e:lower-half-indicator}
\varepsilon _{n,m} = \begin{cases}
1&	\text{$n<0$ or $(m,n)=(-1,0)$}\\
0&	\text{otherwise}.
\end{cases}
\end{equation}
The expression $e_{k}[-\widehat{M} X^{m,n}]$ in \eqref{e:E-notation}
uses plethystic substitution (\S \ref{ss:plethysm}) with
\begin{equation}\label{e:M-hat}
\widehat{M} = (1-\frac{1}{q\, t})M \quad \text{where}\quad M =
(1-q)(1-t).
\end{equation}
The quantity $M$ will be referred to repeatedly.

\subsection{Action of \texorpdfstring{$\Ecal $}{E} on
\texorpdfstring{$\Lambda $}{\textLambda}}
\label{ss:E-action}

A natural action of $\Ecal $ by operators on $\Lambda (X)$ has been
constructed in \cite{FeigTsym11,SchiVass13}.  Actually these
references give several different normalizations of essentially the
same action.  The action we use is a slight variation of the action in
\cite[Theorem 3.1]{SchiVass13}.

To write it down we need to recall some notions from the theory of
Macdonald polynomials.  Let $\Htild _{\mu }(X;q,t)$ denote the
modified Macdonald polynomials \cite{GarsHaim96}, which can be defined
in terms of the integral form Macdonald polynomials $J_{\mu }(X;q,t)$ of
\cite{Macdonald95} by
\begin{equation}\label{e:H-tilde}
\Htild _{\mu }(X;q,t) = t^{n(\mu )} J_{\mu
}[\frac{X}{1-t^{-1}};q,t^{-1}],
\end{equation}
with $n(\mu )$ as in \eqref{e:n-lambda}.  For any symmetric function
$f\in \Lambda $, let $f[B]$, $f[\overline{B}]$ denote the
eigenoperators on the basis $\{\tilde{H_{\mu }} \}$ of $\Lambda $ such
that
\begin{equation}\label{e:eigenoperators}
f[B]\, \Htild _{\mu } = f[B_{\mu }(q,t)]\, \Htild _{\mu },\quad
f[\overline{B}]\, \Htild _{\mu } = f[\overline{B_{\mu }(q,t)}]\, \Htild
_{\mu }
\end{equation}
with $B_{\mu }(q,t)$ as in \eqref{e:B-lambda} and $\overline{B_{\mu
}(q,t)} = B_{\mu }(q^{-1},t^{-1})$.  More generally we use the overbar
to signify inverting the variables in any expression, for example
\begin{equation}\label{e:M-bar}
\overline{M} = (1-q^{-1})(1-t^{-1}).
\end{equation}

The next proposition essentially restates the contents of
\cite[Theorem 3.1 and Proposition 4.10]{SchiVass13} in our notation.
To be more precise, since these two theorems refer to different
actions of $\Ecal $ on $\Lambda $, one must first use the plethystic
transformation in \cite[\S 4.4]{SchiVass13} to express
\cite[Proposition 4.10]{SchiVass13} in terms of the action in
\cite[Theorem 3.1]{SchiVass13}.  Rescaling the generators $u_{\pm
1,l}$ then yields the following.

\begin{prop}\label{prop:E-action}
There is an action of $\Ecal $ on $\Lambda $ characterized as follows.

\noindent
(i) The central parameters $c_{1},c_{2}$ act as scalars
\begin{equation}\label{e:center-action}
c_{1}\mapsto 1,\quad c_{2}\mapsto (q\, t)^{-1}.
\end{equation}
(ii) The subalgebras $\Lambda (X^{\pm 1,0})$ act as
\begin{equation}\label{e:X10-action}
f(X^{1,0})\mapsto (\omega f)[B-1/M],\quad f(X^{-1,0})\mapsto (\omega
f)[\overline{1/M-B}].
\end{equation}
(iii) The subalgebras $\Lambda (X^{0,\pm 1})$ act as
\begin{equation}\label{e:X01-action}
f(X^{0,1})\mapsto f[-X/M]^{\bullet },\quad f(X^{0,-1})\mapsto f(X)^{\perp },
\end{equation}
using the notation in \S \ref{ss:symmetric}.
\end{prop}

\begin{remark}\label{rem:Heisenberg}
The subalgebras $\Lambda (X^{\pm (m,n)})\subseteq \Ecal $ satisfy
Heisenberg relations that depend on the central element
$c_{1}^{m}c_{2}^{n}$.  If $c_{1}^{m}c_{2}^{n} = 1$, the Heisenberg
relation degenerates and $\Lambda (X^{\pm (m,n)})$ commute.  In
particular, the value $c_{1} \mapsto 1$ in \eqref{e:center-action}
makes $\Lambda (X^{\pm 1,0})$ commute, consistent with
\eqref{e:X10-action}.  The value $c_{2}\mapsto 1/(qt)$ makes $\Lambda
(X^{0,\pm 1})$ satisfy Heisenberg relations consistent with
\eqref{e:X01-action}.
\end{remark}

We will show in Proposition \ref{prop:Db}, below, that the elements
$p_{1}[-MX^{1,a}]\in \Ecal $ act on $\Lambda $ as operators $D_{a}$
given by the coefficients $D_{a} = \langle z^{-a} \rangle D(z)$ of a
generating series
\begin{equation}\label{e:Db}
D(z) = \sum _{a\in \ZZ } D_{a} z^{-a}
\end{equation}
defined by either of the equivalent formulas
\begin{equation}\label{e:Dz}
D(-z) = \Omega [-z^{-1}X]^{\bullet } \, \Omega [z M X]^{\perp }\quad
\text{or}\quad D(z) = (\omega \Omega [z^{-1}X])^{\bullet }\, (\omega
\Omega [-z M X])^{\perp },
\end{equation}
using the operator notation from \S \ref{ss:symmetric}.  These
operators $D_{a}$ differ by a sign $(-1)^{a}$ from those studied in
\cite{BeGaHaTe99,GaHaTe99}, and by a plethystic transformation from
operators previously introduced by Jing \cite{Jing94}.

\begin{lemma}\label{lem:pk-Db} We have the identities
\begin{equation}\label{e:pk-Db}
[\, (\omega p_{k}[-X/M])^{\bullet },\, D_{a}\, ] = -D_{a+k},\quad [\,
(\omega p_{k}(X))^{\perp },\, D_{a}\, ] = D_{a-k}.
\end{equation}
\end{lemma}

\begin{proof}
We start with the second identity, which is equivalent to
\begin{equation}\label{e:pk-Db-II}
[\, (\omega p_{k}(X))^{\perp },\, D(z)\, ] = z^{-k}D(z).
\end{equation}
Since all operators of the form $f(X)^{\perp }$ commute with each
other, \eqref{e:pk-Db-II} follows from the definition of $D(z)$ and
\begin{equation}\label{e:pk-Db-II-a}
[\, (\omega p_{k}(X))^{\perp },\, (\omega \Omega [z^{-1}X])^{\bullet }
\,] = z^{-k} (\omega \Omega [z^{-1}X])^{\bullet }.
\end{equation}
To verify the latter identity, note first that \eqref{e:pk-perp} and
$\Omega [z^{-1}X] = \exp \sum _{k>0} p_{k} z^{-k}/k$ imply
\begin{equation}\label{e:pk-Db-II-b}
[\, p_{k}(X)^{\perp },\, \Omega [z^{-1}X]^{\bullet }\, ] = z^{-k} \,
\Omega [z^{-1}X]^{\bullet }.
\end{equation}
Conjugating both sides by $\omega $ and using $(\omega f)^{\bullet } =
\omega \cdot f^{\bullet }\cdot \omega $ and $(\omega f)^{\perp } =
\omega \cdot f^{\perp }\cdot \omega $ gives \eqref{e:pk-Db-II-a}.

For the first identity in \eqref{e:pk-Db}, consider the modified inner
product
\begin{equation}\label{e:M-inner}
\langle f,g \rangle' = \langle f[-MX],g \rangle = \langle
f,g[-MX] \rangle.
\end{equation}
The second equality here, which shows that $\langle -,- \rangle'$ is
symmetric, follows from orthogonality of the power-sums $p_{\lambda
}$.  For any $f$, the operators $f^{\perp }$ and $f[-X/M]^{\bullet }$
are adjoint with respect to $\langle -,- \rangle'$.  Using this and
the definition of $D(z)$, we see that $D(z^{-1})$ is the $\langle -,-
\rangle'$ adjoint of $D(z)$, hence $D_{-a}$ is adjoint to $D_{a}$.
Taking adjoints on both sides of the second identity in
\eqref{e:pk-Db} now implies the first.
\end{proof}

\begin{prop}\label{prop:Db}
In the action of $\Ecal $ on $\Lambda $ given by Proposition
\ref{prop:E-action}, the element $p_{1}[-M X^{1,a}] = -M
p_{1}(X^{1,a})\in \Ecal $ acts as the operator $D_{a}$ defined by
\eqref{e:Db}.
\end{prop}

\begin{proof}
It is known \cite[Proposition 2.4]{Haiman99} that $D_{0}\Htild _{\mu }
= (1-MB_{\mu })\Htild _{\mu }$.  From Proposition \ref{prop:E-action}
(ii), we see that $p_{1}[-M X^{1,0}]$ acts by the same operator,
giving the case $a=0$.  

Among the defining relations of $\Ecal $ are
\begin{equation}\label{e:E-commutators}
[\, \omega p_{k}(X^{0,1}),\, p_{1}(X^{1,a})\, ] = -
p_{1}(X^{1,a+k}),\quad [\, \omega p_{k}(X^{0,-1}),\, p_{1}(X^{1,a})\,
] = p_{1}(X^{1,a-k}).
\end{equation}
Multiplying these by $-M$ and using Proposition \ref{prop:E-action}
(iii) to compare with \eqref{e:pk-Db} reduces the general result to
the case $a=0$.
\end{proof}

\subsection{The operator \texorpdfstring{$\nabla $}{\textnabla}}
\label{ss:nabla}

As in \cite{BeGaHaTe99}, we define an eigenoperator 
$\nabla $ on the
Macdonald basis by
\begin{equation}\label{e:nabla}
\nabla \Htild _{\mu } = t^{n(\mu )}q^{n(\mu ^{*})} \Htild _{\mu },
\end{equation}
with $n(\mu )$ as in \eqref{e:n-lambda}.  Since $t^{n(\mu )}q^{n(\mu
^{*})} = e_{n}[B_{\mu }(q,t)]$ for $n = |\mu |$, we see that $\nabla $
coincides in degree $n$ with $e_{n}[B]$.  Although the operators
$e_{n}[B]$ belong to $\Ecal $ acting on $\Lambda $, the operator
$\nabla $ does not.  Its role, rather, is to internalize a symmetry of
this action.

\begin{lemma}\label{lem:nabla-sym}
Conjugation by the operator $\nabla $ provides a symmetry of the
action of $\Ecal $ on $\Lambda $, namely
\begin{equation}\label{e:nabla-sym}
\nabla\, f(X^{m,n})\, \nabla^{-1} = f(X^{m+n,n}).
\end{equation}
\end{lemma}

\begin{proof}
For $m=\pm 1$, $n=0$, this says that $\nabla $ commutes with the other
Macdonald eigenoperators, which is clear.

It is known from \cite{BurbSchi12} that the group of $k$-algebra
automorphisms of $\Ecal $ includes one which acts on the subalgebras
$\Lambda (X^{m,n})$ by $f(X^{m,n})\mapsto f(X^{m+n,n})$, and on the
central subalgebra $\kk [c_{1}^{\pm 1},c_{2}^{\pm 1}]$ by an
automorphism which fixes the central character in Proposition
\ref{prop:E-action} (i).

The $\Lambda (X^{m,n})$ for $n>0$ are all contained in the subalgebra
of $\Ecal $ generated by the elements $p_{1}(X^{a,1})$.  To prove
\eqref{e:nabla-sym} for $n>0$, it therefore suffices to verify the
operator identity $\nabla \, p_{1}(X^{a,1})\, \nabla^{-1} =
p_{1}(X^{a+1,1})$.

In $\Ecal $ there are relations
\begin{equation}\label{e:E-commmutators-II}
[\, \omega p_{k}(X^{1,0}),\, p_{1}(X^{a,1})\, ] =
p_{1}(X^{a+k,1}),\quad [\, c_{1}^{-k}\omega p_{k}(X^{-1,0}),\,
p_{1}(X^{a,1})\, ] = -p_{1}(X^{a-k,1}).
\end{equation}
Since $\nabla $ commutes with the action of $\Lambda (X^{\pm 1,0})$,
these relations reduce the problem to the case $a=0$, that is, to the
identity $\nabla \, p_{1}(X^{0,1})\, \nabla ^{-1} = p_{1}(X^{1,1})$.
By Propositions \ref{prop:E-action} and \ref{prop:Db}, this is
equivalent to the operator identity $\nabla\cdot p_{1}(X)^{\bullet
}\cdot \nabla ^{-1} = D_{1}$, which is \cite[I.12 (iii)]{BeGaHaTe99}.

We leave the similar argument for the case $n<0$, using
\cite[I.12 (iv)]{BeGaHaTe99}, to the reader.
\end{proof}

\subsection{Shuffle algebra}
\label{ss:shuffle-algebra}

The operators of interest to us belong to the `right half-plane'
subalgebra $\Ecal ^{+}\subseteq \Ecal $ generated by the $\Lambda
(X^{m,n})$ for $m>0$, or equivalently by the elements
$p_{1}(X^{1,a})$.  The subalgebra $\Ecal ^{+}$ acts on $\Lambda $ as
the algebra generated by the operators $D_{a}$.  It was shown in
\cite{SchiVass13} that $\Ecal ^{+}$ is isomorphic to the {\em shuffle
algebra} constructed in \cite{FeigTsym11} and studied further in
\cite{Negut14}, whose definition we now recall.

We fix the rational function
\begin{equation}\label{e:Gamma}
\Gamma (x/y) = \frac{1-q\, t\, x/y}{(1-y/x)(1-q\, x/y)(1-t\, x/y)},
\end{equation}
and define, for each $l$, a $q,t$ analog of the symmetrization
operator $\Hbold _{q}$ in \eqref{e:Hq} by
\begin{multline}\label{e:Hqt}
\Hbold _{q,t}(\phi (x_{1},\ldots,x_{l})) = \sum _{w\in S_{l}} \left(
\phi(x) \cdot \prod _{i<j} \Gamma \bigl(\frac{x_{i}}{x_{j}}\bigr) \right)\\
 = \sigmabold \left(\frac{\phi (x)\, \prod _{i<j} (1 - q\, t\,
x_{i}/x_{j})}{\prod _{i<j} ((1 - q\, x_{i}/x_{j}) (1 - t\,
x_{i}/x_{j}))} \right).
\end{multline}
We write $\Hbold _{q,t}^{(l)}$ when we want to make the number of
variables explicit.

Let $T = T(\kk [x^{\pm 1}])$ be the tensor algebra on the Laurent
polynomial ring $\kk [x^{\pm 1}]$ in one variable, that is, the
non-commutative polynomial algebra with generators corresponding to
the basis elements $x^{a}$ of $\kk [x^{\pm 1}]$ as a vector space.
Identifying $T^{k} = T^{k}(\kk [x^{\pm 1}])$ with $\kk [x_{1}^{\pm
1},\ldots,x_{k}^{\pm 1}]$, the product in $T$ is given by
`concatenation,'
\begin{equation}\label{e:concatenation-product}
f\cdot g = f(x_{1},\ldots,x_{k})g(x_{k+1},\ldots,x_{k+l}),\quad
\text{for $f\in T^{k}$, $g\in T^{l}$}.
\end{equation}
For each $l$, let $I^{l}\subseteq T^{l}$ be the kernel of the
symmetrization operator $\Hbold _{q,t}^{(l)}$.  Since $\Hbold
_{q,t}^{(l)}$ factors through the operator $\Hbold _{q,t}^{(k)}$ in
any $k$ consecutive variables $x_{i+1},\ldots,x_{i+k}$, the graded
subspace $I = \bigoplus _{l}I^{l}\subseteq T$ is a two-sided ideal.
The shuffle algebra is defined to be the quotient $S = T/I$.  Note
that $S$ is generated by its tensor degree $1$ component $S^{1}$ by
definition.  We will not use the second, larger, type of shuffle
algebra that was also introduced in \cite{FeigTsym11,Negut14}.

\begin{prop}[{\cite[Theorem 10.1]{SchiVass13}}]
\label{prop:shuffle-isomorphism} There is an algebra isomorphism $\psi
\colon S \rightarrow \Ecal ^{+}$ given on the generators by $\psi
(x^{a}) = p_{1}[-MX^{1,a}]$.
\end{prop}

For clarity, we note that the factor $-M$ in $p_{1}[-MX^{1,a}] = -M
p_{1}(X^{1,a})$ makes no difference to the statement, but is a
convenient normalization for us, since it makes $\psi
(x_{1}^{a_{1}}\cdots x_{l}^{a_{l}})$ act on $\Lambda $ as
$D_{a_{1}}\cdots D_{a_{l}}$.  We also note that our $\Gamma (x/y)$
differs by a factor symmetric in $x,y$ from the function $g(y/x)$ in
\cite[(10.3)]{SchiVass13}, which makes our shuffle algebra $S$
opposite to the algebra $\Sbold $ in \cite{SchiVass13}.  This is as it
should be, since the isomorphism in \cite{SchiVass13} is from $\Sbold
$ to the `upper half-plane' subalgebra of $\Ecal $ generated by the
elements $p_{1}(X^{a,1})$, and the symmetry
$p_{1}(X^{a,1})\leftrightarrow p_{1}(X^{1,a})$ is an antihomomorphism.

By construction, Laurent polynomials $\phi(x), \phi '(x)$ in variables
$x_{1},\ldots,x_{l}$ define the same element of $S$, or equivalently,
map via $\psi $ to the same element of $\Ecal ^{+}$, if and only if
$\Hbold _{q,t}(\phi ) = \Hbold _{q,t}(\phi ')$.  We can regard $\Hbold
_{q,t}(\phi )$ as an infinite formal sum of $\GL _{l}$ characters with
coefficients in $\kk $, in the same manner as for $\Hbold _{q}$.
Representing elements of $S$, or of $\Ecal ^{+}$, in this way leads to
the following useful formula for describing their action on $1\in
\Lambda $.

\begin{prop}\label{prop:H-formula}
Given a Laurent polynomial $\phi=\phi (x_{1},\ldots,x_{l})$, let $\zeta =
\psi (\phi ) \in \Ecal ^{+}$ be its image under the isomorphism in
Proposition \ref{prop:shuffle-isomorphism}.  Then with the action of
$\Ecal $ on $\Lambda $ given by Proposition \ref{prop:E-action}, we
have
\begin{equation}\label{e:H-formula}
\omega (\zeta \cdot 1)(x_{1},\ldots,x_{l}) = \Hbold _{q,t}(\phi)_{\pol }.
\end{equation}
Moreover, the Schur function expansion of the symmetric function
$\omega (\zeta \cdot 1)(X)$ contains only terms $s_{\lambda }$ with
$\ell(\lambda )\leq l$, so \eqref{e:H-formula} determines $\zeta \cdot
1$.
\end{prop}

\begin{proof}
It suffices to consider the case when $\phi (x)=x^{\aA} =
x_{1}^{a_{1}}\cdots x_{l}^{a_{l}}$ and thus (by Proposition
\ref{prop:Db}) $\zeta$ acts on $\Lambda $ as the operator
$D_{a_{1}}\cdots D_{a_{l}}$.  To find $\zeta \cdot 1$, we use
\eqref{e:Omega-perp-Omega} to compute
\begin{equation}\label{e:Dz-straightened}
D(z_{1})\cdots D(z_{l}) =\bigl(\prod _{i<j} \Omega [-z_{i}/z_{j} M]
\bigr)\, (\omega \Omega [\overline{Z} X])^{\bullet }\, (\omega \Omega
[-Z M X])^{\perp },
\end{equation}
where $Z = z_{1}+\cdots +z_{l}$.  Acting on $1$, applying $\omega $,
and taking the coefficient of $z^{-\aA }$ gives
\begin{multline}\label{e:Dz-on-1}
\omega (\zeta \cdot 1)(X) = \langle z^{-\aA } \rangle \bigl(\prod _{i<j}
\Omega [-z_{i}/z_{j} M] \bigr)\, \Omega [\overline{Z} X] \\
= \langle z^{0} \rangle \bigl(z^{\aA }\prod _{i<j} \frac{1-q\, t\,
z_{i}/z_{j}}{(1-q \, z_{i}/z_{j})(1-t \, z_{i}/z_{j})} \bigr)\, \Omega
[\overline{Z} X] \, \prod _{i<j}(1-z_{i}/z_{j}).
\end{multline}
Since $Z$ has $l$ variables, this implies that all Schur functions $s_\lambda $ in $\omega (\zeta \cdot 1)(X)$ have $\ell(\lambda )\leq l$.
Identity \eqref{e:H-formula} for $\phi (x) = x^{\aA }$
follows by specializing $X$ to $x_1+\cdots +x_l$ and applying \eqref{e:Cauchy-Weyl3}.
\end{proof}

\subsection{Distinguished elements}
\label{ss:Negut-elements}

Given a rational function $\phi (x_{1},\ldots,x_{l})$, it may happen
that we have an identity of rational functions $\Hbold _{q,t}(\phi ) =
\Hbold _{q,t}(\eta )$ for some Laurent polynomial $\eta (x)$.  In this
case, $\Hbold _{q,t}(\phi )$ is the representative of the image of
$\eta$ in $S$, or of $\psi (\eta )\in \Ecal ^{+}$, even though $\phi
(x)$ is not necessarily a Laurent polynomial.  For the shuffle algebra
$S$ under consideration here, Negut \cite[Proposition 6.1]{Negut14}
showed that this happens for
\begin{equation}\label{e:Negut-phi}
\phi (x) = \frac{x_{1}^{b_{1}}\cdots x_{l}^{b_{l}}}{\prod
_{i=1}^{l-1}(1-q\, t\, x_{i}/x_{i+1})}.
\end{equation}
Accordingly, there are distinguished elements
\begin{equation}\label{e:Negut-element}
D_{b_{1},\ldots,b_{l}} = \psi (\eta )\in \Ecal ^{+},
\end{equation}
where $\psi \colon S \rightarrow \Ecal ^{+}$ is the isomorphism in
Proposition \ref{prop:shuffle-isomorphism} and $\eta (x)$ is any
Laurent polynomial such that $\Hbold _{q,t}(\phi) = \Hbold _{q,t}(\eta
)$ for the function $\phi $ in \eqref{e:Negut-phi}.

Negut identified certain of the elements $D_{b_{1},\ldots,b_{l}}$ as
(in our notation) ribbon skew Schur functions $s_{R}[-MX^{m,n}]$.  The
following result is a special case.

\begin{prop}[{\cite[Proposition 6.7]{Negut14}}] \label{prop:ek-vs-Db}
Let $m$, $k$
be positive integers and $n$ any integer with $m,n$ coprime.

For $i=1,\ldots,km$, let $b_{i} = \lceil i n/m \rceil -\lceil (i-1) n/m
\rceil $; if $n \geq 0$, this is
the number of south steps at $x=i-1$ in the highest
south-east lattice path weakly below the line from $(0,kn)$ to
$(km,0)$.  Then
\begin{equation}\label{e:ek-vs-Db}
e_{k}[-M X^{m,n}] = D_{b_{1},\ldots,b_{km}}
\end{equation}
\end{prop}

\begin{lemma}\label{lem:trailing-zero}
For any indices $b_{1},\ldots,b_{l}$ we have
\begin{equation}\label{e:trailing-zero}
D_{b_{1},\ldots,b_{l},0} \cdot 1 = D_{b_{1},\ldots,b_{l}} \cdot 1.
\end{equation}
\end{lemma}

\begin{proof}
Using Proposition \ref{prop:H-formula}, each side of
\eqref{e:trailing-zero} is characterized by its evaluation
\begin{gather}\label{e:trailing-zero-RHS}
\omega (D_{b_{1},\ldots,b_{l}} \cdot 1)(x_{1},\ldots,x_{l}) = \Hbold
^{(l)}_{q,t}\Bigl(\frac{x_{1}^{b_{1}}\cdots x_{l}^{b_{l}}}{\prod
_{i=1}^{l-1}(1 - q\, t\, x_{i}/x_{i+1})}\Bigr)_{\pol }\\
\label{e:trailing-zero-LHS}
\omega (D_{b_{1},\ldots,b_{l},0} \cdot 1)(x_{1},\ldots,x_{l+1}) =
\Hbold ^{(l+1)}_{q,t}\Bigl(\frac{x_{1}^{b_{1}}\cdots x_{l}^{b_{l}}}{(1
- q\, t\, x_{l}/x_{l+1})\prod _{i=1}^{l-1}(1 - q\, t\,
x_{i}/x_{i+1})}\Bigr)_{\pol }.
\end{gather}
Terms with a negative exponent of $x_{l+1}$ inside the parenthesis in
\eqref{e:trailing-zero-LHS} contribute zero after we apply $\Hbold
_{q,t}(-)_{\pol }$. We can therefore drop all but the constant term of
the geometric series factor $1/(1 - q\, t\, x_{l}/x_{l+1})$, since the
other factors are independent of $x_{l+1}$.  This shows that the right
hand side of \eqref{e:trailing-zero-LHS} is the same as in
\eqref{e:trailing-zero-RHS}, except that it has $\Hbold
^{(l+1)}_{q,t}$ in place of $\Hbold ^{(l)}_{q,t}$.  

It follows that $\omega (D_{b_{1},\ldots,b_{l},0} \cdot 1)$ is a
linear combination of Schur functions $s_{\lambda }(X)$ with
$\ell(\lambda )\leq l$, and that $\omega (D_{b_{1},\ldots,b_{l},0} \cdot
1)$ and $\omega (D_{b_{1},\ldots,b_{l}} \cdot 1)$ evaluate to the same
symmetric function in $l$ variables. Hence, they are identical.
\end{proof}

\subsection{Summary}
\label{ss:schiffmann-summary}

Most of what we use from this section in other parts of the paper can
be summarized as follows.

\begin{defn}\label{def:negut-catalanimal} Given $\bb =
(b_{1},\ldots,b_{l})\in \ZZ ^{l}$, the infinite series of $\GL _{l}$
characters $\Hcat _{\bb }(x) = \Hcat
_{b_{1},\ldots,b_{l}}(x_{1},\ldots,x_{l})$ is defined by
\begin{equation}\label{e:Hb}
\Hcat _{\bb }(x) = \Hbold _{q,t}\Bigl(\frac{x^{\bb }}{\prod
_{i=1}^{l-1}(1-q\, t\, x_{i}/x_{i+1})}\Bigr) = \Hbold _{q}\Bigl(
x^{\bb }\, \frac{\prod _{i+1<j}(1-q\, t\, x_{i}/x_{j})}{\prod
_{i<j}(1-t\, x_{i}/x_{j})} \Bigr),
\end{equation}
where $\Hbold _{q,t}$ is given by \eqref{e:Hqt} and $\Hbold _{q}$ by
\eqref{e:Hq}; or where $\Hcat _{\bb }(x)$ is given in fully expanded
form by \eqref{e:Hb-pre}.
\end{defn}

In terms of this definition, we have the following special case of
Proposition \ref{prop:H-formula}, which was stated as identity
\eqref{e:Db-Hb-pre} in the introduction.

\begin{cor}\label{cor:Db-Hb} For the Negut element $D_{\bb }\in \Ecal
$ acting on $\Lambda $, the symmetric function $\omega (D_{\bb }\cdot
1)$ evaluated in $l$ variables is given by
\begin{equation}\label{e:Db-Hb}
\omega (D_{\bb } \cdot 1)(x_{1},\ldots,x_{l}) = \Hcat _{\bb }(x)_{\pol
}\,.
\end{equation}
Moreover, all terms $s_{\lambda }$ in the Schur expansion of $\omega
(D_{\bb } \cdot 1)(X)$ have $\ell(\lambda )\leq l$, so $\omega (D_{\bb }
\cdot 1)$ is determined by its evaluation in $l$ variables.
\end{cor}

In the special cases where the index $\bb $ is the sequence of south
runs in the highest $(km,kn)$ Dyck path, $D_{\bb }\cdot 1$ can also be
expressed as follows.

\begin{cor}\label{cor:ek-vs-Db}
For $i = 1,\ldots,l=km+1$, let $b_{i}$ be the number of south steps at
$x = i-1$ in the highest south-east lattice path weakly below the
line from $(0,kn)$ to $(km,0)$, including $b_{l}=0$.  Then the Negut
element $D_{b_{1},\ldots,b_{l}}$ and the operator $e_{k}[-MX^{m,n}]$
agree when applied to $1\in \Lambda $, that is, we have
\begin{equation}\label{e:ek-vs-Db-1}
D_{b_{1},\ldots,b_{l}} \cdot 1 = e_{k}[-MX^{m,n}]\cdot 1.
\end{equation}
\end{cor}

\begin{proof}
Immediate from Proposition \ref{prop:ek-vs-Db} and Lemma
\ref{lem:trailing-zero}.
\end{proof}

\begin{cor}[also proven in \cite{BeGaSeXi16}]\label{cor:nabla-ek}
In the case $n=1$ of \eqref{e:ek-vs-Db-1}, we further have
\begin{equation}\label{e:nabla-ek}
\nabla ^{m} e_{k}(X) = e_{k}[-MX^{m,1}]\cdot 1.
\end{equation}
\end{cor}

\begin{proof}
By Proposition \ref{prop:E-action} (iii), $e_{k}[-MX^{0,1}]\cdot 1 =
e_{k}(X)$.  Since $\nabla (1) = 1$, the result now follows from Lemma
\ref{lem:nabla-sym}.
\end{proof}

\begin{remark}\label{rem:nabla-formula}
Equations \eqref{e:Hqt}, \eqref{e:trailing-zero}, \eqref{e:Hb},
\eqref{e:Db-Hb}, \eqref{e:ek-vs-Db-1}, and \eqref{e:nabla-ek} imply
the raising operator formula
\begin{equation}\label{e:nabla-formula}
(\omega\nabla^m e_{k+1})(x_1,\dots,x_l) = \sigmabold \left(\frac{x_1
x_{m+1} x_{2m+1}\cdots x_{km+1} \, \prod _{i+1<j} (1 - q\, t\,
x_{i}/x_{j})}{\prod _{i<j} ((1 - q\, x_{i}/x_{j}) (1 -
t\,x_{i}/x_{j}))} \right)_{\pol},
\end{equation}
provided $l \geq km+1$.
\end{remark}

\section{LLT polynomials} \label{s:LLT}

In this section we review the definition of the combinatorial LLT
polynomials $\Gcal _{\nu }(X;q)$, using the attacking inversions
formulation from \cite{HaHaLoReUl05}, which is better suited to our
purposes than the original ribbon tableau formulation in
\cite{LaLeTh97}.

We also define and prove some results on the infinite LLT series
$\Lcal _{\beta /\alpha }(x;q)$ introduced in \cite{GrojHaim07}.  Since
\cite{GrojHaim07} is unpublished, due for revision, and doesn't cover
the `twisted' variants $\Lcal ^{\sigma }_{\beta /\alpha }(x;q)$, we
give here a self-contained treatment of the material we need.

\subsection{Combinatorial LLT polynomials}
\label{ss:G-nu}

The {\em content} of a box $a=(i,j)$ in row $j$, column $i$ of any
skew 
diagram is $c(a)=i-j$.

Let $\nu = (\nu ^{(1)},\ldots,\nu ^{(k)})$ be a tuple of skew
diagrams.  When referring to boxes of $\nu $, we identify $\nu $ with
the disjoint union of the $\nu ^{(i)}$.  Fix $\epsilon >0$ small
enough that $k\epsilon <1$.  The {\em adjusted content} of a box $a\in
\nu ^{(i)}$ is $\ctild (a) = c(a)+i\epsilon $.  A {\em reading order}
is any total ordering of the boxes $a\in \nu $ on which $\ctild (a)$
is increasing.  In other words, the reading order is lexicographic,
first by content, then by the index $i$ for which $a\in \nu ^{(i)}$,
with boxes of the same content in the same component $\nu ^{(i)}$
ordered arbitrarily.

Boxes $a,b\in \nu $ {\em attack} each other if $0<|\ctild (a)-\ctild
(b)|<1$.  If $a\in \nu ^{(i)}$ precedes $b\in \nu ^{(j)}$ in the
reading order, the attacking condition means that either $c(a)=c(b)$
and $i<j$, or $c(b)=c(a)+1$ and $i>j$.  We also say that $a,b$ form an
{\em attacking pair} in $\nu $.

By a semistandard Young tableau on the tuple $\nu $ we mean a map
$T\colon \nu \rightarrow \ZZ _{+}$ which restricts to a semistandard
tableau on each component $\nu ^{(i)}$.  We write $\SSYT (\nu )$ for
the set of these.  The {\em weight} of $T\in \SSYT (\nu )$ is $x^{T} =
\prod _{a\in \nu }x_{T(a)}$.  An {\em attacking inversion} in $T$ is
an attacking pair $a,b$ such that $T(a)>T(b)$, where $a$ precedes $b$
in the reading order.  We define $\inv (T)$ to be the number of
attacking inversions in $T$.

\begin{defn}\label{def:G-nu}
The {\em combinatorial LLT polynomial} indexed by a tuple of skew
diagrams $\nu $ is the generating function
\begin{equation}\label{e:G-nu}
\Gcal_{\nu }(X;q) = \sum _{T\in \SSYT (\nu )}q^{\inv (T)}x^{T}.
\end{equation}
\end{defn}
In \cite{HaHaLoReUl05} it was shown that $\Gcal_{\nu }(X;q^{-1})$
coincides up to a factor $q^{e}$ with a ribbon tableau LLT polynomial
as defined in \cite{LaLeTh97}, and is therefore a symmetric function.
A direct and more elementary proof that $\Gcal_{\nu }(X;q)$ is symmetric
was given in \cite{HaHaLo05}.

It is useful in working with the LLT polynomials $\Gcal _{\nu }(X;q)$
to consider a more general combinatorial formalism, as in \cite[\S
10]{HaHaLo05}.  Let $\Acal = \Acal _{+} \coprod \Acal _{-}$ be a
`signed' alphabet with a {\em positive} letter $v\in \Acal _{+}$ and a
{\em negative} letter $\overline{v}\in \Acal _{-}$ for each $v\in \ZZ
_{+}$, and an arbitrary total ordering on $\Acal $.
 
A {\em super tableau} on a tuple of skew shapes $\nu $ is a map
$T\colon \nu \rightarrow \Acal $, weakly increasing along rows and
columns, with positive letters increasing strictly on columns and
negative letters increasing strictly on rows.  A usual semistandard
tableau is thus the same thing as a super tableau with all entries
positive.  Let $\SSYT _{\pm }(\nu )$ denote the set of super tableaux
on $\nu $.

An attacking inversion in a super tableau is an attacking pair $a,b$,
with $a$ preceding $b$ in the reading order, such that either
$T(a)>T(b)$ in the ordering on $\Acal $, or $T(a)=T(b) = \overline{v}$
with $\overline{v}$ negative.  As before, $\inv (T)$ denotes the
number of attacking inversions.

\begin{lemma}[{\cite[(81--82) and Proposition 4.2]{HaHaLo05}}]
\label{lem:super-G}
We have the identity
\begin{equation}\label{e:super-G}
\omega _{Y} \Gcal _{\nu }[X+Y;q] = \sum _{T\in \SSYT _{\pm }(\nu )}
q^{\inv (T)} x^{T_{+}} y^{T_{-}},
\end{equation}
where the weight is given by
\begin{equation}\label{e:super-weight}
x^{T_{+}} y^{T_{-}} = \prod _{a\in \nu }\begin{cases}
x_{i},&	    T(a) = i\in  \Acal _{+},\\
y_{i},&	    T(a) = \overline{i} \in \Acal _{-}.
\end{cases}
\end{equation}
This holds for any choice of the ordering on the signed alphabet
$\Acal $.
\end{lemma}

\begin{cor}\label{cor:omega-super} We have
\begin{equation}\label{e:omega-super}
\omega\, \Gcal _{\nu }(X;q) = \sum _{T\in \SSYT_{-}(\nu)} q^{\inv
(T)}x^{T},
\end{equation}
where the sum is over super tableaux $T$ with all entries negative,
and we abbreviate $x^{T_{-}}$ to $x^{T}$ in this case.
\end{cor}

\begin{prop}\label{prop:omega-G} Given a tuple of skew
diagrams $\nu = (\nu ^{(1)},\ldots,\nu ^{(k)})$, let $\nu ^{R}= ((\nu
^{(1)})^{R},\ldots,(\nu ^{(k)})^{R})$, where $(\nu ^{(i)})^{R}$ is the
$180^{\circ }$ rotation of the transpose $(\nu ^{(i)})^{*}$,
positioned so that each box in $\nu ^{R}$ has the same content as the
corresponding box in $\nu $.  Then
\begin{equation}\label{e:omega-G}
\omega \, \Gcal _{\nu }(X;q) = q^{I(\nu )}\, \Gcal _{\nu
^{R}}(X;q^{-1}),
\end{equation}
where $I(\nu )$ is the total number of attacking pairs in $\nu $.
\end{prop}

\begin{proof}
Use Corollary \ref{cor:omega-super} on the left hand side, ordering
the negative letters as $\overline{1}>\overline{2}>\cdots $.  Given a
negative tableau $T$ on $\nu $, let $T^{R}$ be the tableau on $\nu
^{R}$ obtained by reflecting the tableau along with $\nu $ and
changing negative letters $\overline{v}$ to positive letters $v$.
Then $T^{R}$ is an ordinary semistandard tableau, and $T\mapsto T^{R}$
is a weight preserving bijection from negative tableaux on $\nu $ to
$\SSYT (\nu ^{R})$.  An attacking pair in $\nu $ is an inversion in
$T$ if and only if the corresponding attacking pair in $\nu ^{R}$ is a
non-inversion in $T^{R}$, hence $\inv (T^{R}) = I(\nu ) - \inv (T)$.
This implies \eqref{e:omega-G}.
\end{proof}

\begin{example}
Consider a tuple $\nu = (\nu ^{(1)},\ldots,\nu ^{(k)})$ in which each
skew shape is a column so that $\nu^R$ is a tuple of rows.  The super
tableau of shape $\nu^R$ with all entries a positive letter $1$ has no
inversions, whereas the distinguished tableau $T$ of shape $\nu$ with
all entries a negative letter $\overline{1}$ has
\begin{equation} \label{e:inv2Tnot}
\inv(T)=I(\nu),
\end{equation}
where $I(\nu)$ is the total number of attacking pairs in $\nu$.
\end{example}

\begin{lemma}\label{lem:finite-variables}
The LLT polynomial $\Gcal_{\nu }(X;q)$ is a linear combination of Schur
functions $s_{\lambda }(X)$ such that $\ell(\lambda )$ is bounded by the
total number of rows in the diagram $\nu$.
\end{lemma}

\begin{proof} Let $r$ be the total number of rows in $\nu $.  It is
equivalent to show that $\omega \, \Gcal_{\nu }(X;q)$ is a linear
combination of monomial symmetric functions $m_{\lambda }(X)$ such
that $\lambda _{1}\leq r$.  By Proposition \ref{prop:omega-G}, $\omega
\, \Gcal_{\nu }(X;q)$ has a monomial term $q^{I(\nu )-\inv (T)} x^{T}$ for
each semistandard tableau $T\in \SSYT(\nu ^{R})$ on the tuple of
reflected shapes $\nu ^{R}$.  Since a letter can appear at most once
in each column of $T$, the exponents of $x^{T}$ are bounded by $r$.
\end{proof}

\subsection{Reminder on Hecke algebras} \label{ss:Hecke}

We recall, in the case of $\GL _{l}$, the Hecke algebra action on the
group algebra of the weight lattice, as in Lusztig \cite{Lusztig89} or
Macdonald \cite{Macdonald03} and due originally to Bernstein and
Zelevinsky.

For $\GL _{l}$, we identify the group algebra $\kk X$ of the weight
lattice $X = \ZZ ^{l}$ with the Laurent polynomial algebra $\kk
[x_{1}^{\pm 1},\ldots,x_{l}^{\pm 1}]$.  Here $\kk $ is any ground
field containing $\QQ (q)$.

The Demazure-Lusztig operators
\begin{equation}\label{e:Demazure-Lusztig}
T_{i} = q s_{i} + (1-q)\frac{1}{1-x_{i+1}/x_{i}} (s_{i}-1)
\end{equation}
for $i = 1,\ldots,l-1$ generate an action of the Hecke algebra 
of $S_{l}$ on $\kk [x_{1}^{\pm 1},\ldots,x_{l}^{\pm 1}]$.
We have normalized the generators so that the quadratic
relations are $(T_{i}-q)(T_{i}+1)=0$.  The elements $T_w$, defined by
$T_w = T_{i_1} T_{i_2}\cdots T_{i_m}$ for any reduced expression
$w=s_{i_1}s_{i_2}\cdots s_{i_m}$, form a $\kk$-basis of the Hecke
algebra, as $w$ ranges over $S_l$.

Let $R_{+}$ be the set of positive roots and $Q_{+} = \NN R_{+}$ the
cone they generate in the root lattice $Q$.  For dominant weights we
define $\lambda \leq \mu $ if $\mu -\lambda \in Q_{+}$.  For
polynomial weights of $\GL_{l}$, this coincides with the standard
partial ordering \eqref{e:dominance} on partitions.

For any weight $\lambda$, let $\lambda _{+}$ denote the dominant
weight in the orbit $S_{l}\cdot \lambda $.

Let $\conv (S_{l}\cdot \lambda )$ be the convex hull of the orbit
$S_{l}\cdot \lambda $ in the coset $\lambda +Q$ of the root lattice,
i.e., the set of weights that occur with non-zero multiplicity in the
irreducible character $\chi _{\lambda _{+}}$.  Note that $\conv
(S_{l}\cdot \lambda )\subseteq \conv (S_{l}\cdot \mu )$ if and only if
$\lambda _{+}\leq \mu _{+}$.

Each orbit $S_{l}\cdot \lambda _{+}$ has a partial ordering induced by
the Bruhat ordering on $S_{l}$.  More explicitly, this ordering is the
transitive closure of the relation $s_{i}\lambda >\lambda $ if
$\langle \alpha _{i}\chek ,\lambda \rangle>0$.  We extend this to a
partial ordering on all of $X = \ZZ ^{l}$ by defining $\lambda \leq
\mu $ if $\lambda _{+}<\mu _{+}$, or if $\lambda _{+} = \mu _{+} $ and
$\lambda \leq \mu $ in the Bruhat order on $S_{l}\cdot \lambda _{+}$.

Suppose $\langle \alpha
_{i}\chek ,\lambda \rangle \geq 0$.  If $\langle \alpha _{i}\chek
,\lambda \rangle = 0$, that is, if $\lambda =s_{i}\lambda $, then
\begin{equation}\label{e:Ti-triangularity-0}
T_{i}\, x^{\lambda } = q\, x^{\lambda }.
\end{equation}
Otherwise, if $\langle \alpha _{i}\chek ,\lambda \rangle >
0$, then
\begin{equation}\label{e:Ti-triagularity-1}
\begin{aligned}
T_{i}\, x^{\lambda } &	\equiv q\, x^{s_{i}\lambda } + (q-1)x^{\lambda},\\
T_{i}\, x^{s_{i}\lambda } &    \equiv x^{\lambda }
\end{aligned}
\end{equation}
modulo the space spanned by monomials $x^{\mu }$ for $\mu $ strictly
between $\lambda $ and $s_{i}\lambda $ on the root string $\lambda
+\ZZ \alpha _{i}$.  Note that $\mu <\lambda $ for these weights $\mu
$, since they lie on orbits strictly inside $\conv (S_{l}\cdot \lambda
)$.  Furthermore, the set of all weights
$\mu \leq s_{i} \lambda $ is
$s_{i}$-invariant and has convex intersection with every root string
$\nu +\ZZ \alpha _{i}$, hence the space $\kk \cdot \{x^{\mu }\mid \mu
\leq s_{i}\lambda \}$ is closed under $T_{i}$.  It follows that if $\langle
\alpha _{i}\chek, \lambda \rangle \ge 0$, then $T_{i}$ applied to any
Laurent polynomial of the form $x^{\lambda}+\sum_{\mu < \lambda}c_\mu
x^{\mu}$ yields a result of the form
\begin{equation}\label{e:Ti-triangularity-2}
T_{i}\, \bigl( x^{\lambda}+\sum_{\mu < \lambda}c_\mu x^{\mu} \bigr) =
q\, x^{s_{i}\lambda } + \sum_{\mu < s_i \lambda} d_\mu x^{\mu}\,.
\end{equation}

\subsection{Non-symmetric Hall-Littlewood polynomials}
\label{ss:non-symm-HL}

For each $\GL _{l}$ weight $\lambda \in \ZZ ^{l}$, we define the {\em
non-symmetric Hall-Littlewood polynomial}
\begin{equation}\label{e:E-lambda}
E_{\lambda }(x;q) = E_{\lambda }(x_1,\cdots x_l;q) = q^{-\ell(w)}T_{w} x^{\lambda _{+}},
\end{equation}
where $w\in S_{l}$ is such that $\lambda =w(\lambda _{+})$.  If
$\lambda $ has non-trivial stabilizer then $w$ is not unique, but it
follows from
\eqref{e:Ti-triangularity-0}--\eqref{e:Ti-triangularity-2} that
$E_{\lambda }(x;q)$ is independent of the choice of $w$ and has the
monic and triangular form
\begin{equation}\label{e:E-monic}
E_{\lambda }(x;q) = x^{\lambda } +\sum _{\mu <\lambda } c_{\mu }
x^{\mu }.
\end{equation}
See Figure \ref{fig:ns Hall-Littlewood} for examples.

\begin{figure}
\begin{align*}
E_{000}& = 1 && F_{000}= 1 \\
E_{100}& =  x_1 && F_{100}= y_1 \\
E_{010}& = (1-q)x_1 + x_2
&& F_{010} = y_2 \\
E_{001}& = (1-q)x_1 + (1-q)x_2 + x_3 &&
F_{001}= y_3 \\
E_{110}& = x_1x_2 &&
F_{110} = y_1y_2 \\
E_{101}& = (1-q)x_1x_2 + x_1x_3 &&
F_{101} = y_1y_3 \\
E_{011}& = (1-q)x_1x_2 + (1-q)x_1x_3 + x_2x_3 &&
F_{011} = y_2y_3 \\
E_{200}& = x_1^2  &&
F_{200} = y_1^2 + (1-q)y_1y_2 + (1-q)y_1y_3\\
E_{020}& = (1-q)x_1^2 + (1-q)x_1x_2 + x_2^2 &&
F_{020} = y_2^2 + (1-q)y_2y_3 \\
E_{002}& =\!\!\! \begin{array}[t]{l}
(1-q) x_1^2 + (1-q)^{2} x_1 x_2 + (1-q)x_1 x_3\\
\; + (1-q) x_2^2 + (1-q) x_2 x_3 + x_3^2
\end{array}
&& F_{002} = y_3^2
\end{align*}
\caption{\label{fig:ns Hall-Littlewood}
Non-symmetric Hall-Littlewood polynomials
$E^{\sigma}_{\aA}(x_{1},x_{2},x_{3};q^{-1})$ and $F^{\sigma
}_{\aA}(y_{1},y_{2},y_{3};q)$ for $l = 3$, $\sigma = 1$, and $|\aA|
\le 2$.}
\end{figure}

For context, we remark that several distinct notions of `non-symmetric
Hall-Littlewood polynomial' can be found in the literature.  Our
$E_{\lambda }$ (and $F_\lambda $, below) coincide with specializations
of non-symmetric Macdonald polynomials considered by Ion in
\cite[Theorem 4.8]{Ion08}.  The twisted variants $E_\lambda ^\sigma$
below are specializations of the `permuted basement' non-symmetric
Macdonald polynomials studied (for $\GL _{l}$) by Alexandersson
\cite{Alexandersson19} and Alexandersson and Sawhney
\cite{AlexSawh19}.
We also note that $E_\lambda(x;q^{-1})$ and $F_\lambda(x;q)$
have coefficients in $\ZZ[q]$ and specialize at $q=0$ to Demazure
characters and Demazure atoms respectively.

For any $\mu \in \RR ^{l}$ we define $\Inv (\mu ) = \{(i<j)\mid \mu
_{i}>\mu _{j} \}$.  In the case of a permutation, $\Inv (\sigma )$ is
then the usual inversion set of $\sigma =(\sigma (1),\ldots,\sigma
(l)) \in S_l$.

Taking $\rho $ as in \S \ref{ss:GL-characters} and $\epsilon >0$
small, the notation $\Inv (\lambda +\epsilon \rho )$ denotes the set
of pairs $i<j$ such that $\lambda _{i}\geq \lambda _{j}$.

Given $\sigma \in S_{l}$, we define {\em twisted} non-symmetric
Hall-Littlewood polynomials
\begin{gather}\label{e:E-twist}
E^{\sigma }_{\lambda }(x;q) = q^{|\Inv (\sigma ^{-1})\cap \Inv
(\lambda +\epsilon \rho )|} T_{\sigma ^{-1}}^{-1} E_{\sigma
^{-1}(\lambda
)}(x;q)\\
\label{e:F-twist}
F^{\sigma }_{\lambda }(x;q) = \overline{E^{\sigma w_{0}}_{-\lambda
}(x;q)} = E^{\sigma w_{0}}_{-\lambda
}(x_{1}^{-1},\ldots,x_{l}^{-1};q^{-1}),
\end{gather}
where $w_{0}\in S_{l}$ is the longest element, given by $w_{0}(i) =
l+1-i$. The normalization in \eqref{e:E-twist} implies the recurrence
\begin{equation}\label{e:E-recurrence}
E^{\sigma }_{\lambda } =
\begin{cases}
q^{-I(\lambda _{i}\leq \lambda _{i+1})}\, T_{i}\, E^{s_{i}\sigma
}_{s_{i}\lambda },
 &	  s_{i}\sigma >\sigma \\
q^{I(\lambda _{i}\geq \lambda _{i+1})}\, T_{i}^{-1}\, E^{s_{i}\sigma
}_{s_{i}\lambda }, & s_{i}\sigma <\sigma,
\end{cases}
\end{equation}
where $I(P)=1$ if $P$ is true, $I(P)=0$ if $P$ is false.  
Together
with the initial condition $E^{\sigma }_{\lambda } = x^{\lambda }$ for
all $\sigma $ if $\lambda =\lambda _{+}$, this determines $E^{\sigma
}_{\lambda }$ for all $\sigma $ and $\lambda $.

\begin{cor}\label{cor:general-monic}
$E^{\sigma }_{\lambda }$ has the monic and triangular form in
\eqref{e:E-monic} for all $\sigma $.
\end{cor}

\begin{proof}
This follows from \eqref{e:Ti-triangularity-2} and \eqref{e:E-recurrence}.
\end{proof}

\begin{prop}\label{prop:twist-orthogonality}
For every $\sigma \in S_{l}$, the $E_{\lambda }^{\sigma }(x;q)$ and
$\overline{F_{\lambda }^{\sigma }(x;q)}$ are dual bases of $\kk
[x_{1}^{\pm 1},\ldots,x_{l}^{\pm 1}]$ with respect to the inner
product defined by
\begin{equation}\label{e:HL-inner}
\langle f,g \rangle_{q} = \langle x^{0} \rangle\, f\, g\, \prod _{i<j}
\frac{1-x_{i}/x_{j}}{1 - q^{-1} x_{i}/x_{j}}.
\end{equation}
In other words, $\langle E_{\lambda }^{\sigma }, \overline{F_{\mu
}^{\sigma }} \rangle _{q} = \delta _{\lambda \mu }$ for all $\lambda
,\mu \in \ZZ ^{l}$ and all $\sigma \in S_{l}$.
\end{prop}

To prove Proposition \ref{prop:twist-orthogonality} we need the
following lemma.

\begin{lemma}\label{lem:Ti-self-adjoint}
The Demazure-Lusztig operators $T_{i}$ in \eqref{e:Demazure-Lusztig}
are self-adjoint with respect to $\langle -,- \rangle_{q}$.
\end{lemma}

\begin{proof}
It's the same to show that $T_{i}-q$ is self-adjoint.  A bit of
algebra gives
\begin{equation}\label{e:Ti-q}
T_{i}-q =q\, \frac{1-q^{-1} x_{i}/x_{i+1}}{1-x_{i}/x_{i+1}}(s_{i}-1),
\end{equation}
and therefore
\begin{equation}\label{e:Ti-q-adjoint}
\langle (T_{i}-q)f,g \rangle_{q} = q\, \langle x^{0} \rangle\,
(s_{i}(f)\, g-f\, g) \prod \frac{1-x_{j}/x_{k}}{1-q^{-1}x_{j}/x_{k}},
\end{equation}
where the product is over $j<k$ with $(j,k) \not = (i,i+1)$.  We want
to show that this is symmetric in $f$ and $g$, i.e., that the right
hand side is unchanged if we replace $s_{i}(f)\, g$ with $f\,
s_{i}(g)$.  Let $\Delta $ denote the product factor in
\eqref{e:Ti-q-adjoint}, and note that $\Delta $ is symmetric in
$x_{i}$ and $x_{i+1}$.  The constant term $\langle x^{0} \rangle\,
\varphi (x)$ of any $\varphi (x_{1}, \ldots, x_{l})$ is equal to
$\langle x^{0} \rangle\, s_{i} (\varphi (x))$.  In particular,
$\langle x^{0} \rangle s_{i}(f) g\, \Delta = \langle x^{0} \rangle f
s_{i}(g)\, \Delta $, which implies the desired result.
\end{proof}

\begin{proof}[Proof of Proposition \ref{prop:twist-orthogonality}]
The desired identity is just a tidy notation for $\langle E_{\lambda
}^{\sigma },E_{-\mu }^{\sigma w_{0}} \rangle _{q} = \delta _{\lambda
\mu }$.

By \eqref{e:E-recurrence}, for every $i$, we have either $\langle
E_{\lambda }^{\sigma }, E_{-\mu }^{\sigma w_{0}} \rangle _{q} =
q^{e}\, \langle T_{i} E_{s_{i}\lambda }^{s_{i}\sigma },
T_{i}^{-1}E_{-s_{i} \mu }^{s_{i} \sigma w_{0}} \rangle _{q}$ or
$\langle E_{\lambda }^{\sigma }, E_{-\mu }^{\sigma w_{0}} \rangle _{q}
= q^{e}\, \langle T_{i}^{-1} E_{s_{i} \lambda }^{s_{i} \sigma }, T_{i}
E_{- s_{i} \mu }^{s_{i} \sigma w_{0}} \rangle _{q}$, depending on
whether $s_{i}\sigma >\sigma $ or $s_{i}\sigma <\sigma $, for some
exponent $e$.  Moreover, if $\lambda =\mu $, then $q^{e} = 1$.

Since $T_{i}$ is self-adjoint, we get $\langle E_{\lambda }^{\sigma },
E_{-\mu }^{\sigma w_{0}} \rangle _{q} = q^{e} \langle E_{s_{i}\lambda
}^{s_{i}\sigma }, E_{-s_{i} \mu }^{s_{i} \sigma w_{0}} \rangle _{q}$
in either case.  Repeating this gives an identity
\begin{equation}\label{e:transform-inner}
\langle E^{\sigma }_{\lambda }, E^{\sigma w_{0}}_{-\mu } \rangle _{q} =
q^{e} \langle E^{v \sigma }_{v \lambda }, E^{v \sigma w_{0}}_{-v \mu }
\rangle _{q}
\end{equation}
for all $\lambda ,\mu \in \ZZ ^{l}$ and all $\sigma ,v\in S_{l}$,
again with $q^{e} = 1$ if $\lambda =\mu $.

Choose $v\in S_{l}$ such that $\mu _{-} = v (\mu )$ is antidominant.
Then \eqref{e:transform-inner} gives
\begin{equation}\label{e:reduce-mu}
\langle E^{\sigma }_{\lambda }, E^{\sigma w_{0}}_{-\mu } \rangle _{q}
= q^{e} \langle E^{v \sigma }_{v \lambda }, x^{-(\mu _{-})} \rangle
_{q} =q^{e}\, \langle x^{\mu _{-}} \rangle \, \Delta\, E^{v
\sigma }_{v \lambda },
\end{equation}
where $\Delta $ is the product factor
in \eqref{e:HL-inner}.  Let
$\supp (f)$ denote the set of weights $\nu $ for which $x^{\nu }$
occurs with non-zero coefficient in $f$.  Since $\supp (\Delta )
= Q_{+}$, and $\supp (E^{v \sigma }_{v \lambda })\subseteq \conv
(S_{l}\cdot \lambda )$, it follows from \eqref{e:reduce-mu} that if
$\langle E^{\sigma }_{\lambda }, E^{\sigma w_{0}}_{-\mu } \rangle
_{q}\not =0$, then $(\mu _{-}-Q_{+})\cap \conv (S_{l}\cdot \lambda
)\not =\emptyset $ and therefore $\mu _{-} - \lambda _{-}\in
Q_{+}$. Since $w_{0}(Q_{+}) = -Q_{+}$, this is equivalent to $\lambda
_{+}\geq \mu _{+}$.

By symmetry, exchanging $\lambda $ with $-\mu $ and $\sigma $ with
$\sigma w_{0}$, if $\langle E^{\sigma }_{\lambda }, E^{\sigma
w_{0}}_{-\mu } \rangle _{q}\not =0$ then we also have $(-\lambda )_{-}
- (-\mu )_{-}\in Q_{+}$, hence $\lambda _{+}-\mu _{+}\in -Q_{+}$, that
is, $\lambda _{+}\leq \mu _{+}$.  Hence, $\langle E^{\sigma }_{\lambda
}, E^{\sigma w_{0}}_{-\mu } \rangle _{q}\not =0$ implies $\lambda _{+}
= \mu _{+}$, so $\lambda $ and $\mu $ belong to the same $S_{l}$
orbit.  This reduces the problem to the case that $S_{l}\cdot\lambda =
S_{l}\cdot \mu $.

In this case, $(\mu _{-}-Q_{+})\cap \conv (S_{l}\cdot \lambda ) =
\{\mu _{-} \}$.  Furthermore, if $\lambda \not =\mu $, then $v \lambda
\not =\mu _{-}$, and Corollary \ref{cor:general-monic} implies that
$(\mu _{-}-Q_{+})\cap \supp (E^{v \sigma }_{v \lambda }) = \emptyset
$, hence $\langle E^{\sigma }_{\lambda }, E^{\sigma w_{0}}_{-\mu }
\rangle _{q} =0$.

If $\lambda =\mu $, then the right hand side of \eqref{e:reduce-mu}
reduces to $\langle x^{\mu _{-}} \rangle\, \Delta\, E^{v \sigma
}_{\mu _{-}}$.  Since $\supp (\Delta)=Q_{+}$ and $\supp (E^{v
\sigma }_{\mu _{-}})\subset \mu _{-}+Q_{+}$, only the constant term of
$\Delta$ and the $x^{\mu _{-}}$ term of $E^{v\sigma }_{\mu
_{-}}$ contribute to the coefficient of $x^{\mu _{-}}$ in $\Delta \,
E^{v \sigma }_{\mu _{-}}$, and we have $\langle x^{\mu _{-}}
\rangle E^{v\sigma }_{\mu _{-}} = 1$ by Corollary
\ref{cor:general-monic}.  Hence, $\langle E^{\sigma }_{\lambda },
E^{\sigma w_{0}}_{-\lambda } \rangle _{q} =1$.
\end{proof}

\begin{lemma}\label{lem:factorization}
Given $\lambda \in \ZZ ^{l}$, suppose there is an index $k$ such that
$\lambda _{i}\geq \lambda _{j}$ for all $i\leq k$ and $j>k$.  Given
$\sigma \in S_{l}$, let $\sigma _{1}\in S_{k}$ and $\sigma _{2}\in
S_{l-k}$ be the permutations such that $\sigma _{1}(1),\ldots,\sigma
_{1}(k)$ are in the same relative order as $\sigma (1),\ldots,\sigma
(k)$, and $\sigma _{2}(1),\ldots,\sigma _{2}(l-k)$ are in the same
relative order as $\sigma (k+1),\ldots,\sigma (l)$.  Then
\begin{equation}\label{e:factorization}
E^{\sigma ^{-1}}_{\lambda }(x_{1},\ldots,x_{l};q) = E^{\sigma_{1}
^{-1}}_{(\lambda_{1},\ldots,\lambda _{k}) }(x_{1},\ldots,x_{k};q)\,
E^{\sigma_{2} ^{-1}}_{(\lambda_{k+1},\ldots,\lambda _{l})
}(x_{k+1},\ldots,x_{l};q).
\end{equation}
\end{lemma}

\begin{proof}
If $\lambda $ is dominant, the result is trivial.  Otherwise, the
recurrence \eqref{e:E-recurrence} determines $E^{\sigma
^{-1}}_{\lambda }$ by induction on $|\Inv (-\lambda )|$.  The
condition on $\lambda $ implies that we only need to use
\eqref{e:E-recurrence} for $i\not =k$, that is, for $s_{i}$ in the
Young subgroup $S_{k}\times S_{l-k}\subset S_{l}$.  For $i\not =k$,
the right hand side of \eqref{e:factorization} satisfies the same
recurrence.
\end{proof}

\subsection{LLT series}
\label{ss:L-beta-alpha}

\begin{defn}\label{def:L-beta-alpha}
Given $\GL _{l}$ weights $\alpha ,\beta \in \ZZ ^{l}$ and a
permutation $\sigma \in S_{l}$, the {\em LLT series} $\Lcal^{\sigma }
_{\beta /\alpha }(x_{1},\ldots,x_{l};q)$ is the infinite formal sum of
irreducible $\GL _{l}$ characters in which the coefficient of $\chi
_{\lambda }$ is defined by
\begin{equation}\label{e:L-beta-alpha}
\langle \chi _{\lambda } \rangle\, \Lcal ^{\sigma ^{-1}}_{\beta /\alpha
}(x;q^{-1}) = \langle E^{\sigma }_{\beta } \rangle \, \chi
_{\lambda } \cdot E^{\sigma }_{\alpha },
\end{equation}
where $E^{\sigma }_{\lambda }(x;q)$ are the twisted non-symmetric
Hall-Littlewood polynomials from \S \ref{ss:non-symm-HL}.
\end{defn}

We remark that the coefficients of $E^{\sigma }_{\lambda }(x;q)$ are
polynomials in $q^{-1}$, so the convention of inverting $q$ in
\eqref{e:L-beta-alpha} makes the coefficients of $\Lcal ^{\sigma
}_{\beta /\alpha }(x;q)$ polynomials in $q$.  Inverting $\sigma $ as
well leads to a more natural statement and proof in
Corollary~\ref{cor:G-versus-L}, below.

\begin{prop}\label{prop:L-formula}
We have the formula
\begin{equation}\label{e:L-formula}
\Lcal ^{\sigma }_{\beta /\alpha }(x;q) = \Hbold _{q}(w_{0}(F^{\sigma
^{-1} }_{\beta }(x;q) \overline{E^{\sigma ^{-1}}_{\alpha }(x;q)})),
\end{equation}
where $\Hbold _{q}$ is the Hall-Littlewood symmetrization operator in
\eqref{e:Hq} and $w_{0}(i)=l+1-i$ is the longest element in $S_{l}$.
\end{prop}

\begin{proof}
By Proposition \ref{prop:twist-orthogonality}, the coefficient $
\langle E^{\sigma ^{-1}}_{\beta }(x;q^{-1}) \rangle \, \chi _{\lambda
} \cdot E^{\sigma ^{-1}}_{\alpha }(x;q^{-1})$ of $\chi _{\lambda }$ in
$\Lcal ^{\sigma }_{\beta /\alpha }$ is given by the constant term
\begin{equation}\label{e:LLT-coef}
\langle x^{0} \rangle\, \chi _{\lambda }\, F^{\sigma ^{-1}}_{\beta
}(x^{-1};q)\, E^{\sigma ^{-1}}_{\alpha }(x;q^{-1}) \prod
_{i<j}\frac{1-x_{i}/x_{j}}{1-q\, x_{i}/x_{j}}.
\end{equation}
Substituting $x_{i}\mapsto x_{i}^{-1}$ and applying $w_{0}$, this is
equal to
\begin{equation}\label{e:LLT-coef-bis}
\langle x^{0} \rangle\, \overline{\chi _{\lambda }}\, w_{0}(F^{\sigma
^{-1}}_{\beta }(x;q) \overline{E^{\sigma ^{-1}}_{\alpha }(x;q)}) \prod
_{i<j}\frac{1-x_{i}/x_{j}}{1-q\, x_{i}/x_{j}}.
\end{equation}
Considering this expression as a formal Laurent series in $q$ and
applying \eqref{e:Cauchy-Weyl2} coefficient-wise yields
\begin{equation}
\langle \chi _{\lambda } \rangle \sigmabold \left(
w_{0}(F^{\sigma^{-1}}_{\beta }(x;q) \overline{E^{\sigma ^{-1}}_{\alpha
}(x;q)}) \prod_{i<j}\frac{1}{1-q\, x_{i}/x_{j}} \right),
\end{equation}
which is $\langle \chi _{\lambda } \rangle \Hbold _{q}(w_{0}(F^{\sigma
^{-1} }_{\beta }(x;q) \overline{E^{\sigma ^{-1}}_{\alpha }(x;q)}))$,
as desired.
\end{proof}

\begin{remark}\label{rem:general-type}
All definitions and results in \S\S
\ref{ss:GL-characters}--\ref{ss:Hall-Littlewood} and
\ref{ss:Hecke}--\ref{ss:L-beta-alpha} extend naturally from the weight
lattice and root system of $\GL_{l}$ to those of any reductive
algebraic group $G$, as in \cite{GrojHaim07}.  The reader may observe
that apart from changes in notation, the arguments given here also
apply in the general case.
\end{remark}

\subsection{Tableaux for LLT series}
\label{ss:LLT-tableaux}

We now work out a tableau formalism which relates the polynomial part
$\Lcal ^{\sigma }_{\beta /\alpha }(x;q)_{\pol }$ to a combinatorial
LLT polynomial $\Gcal _{\nu }(x;q)$.

\begin{lemma}\label{lem:ek-on-E-lambda}
For all $\sigma \in S_{l}$, $\lambda \in \ZZ ^{l}$ and $k$,
the product of the elementary symmetric function $e_{k}(x)$ and the
non-symmetric Hall-Littlewood polynomial $E^{\sigma^{-1} }_{\lambda
}(x;q)$ is given by
\begin{equation}\label{e:ek-on-E-lambda}
e_{k}(x)\, E^{\sigma ^{-1}}_{\lambda }(x;q) = \sum _{|I| = k}
q^{-h_{I}} E^{\sigma ^{-1}}_{\lambda + \varepsilon _{I}}(x;q),
\end{equation}
where $I\subseteq \{1,\ldots,l \}$ has $k$ elements, $\varepsilon _{I}
= \sum _{i\in I}\varepsilon _{i}$ is the indicator vector of $I$, and
\begin{equation}\label{e:hI}
h_{I} = |\Inv (\lambda + \varepsilon _{I} + \epsilon \sigma )\setminus
\Inv (\lambda + \epsilon \sigma )|.
\end{equation}
Equivalently, $h_{I}$ is the number of pairs $i<j$ such that $i\in I$,
$j\not \in I$, and we have $\lambda _{j} = \lambda _{i}$ if $\sigma
(i)<\sigma (j)$, or $\lambda _{j} = \lambda _{i}+1$ if $\sigma
(i)>\sigma (j)$.
\end{lemma}

\begin{proof}
First consider the case $\sigma =1$.  Being symmetric, $e_{k}(x)$
commutes with $T_{w}$, giving
\begin{equation}\label{e:ek-E-lambda}
e_{k}\, E_{\lambda } = q^{-\ell(w)} T_{w}\, e_{k}\, x^{\lambda _{+}} =
q^{-\ell(w)}\sum _{|J| = k} T_{w}\, x^{\lambda
_{+}+\varepsilon _{J}},
\end{equation}
where $\lambda = w(\lambda _{+})$, as in \eqref{e:E-lambda}.  To fix
the choice, we take $w$ maximal in its coset $w\cdot \Stab (\lambda
_{+})$.

In each term of the sum in \eqref{e:ek-E-lambda}, the weight $\mu
=\lambda _{+}+\varepsilon _{J}$ can fail to be dominant at worst by
having some entries $\mu _{j} = \mu _{i}+1$ for indices $i<j$ such
that $(\lambda _{+})_{i} = (\lambda _{+})_{j}$, $i\not \in J$ and
$j\in J$.  Let $v$ be the minimal permutation such that $\mu _{+} =
v(\mu )$, that is, the permutation that moves indices $j\in J$ to the
left within each block of constant entries in $\lambda _{+}$.  The
formula $T_{i} x_{i}^{a}x_{i+1}^{a+1} = x_{i}^{a+1}x_{i+1}^{a}$ is
immediate from the definition of $T_{i}$, and implies that $T_{v}
x^{\mu } = x^{\mu _{+}}$.  By the maximality of $w$, since $v\in \Stab
(\lambda _{+})$, there is a reduced factorization $w = uv$, hence
$T_{w} = T_{u}T_{v}$.  Then
\begin{equation}\label{e:Tw-term}
T_{w}\, x^{\lambda _{+}+\varepsilon _{J}} = T_{u} x^{\mu _{+}} =
q^{\ell(u)} E_{\lambda +w(\varepsilon _{J})},
\end{equation}
since $\lambda +w(\varepsilon _{J}) = w(\mu ) = u(\mu _{+})$.

Now, $\ell(v)$ is equal to the number of pairs $i<j$ such that $\mu
_{i}<\mu _{j}$, that is, such that $(\lambda _{+})_{i}= (\lambda
_{+})_{j}$, $i\not \in J$ and $j\in J$.  By maximality, the
permutation $w$ carries these to the pairs $j' = w(i)$, $i' = w(j)$
such that $i' < j'$, $\lambda _{i'} = \lambda _{j'}$, $i'\in I$ and
$j'\not \in I$, where $I = w(J)$.  For $\sigma =1$, the definition of
$h_{I}$ is the number of such pairs $i',j'$, so we have $\ell(u)-\ell(w) =
-\ell(v) = -h_{I}$.  Hence, the term for $J$ in \eqref{e:ek-E-lambda} is
$q^{-\ell(w)} T_{w} x^{\lambda _{+}+\varepsilon _{J}} = q^{-h_{I}}
E_{\lambda +\varepsilon _{I}}$.  As $J$ ranges over subsets of size
$k$, so does $I = w(J)$, giving \eqref{e:ek-on-E-lambda} in this case.

Substituting $\sigma (\lambda )$ for $\lambda $ and $\sigma (I)$ for
$I$ in the $\sigma =1$ case, and acting on both sides with $T_{\sigma
}^{-1}$, yields
\begin{equation}\label{e:ek-on-E-lambda-twist}
q^{-|\Inv (\sigma )\cap \Inv (\lambda +\epsilon \rho )|} e_{k}\,
E^{\sigma ^{-1}}_{\lambda } = \sum _{|I|=k} q^{-|\Inv (\sigma (\lambda
+\varepsilon _{I}))\setminus \Inv (\sigma (\lambda ))|} q^{-|\Inv
(\sigma )\cap \Inv (\lambda +\varepsilon _{I} + \epsilon \rho )|}
E^{\sigma ^{-1}}_{\lambda +\varepsilon _{I}}.
\end{equation}

Combining powers of $q$ gives the desired identity
\eqref{e:ek-on-E-lambda} if we verify that
\begin{multline}\label{e:Inv-identity}
|\Inv (\sigma )\cap \Inv (\lambda +\varepsilon _{I} + \epsilon \rho )|
- |\Inv (\sigma )\cap \Inv (\lambda +\epsilon \rho )|\\
 = |\Inv (\lambda + \varepsilon _{I} + \epsilon \sigma )\setminus \Inv
(\lambda + \epsilon \sigma )| - |\Inv (\sigma (\lambda +\varepsilon
_{I}))\setminus \Inv (\sigma (\lambda ))|.
\end{multline}
On the left hand side, cancelling the contribution from the
intersection of the two sets leaves
\begin{equation}\label{e:Inv-LHS}
|\Inv (\sigma )\cap (\Inv (\lambda +\varepsilon _{I} + \epsilon \rho
)\setminus \Inv (\lambda + \epsilon \rho ))| - |\Inv (\sigma )\cap
(\Inv (\lambda +\epsilon \rho )\setminus \Inv (\lambda +\varepsilon
_{I} + \epsilon \rho ))|.
\end{equation}
The first term in \eqref{e:Inv-LHS} counts pairs $i<j$ such that $i\in
I$, $j\not \in I$, $\sigma (i)>\sigma (j)$, and $\lambda _{j} = \lambda
_{i}+1$.
The second term counts pairs $i>j$ such that $i\in I$, $j\not \in I$,
$\sigma (i)<\sigma (j)$, and $\lambda _{j} = \lambda _{i}$.
The first term on the right hand side of \eqref{e:Inv-identity} counts
pairs $i<j$ such that $i\in I$, $j\not \in I$, and $\lambda _{j} =
\lambda _{i}$ if $\sigma (i)<\sigma (j)$, or $\lambda _{j} = \lambda
_{i}+1$ if $\sigma (i)>\sigma (j)$.  The second term on the right hand
side of \eqref{e:Inv-identity} counts the set of pairs whose images
under $\sigma ^{-1}$ are pairs $i,j$ (in either order) such that $i\in
I$, $j\not \in I$, $\sigma (i)<\sigma (j)$ and $\lambda _{i} = \lambda
_{j}$.  The cases in the first term with $\sigma (i)<\sigma (j)$
cancel those in the second term with $i<j$.  The remaining cases in
the first term, with $\sigma (i)>\sigma (j)$, match the first term in
\eqref{e:Inv-LHS}, while the remaining cases in the second term, with
$i>j$, match the second term in \eqref{e:Inv-LHS}.  This proves
\eqref{e:Inv-identity} and completes the proof of the lemma.
\end{proof}

\begin{figure}
\[
\begin{array}{cc}
\begin{tikzpicture}[scale=.425]
\draw [yshift=3cm] (1,0) grid (4,1);
\node [left] at (1,3.5) {$-\infty \!$};
\node [right] at (4,3.5) {$\! \infty $};
\foreach \x/\nn in {1.5/4, 2.5/8, 3.5/9}{\node at (\x,3.5) {$\nn$};}

\draw [yshift=1.5cm] (0,0) grid (2,1);
\node [left] at (0,2) {$-\infty \!$};
\node [right] at (2,2) {$\! \infty $};
\foreach \x/\nn in {.5/1, 1.5/3}{\node at (\x,2) {$\nn$};}

\draw (1,0) grid (5,1);
\node [left] at (1,.5) {$-\infty \! $};
\node [right] at (5,.5) {$\! \infty $};
\foreach \x/\nn in {1.5/2, 2.5/5, 3.5/6, 4.5/7}{\node at (\x,.5) {$\nn$};}
\end{tikzpicture}
&
\begin{tikzpicture}[scale=.425]
\draw [shift={(0,3)},thick] (0,0) grid (1,1);
\node [left] at (0,3.5) {$-\infty \!$};
\node at (.5,3.5) {$\mathbf 4$};
\draw [shift={(0,1.5)},thick] (0,0) grid (1,1);
\node at (.5,2) {$\mathbf 3$};

\draw [shift={(4,3)},thick] (0,0) grid (1,1);
\node [left] at (4,3.5) {$-\infty \!$};
\node at (4.5,3.5) {$\mathbf 4$};
\draw [shift={(4,0)},thick] (0,0) grid (1,1);
\node at (4.5,.5) {$\mathbf 2$};

\draw [shift={(8,3)},thick] (-1,0) grid (1,1);
\node at (7.5,3.5) {$\mathbf 4$};
\node at (8.5,3.5) {$\mathbf 8$};
\draw [shift={(8,0)},thick] (0,0) grid (1,1);
\node at (8.5,.5) {$\mathbf 5$};

\draw [shift={(12,3)}] (-1,0) grid (1,1);
\node at (11.5,3.5) {$8$};
\node at (12.5,3.5) {$9$};
\draw [shift={(12,0)}] (0,0) grid (1,1);
\node at (12.5,.5) {$6$};

\draw [shift={(16,3)}] (-1,0) grid (0,1);
\node at (15.5,3.5) {$9$};
\node [right] at (16,3.5) {$\! \infty $};
\draw [shift={(16,0)}] (0,0) grid (1,1);
\node at (16.5,.5) {$7$};

\draw [shift={(20,1.5)},thick] (-1,0) grid (1,1);
\node at (19.5,2) {$\mathbf 1$};
\node at (20.5,2) {$\mathbf 3$};
\draw [shift={(20,0)},thick] (0,0) grid (1,1);
\node at (20.5,.5) {$\mathbf 2$};

\draw [shift={(24,1.5)},thick] (-1,0) grid (0,1);
\node at (23.5,2) {$\mathbf 3$};
\node [right] at (24,2) {$\! \infty $};
\draw [shift={(24,0)},thick] (0,0) grid (1,1);
\node at (24.5,.5) {$\mathbf 5$};

\end{tikzpicture} \\[1ex]
S & \text{$\sigma $-triples for $\sigma =1$}\\[3ex]
&
\begin{tikzpicture}[scale=.425]
\draw [shift={(0,3)},thick] (0,0) grid (1,1);
\node [left] at (0,3.5) {$-\infty \!$};
\node at (.5,3.5) {$\mathbf 4$};
\draw [shift={(0,1.5)},thick] (0,0) grid (1,1);
\node at (.5,2) {$\mathbf 3$};

\draw [shift={(4,3)}] (-1,0) grid (1,1);
\node at (3.5,3.5) {$4$};
\node at (4.5,3.5) {$8$};
\draw [shift={(4,0)}] (-1,0) grid (0,1);
\node at (3.5,.5) {$2$};

\draw [shift={(8,3)}] (-1,0) grid (1,1);
\node at (7.5,3.5) {$8$};
\node at (8.5,3.5) {$9$};
\draw [shift={(8,0)}] (-1,0) grid (0,1);
\node at (7.5,.5) {$5$};

\draw [shift={(12,3)}] (-1,0) grid (0,1);
\node at (11.5,3.5) {$9$};
\node [right] at (12,3.5) {$\! \infty $};
\draw [shift={(12,0)}] (-1,0) grid (0,1);
\node at (11.5,.5) {$6$};

\draw [shift={(16,1.5)}] (-1,0) grid (0,1);
\node at (15.5,2) {$3$};
\node [right] at (16,2) {$\! \infty $};
\draw [shift={(16,0)}] (-1,0) grid (0,1);
\node at (15.5,.5) {$2$};
\end{tikzpicture} \\
& \text{$\sigma $-triples for $\sigma = (1,2,3)\mapsto (3,1,2)$}
\end{array}
\]
\caption{\label{fig:sigma triples}
A negative tableau $S$ on $\beta/\alpha = (6,3,5)/(2,1,2)$ and the
$\sigma$-triples in $\beta/\alpha$, for two choices of $\sigma $,
shown with their entries from $S$.  Triples in boldface are increasing
in $S$.}
\end{figure}

Given $\alpha ,\beta \in \ZZ ^{l}$ such that $\alpha _{j}\leq \beta
_{j}$ for all $j$, we let $\beta /\alpha $ denote the tuple of
one-row skew
%
%
shapes $(\beta _{j})/(\alpha _{j})$ such that the $x$ coordinates of
the right edges of boxes $a$ in the $j$-th row are the integers
$\alpha _{j}+1,\ldots,\beta _{j}$.  The boxes just outside the $j$-th
row, adjacent to the left and right ends of the row, then have $x$
coordinates $\alpha _{j}$ and $\beta _{j}+1$.  In the case of an empty
row with $\alpha _{j} = \beta _{j}$, we still consider these two boxes
to be adjacent to the ends of the row.

For each box $a$, let $i(a)$ denote the $x$ coordinate of its right
edge and $j(a)$ the index of the row containing it.

For $\sigma \in S_{l}$, we define $\sigma (\beta /\alpha ) = \sigma
(\beta )/\sigma (\alpha )$.  If $a$ is a box in row $j(a)$ of $\beta
/\alpha $, then $\sigma (a)$ denotes the corresponding box with
coordinate $i(\sigma (a)) = i(a)$ in row $\sigma (j(a))$ of $\sigma
(\beta /\alpha )$.  The adjusted content of $\sigma (a)$, as defined
in \S \ref{ss:G-nu}, is then $\ctild (\sigma (a)) = i(a)+\epsilon
\sigma (j(a))$. Hence, the reading order on $\sigma (\beta /\alpha )$
corresponds via $\sigma $ to the ordering of boxes in $\beta /\alpha $
by increasing values of $i(a)+\epsilon \sigma (j(a))$.

We define a {\em $\sigma $-triple} in $\beta /\alpha $ to consist of
any three boxes $(a,b,c)$ arranged as follows: boxes $a$ and $c$ are
in or adjacent to the same row $j(a) = j(c)$, and are consecutive,
that is, $i(c) = i(a)+1$, while box $b$ is in a row $j(b)<j(a)$, and
we have $a<b<c$ in the ordering corresponding to the reading order on
$\sigma (\beta /\alpha )$.  More explicitly, this means that if
$\sigma (j(b))<\sigma (j(a))$, then $i(b) = i(c)$, while if $\sigma
(j(b))>\sigma (j(a))$, then $i(b) = i(a)$.  The box $b$ is required to
be a box of $\beta /\alpha $, but box $a$ is allowed to be outside and
adjacent to the left end of a row, while $c$ is similarly allowed to
be outside and adjacent to the right end of a row.

An example of a tuple $\beta /\alpha $, with all its $\sigma $-triples
for two different choices of $\sigma $, is shown in Figure
\ref{fig:sigma triples}.

A negative tableau on $\beta /\alpha $ is a map $S\colon \beta /\alpha
\rightarrow \ZZ _{+}$ strictly increasing on each row.  In the
terminology of \S \ref{ss:G-nu}, $S$ is a super tableau on $\beta
/\alpha $ with entries in $\ZZ _{+}$, considered as a negative
alphabet ordered by $\overline{1}<\overline{2}<\cdots $.  We say that
a $\sigma $-triple $(a,b,c)$ in $\beta /\alpha $ is {\em increasing}
in $S$ if $S(a)<S(b)<S(c)$, with the convention that $S(a) = -\infty $
if $a$ is just outside the left end of a row, and $S(c) = \infty $ if
$c$ is just outside right end of a row.  Along with the $\sigma
$-triples in $\beta /\alpha $, Figure \ref{fig:sigma triples} also
displays which triples are increasing in a sample tableau $S$.

\begin{prop}\label{prop:e-mu-coef}
Given $\alpha ,\beta \in \ZZ ^{l}$ such that $\alpha _{i}\leq \beta
_{i}$ for all $i$, and $\sigma \in S_{l}$, let
\begin{equation}\label{e:triples-N}
N^{\sigma }_{\beta /\alpha }(X;q) = 
\sum_{S\in\SSYT_{-}(\beta/\alpha)} q^{h_\sigma(S)} x^S
\end{equation}
be the generating function for negative tableaux $S$ on the tuple of
one-row skew diagrams $(\beta _{i})/(\alpha _{i})$, weighted by
$q^{h_{\sigma }(S)}$, where $h_{\sigma }(S)$ is the number of
increasing $\sigma $-triples in $S$.  Then $N^{\sigma }_{\beta /\alpha
}(X;q)$ is a symmetric function, and $\omega N^{\sigma }_{\beta
/\alpha }(X;q)$ evaluates in $l$ variables to
\begin{equation}\label{e:N-vs-Lcal}
(\omega N^{\sigma }_{\beta /\alpha })(x_{1},\ldots,x_{l};q) = \Lcal ^{\sigma
}_{\beta /\alpha }(x_{1},\ldots,x_{l};q)_{\pol }.
\end{equation}
If we do not have $\alpha _{i}\leq \beta _{i}$ for all $i$, then
$\Lcal ^{\sigma }_{\beta /\alpha }(x;q)_{\pol } = 0$.
\end{prop}

\begin{proof}
Let $L^{\sigma}_{\beta/\alpha}(X;q)$
be the unique symmetric function such that (i)
$L^{\sigma }_{\beta /\alpha }(X;q)$ is a linear combination of Schur
functions $s_{\lambda }$ with $\ell(\lambda )\leq l$, and (ii) in $l$
variables, it evaluates to
\begin{equation}\label{e:L-vs-Lcal}
L^{\sigma }_{\beta /\alpha }(x_{1},\ldots,x_{l};q) = \Lcal ^{\sigma
}_{\beta /\alpha }(x_{1},\ldots,x_{l};q)_{\pol }.
\end{equation}
What we need to prove is that $\omega\, L^{\sigma }_{\beta /\alpha
}(X;q) = N^{\sigma }_{\beta /\alpha }(X;q)$.

The definition of $\Lcal ^{\sigma }_{\beta /\alpha }(x;q)$ implies
that $L^{\sigma }_{\beta /\alpha }(X;q)$ satisfies
\begin{equation}\label{e:L(X)}
\langle s_{\lambda }(X), L^{\sigma }_{\beta /\alpha }(X;q) \rangle =
\langle E^{\sigma ^{-1}}_{\beta }(x;q^{-1}) \rangle\, s_{\lambda
}(x_{1},\ldots,x_{l})\, E^{\sigma ^{-1}}_{\alpha }(x;q^{-1})
\end{equation}
for every partition $\lambda $, including when $\ell(\lambda ) > l$,
since then both sides are zero.  By linearity, we can replace
$s_\lambda$ by any symmetric function $f$, giving
\begin{equation}\label{e:G-anyf}
\langle f(X),\, L^\sigma_{\beta/\alpha}(X;q)
\rangle = \langle E^{\sigma ^{-1}}_{\beta
}(x;q^{-1}) \rangle\, f(x)\, E^{\sigma ^{-1}}_{\alpha
}(x;q^{-1})\,.
\end{equation}
The coefficient of $m_\mu(X)$ in $\omega\,
L^\sigma_{\beta/\alpha}(X;q)$ is given by taking $f=e_\mu$.

To show that $\omega\, L^{\sigma }_{\beta /\alpha }(X;q)$ is given by
the tableau generating function in \eqref{e:triples-N}, we use Lemma
\ref{lem:ek-on-E-lambda} to express
\begin{equation}\label{e:e-mu-coef}
\langle E^{\sigma ^{-1}}_{\beta }(x;q^{-1}) \rangle\,  e_{\mu  }(x)\,
E^{\sigma ^{-1}}_{\alpha }(x;q^{-1})
\end{equation}
as a sum of powers of $q$ indexed by negative tableaux.  In
particular, this coefficient will vanish unless we have $\alpha
_{i}\leq \beta _{i}$ for all $i$, giving the last conclusion in the
proposition.

Multiplying by $e_{\mu _{1}}$ through $e_{\mu _{n}}$ successively and
keeping track of one chosen term in each product gives a sequence of
terms $E^{\sigma ^{-1}}_{\alpha ^{(0)}}, E^{\sigma ^{-1}}_{\alpha
^{(1)}},\ldots, E^{\sigma ^{-1}}_{\alpha ^{(n)}}$, in which $\alpha
^{(0)} = \alpha $ and $\alpha ^{(m+1)} = \alpha ^{(m)} + \varepsilon
_{I}$ for a set of indices $I$ of size $\mu _{m}$, for each $m$.  Each
sequence with $\alpha ^{(n)} = \beta $ contributes to \eqref{e:e-mu-coef}.

If we record these data in the form of a tableau $S\colon \beta
/\alpha \rightarrow \ZZ _{+}$ with $S(a) = m$ for $a\in (\alpha
^{(m)}/\alpha ^{(m-1)})$, $S$ satisfies the condition that it is a
negative tableau of weight $x^{S} = x^{\mu }$.  The contribution to
\eqref{e:e-mu-coef} from the corresponding sequence of terms is the
product of the $q^{h_{I}}$ with $h_{I}$ as in \eqref{e:hI} for $k= \mu
_{m}$, $\lambda =\alpha ^{(m-1)}$, and $I$ the set of indices $j$ such
that $S(a) = m$ for some box $a$ in row $j$.

We now express the $h_{I}$ corresponding to $(\alpha ^{(m)}/\alpha
^{(m-1)}) = S^{-1}(\{m \})$ as an attribute of $S$.  For $h_{I}$ to
count a pair $i<j$, we must have $i\in I$, which means that $S(b) = m$
for a box $b$ in row $i$, and $j\not \in I$, and one of the following
two situations.

If $\sigma (i)<\sigma (j)$, we must have $\alpha ^{(m-1)}_{j} = \alpha
^{(m-1)}_{i}$.  Since there is no $m$ in row $j$ of $S$, this means
that the boxes $a$ and $c$ in row $j$ with coordinates $i(a) =
i(b)-1$, $i(c) = i(b)$ have $S(a)<S(b)<S(c)$, with the same convention
as above that $S(a) = -\infty $ if $a$ is to the left of a row of
$\beta /\alpha $, and $S(c) = \infty $ if $c$ is to the right of a
row.

If $\sigma (i)>\sigma (j)$, we must have $\alpha ^{(m-1)}_{j} = \alpha
^{(m-1)}_{i}+1$.  This means that the boxes $a$ and $c$ in row $j$
with coordinates $i(a) = i(b)$, $i(c) = i(b)+1$ have $S(a)<S(b)<S(c)$,
with the same convention as before.

These two cases establish that $h_{I}$ is equal to the number of
increasing $\sigma $-triples in $S$ for which $S(b) = m$.  Summing
them up gives the total number of increasing $\sigma $-triples,
implying that
the coefficient in \eqref{e:e-mu-coef} is
the sum of $q^{h_{\sigma }(S)}$ over negative tableaux $S$ of weight
$x^{S} = x^{\mu }$ on $\beta /\alpha $, where $h_{\sigma }(S)$ is the
number of increasing $\sigma $-triples in $S$.
\end{proof}

\begin{lemma}\label{l:triples2dinv}
Given $\sigma\in S_l$ and $\alpha ,\beta \in \ZZ ^{l}$ with $\alpha
_{i}\leq \beta _{i}$ for all $i$, for $T\in
\SSYT_{-}(\sigma(\beta/\alpha))$,
\begin{equation}
h_{\sigma }(\beta /\alpha )-\inv (T)= h_{\sigma }(S)\,,
\end{equation}
where $S = \sigma ^{-1}(T) = T\circ \sigma $.
\end{lemma}

\begin{proof}
Recall from~\S\ref{ss:G-nu} that an attacking inversion in a negative
tableau is defined by $T(a)\geq T(b)$, where $a, b$ is an attacking
pair with $a$ preceding $b$ in the reading order.

One can verify from the definition of $\sigma $-triple that the
attacking pairs in $\sigma (\beta /\alpha )$, ordered by the reading
order, are precisely the pairs $\sigma (a,b)$ or $\sigma (b,c)$ for
$(a,b,c)$ a $\sigma $-triple such that the relevant boxes are in
$\beta /\alpha $.  Moreover, every attacking pair occurs in this manner
exactly once.

If all three boxes of a $\sigma $-triple $(a,b,c)$ are in $\beta
/\alpha $, and $T$ is a negative tableau on $\sigma (\beta /\alpha )$,
then since $T(\sigma (a))<T(\sigma (c))$, at most one of the attacking
pairs $\sigma (a,b)$, $\sigma (b,c)$ can be an attacking inversion in
$T$.  The condition that neither pair is an attacking inversion is
that $S(a)<S(b)<S(c)$ in the negative tableau $S = \sigma ^{-1}(T) =
T\circ \sigma $ on $\beta /\alpha $.  This also holds for triples not
contained in $\beta /\alpha $ with our convention that $S(a) = -\infty
$ or $S(c) = \infty $ for $a$ or $c$ outside the tuple $\beta /\alpha
$.  Hence, the result follows.
\end{proof}
\begin{example}
Let  $S$ be as in Figure \ref{fig:sigma triples} and 
\[
\begin{tikzpicture}[scale =.4][yshift = 200]
\node[anchor=east] at (-.4,2) {$\sigma = (1,2,3)\mapsto (3,1,2)$, \ \ \qquad $ T = S \circ \sigma^{-1} = $};
\node at (6.3,2) { $,$};

\draw [shift={(0,1.5*2)}](6,1) grid (2,0);
\foreach \x/\nn in {2.5/2, 3.5/5, 4.5/6, 5.5/7}{\node at (\x,1.5*2+.5) { $\ifnum\nn=10 \infty \else \ifnum\nn=-10 -\infty \else \pgfmathprintnumber{\nn} \fi \fi $};}

\draw [shift={(0,1.5*1)}] (5,1) grid (2,0);
\foreach \x/\nn in {2.5/4, 3.5/8, 4.5/9}{\node at (\x,1.5*1+.5) { $\ifnum\nn=10 \infty \else \ifnum\nn=-10 -\infty \else \pgfmathprintnumber{\nn} \fi \fi $};}

\draw [shift={(0,1.5*0)}](3,1) grid (1,0);
\foreach \x/\nn in {1.5/1, 2.5/3}{\node at (\x,1.5*0+.5) { $\ifnum\nn=10 \infty \else \ifnum\nn=-10 -\infty \else \pgfmathprintnumber{\nn} \fi \fi $};}
\end{tikzpicture}
\]
so $T$ is a negative tableau on $\sigma(\beta/\alpha) =
(3,5,6)/(1,2,2)$.  Reading $T$ by reading order on
$\sigma(\beta/\alpha)$ gives $134285967$.  The pairs $(T(a),T(b))$ for
attacking inversions $(a,b)$ in $T$ are $(3,2)$, $(4,2)$, $(8,5)$, and
$(9,6)$, so $\inv(T) = 4$.  From the bottom row of Figure
\ref{fig:sigma triples}, we have $h_\sigma(\beta/\alpha) = 5$,
$h_\sigma(S)= 1$, so $h_\sigma(\beta/\alpha) - \inv(T) = h_\sigma(S)$
is indeed satisfied.

Now let $\sigma = 1$ and $T = S \circ \sigma^{-1} = S$.  Reading $T$
by reading order on $\beta/\alpha$ gives $123458697$.  The pairs
$(T(a),T(b))$ for the attacking inversions are $(8,6)$ and $(9,7)$, so
$\inv(T)= 2$.  From Figure \ref{fig:sigma triples}, we see
$h_\sigma(\beta/\alpha) = 7$, $h_\sigma(S)= 5$, so again
$h_\sigma(\beta/\alpha) - \inv(T) = h_\sigma(S)$ holds.
\end{example}
\begin{remark}
 If we define an increasing \(\sigma\)-triple for \(T \in
 \SSYT(\sigma(\beta/\alpha))\) to be any \(\sigma\)-triple \((a,b,c)\)
 satisfying
 \(T(a) \leq T(b) \leq T(c)\), then a similar argument gives the relation
\begin{equation}
h_{\sigma }(\beta /\alpha )-\inv (T)= h_{\sigma }(S)\,
\end{equation}
for ordinary semistandard  tableaux as well,
where again $S = \sigma ^{-1}(T) = T\circ \sigma $.
\end{remark}
\begin{remark}\label{rem:triples}
Given a tuple of rows $\beta _{i}/\alpha _{i}$, if we shift the $i$-th
row to the right by $\epsilon \sigma (i)$ for a small $\epsilon >0$,
then $h_{\sigma }(\beta /\alpha )$ is equal to the number of
alignments between a box boundary in row $j$ and the interior of a box
in row $i$ for $i<j$, where a boundary is the edge common to two
adjacent boxes which are either in or adjacent to the row.  An empty
row has one boundary.
\end{remark}

The relation between $\Gcal$ and $\Lcal$ can now be made precise by
applying Proposition \ref{prop:e-mu-coef} and
Lemma~\ref{l:triples2dinv} to the expression for $q^{h_{\sigma }(\beta
/\alpha )} \Gcal _{\sigma (\beta /\alpha )}(X;q^{-1})$ given in
Corollary \ref{cor:omega-super}.

\begin{cor}\label{cor:G-versus-L}
Given $\alpha ,\beta \in \ZZ ^{l}$ and $\sigma \in S_{l}$,
\begin{equation}\label{e:G-versus-L}
\Lcal ^{\sigma }_{\beta /\alpha }(x;q)_{\pol } =
\begin{cases}
q^{h_{\sigma }(\beta/\alpha )} \Gcal _{\sigma (\beta /\alpha )}(x;q^{-1}) &
\text{if $\alpha _{i}\leq \beta _{i}$ for all $i$}\\
0 & \text{otherwise},
\end{cases}
\end{equation}
where $h_{\sigma }(\beta /\alpha )$ is the number of $\sigma $-triples
in $\beta /\alpha $, and the right hand side is evaluated in $l$
variables $x_{1},\ldots, x_{l}$.
\end{cor}

\section{The generalized shuffle theorem}
\label{s:main}

\subsection{Cauchy identity}
\label{ss:Cauchy}
In this section we derive our main results, Theorems \ref{thm:main-L}
and \ref{thm:main-G}.  The key point is the following delightful
`Cauchy identity' for non-symmetric Hall-Littlewood polynomials.

\begin{thm}\label{thm:Cauchy}
For any permutation $\sigma \in S_{l}$, the twisted non-symmetric
Hall-Littlewood polynomials $E^{\sigma }_{\lambda }(x;q)$ and
$F^{\sigma }_{\lambda }(x;q)$ in (\ref{e:E-twist}, \ref{e:F-twist})
satisfy the identity
\begin{equation}\label{e:Cauchy}
\frac{\prod _{i<j} (1 - q\, t\, x_{i} \, y_{j})}{\prod _{i\leq j} (1 -
t\, x_{i}\, y_{j})} = \sum _{\aA \geq 0} t^{|\aA |}\, E^{\sigma }_{\aA
}(x_{1},\ldots,x_{l};q^{-1}) \, F^{\sigma }_{\aA
}(y_{1},\ldots,y_{l};q),
\end{equation}
where the sum is over $\aA = (a_{1},\ldots,a_{l})$ with all $a_{i}\geq
0$ and $|\aA | \defeq a_{1}+\cdots +a_{l}$.
\end{thm}

\begin{proof}
Let $Z(x,y,q,t)$ denote the product on the left hand side.

From the definitions we see that $E^{\sigma }_{\aA }(x;q)$ and
$F^{\sigma }_{\aA }(x;q)$ for $\aA \geq 0$ 
belong to the polynomial
ring $\QQ (q)[x] = \QQ (q)[x_{1},\ldots,x_{l}]$.  The $E^{\sigma
}_{\aA }(x;q)$ form a graded basis of $\QQ (q)[x]$, since they are
homogeneous and $E^{\sigma }_{\aA }(x;q)$ has leading term $x^{\aA }$.
The $F^{\sigma }_{\aA }(x;q)$ likewise form a graded basis of $\QQ
(q)[x]$.  We are to prove that the expansion of $Z(x,y,q,t)$ as a
power series in $t$, with coefficients expressed in terms of the basis
$\{E^{\sigma }_{\aA }(x;q^{-1}) F^{\sigma }_{\bb }(y;q) \}$ of $\QQ
(q)[x,y]$, is given by the formula on the right hand side.  Put
another way, we are to show that the coefficient of $F^{\sigma }_{\aA
}(y;q)$ in $Z(x,y,q,t)$ is equal to $t^{|\aA |}E^{\sigma }_{\aA
}(x;q^{-1})$, or equivalently that the coefficient of $F^{\sigma
}_{\aA }(y^{-1};q^{-1})$ in $Z(x,y^{-1},q^{-1},t)$ is equal to
$t^{|\aA |}E^{\sigma }_{\aA }(x;q)$.

Using Proposition \ref{prop:twist-orthogonality}, this will follow by taking
$f(y) = E^{\sigma }_{\aA }(y;q)$ in the identity
\begin{equation}\label{e:reproducing}
f(t x) = \langle y^{0} \rangle\; f(y) \frac{\prod _{i<j}(1 - q^{-1}
t\, x_{i}/y_{j})}{\prod _{i\leq j} (1 - t\, x_{i}/y_{j} )} \prod
_{i<j} \frac{1 - y_{i}/y_{j}}{1 - q^{-1} y_{i}/y_{j}},
\end{equation}
provided we prove that this identity is valid for all polynomials
$f(y) = f(y_{1},\ldots,y_{l})$.  Here we mean that $f(y)\in \QQ
(q)[y]$ is a true polynomial and not a Laurent polynomial.  Note that
the denominator factors in \eqref{e:reproducing} should be understood
as geometric series.

The only factor in \eqref{e:reproducing} that involves negative powers
of $y_{1}$ is $1/(1-t\, x_{1}/y_{1})$.  All the rest is a power series
as a function of $y_{1}$.  For any power series $g(y_{1})$, we have
$\langle y_{1}^{0} \rangle\, g(y_{1})/(1-t\, x_{1}/y_{1}) = g(t
x_{1})$.  The factors other than $1/(1- t\, x_{1}/y_{1})$ with index
$i=1$ in \eqref{e:reproducing} cancel upon setting $y_{1} = tx_{1}$.
It follows that when we take the constant term in the variable
$y_{1}$, \eqref{e:reproducing} reduces to the same identity in
variables $y_{2},\ldots,y_{l}$.  We can assume that the latter holds
by induction.
\end{proof}
\begin{example}
The products $t^{|\aA|}E^{\sigma}_{\aA}(x;q^{-1}) F^{\sigma }_{\aA}(y;q)$ 
for the pairs in Figure \ref{fig:ns Hall-Littlewood}
sum to the $t^{0}$ through $t^{2}$ terms in the expansion of
\begin{equation}
\frac{(1 - q\, t\, x_{1} \, y_{2})(1 - q\, t\, x_{1} \,
y_{3})(1 - q\, t\, x_{2} \, y_{3})} {(1 - t\, x_{1}\, y_{1})(1 - t\,
x_{1}\, y_{2})(1 - t\, x_{1}\, y_{3})(1 - t\, x_{2}\, y_{2})(1 - t\,
x_{2}\, y_{3})(1 - t\, x_{3}\, y_{3})}.
\end{equation}
\end{example}

\begin{remark}\label{rem:Cauchy-at-q=0}
Using the fact that our non-symmetric Hall-Littlewood polynomials
agree with those of Ion in~\cite{Ion08}, the $\sigma=1$ case
of~\eqref{e:Cauchy} can be derived from the Cauchy identity for
non-symmetric Macdonald polynomials of Mimachi and
Noumi~\cite{MiNou98}.
We also note that~\eqref{e:Cauchy} for $\sigma =1$
specializes at $q = 0$ to the $\GL _{l}$ case of the non-symmetric
Cauchy identities of Fu and Lascoux~\cite{FuLas90}.
\end{remark}

\subsection{Winding permutations}
\label{ss:winding}

We will apply Theorem \ref{thm:Cauchy} in cases for which the
twisting permutation has a special form, allowing the Hall-Littlewood
polynomial $F^{\sigma }_{\aA }(y;q)$ to be written another way.

\begin{defn}\label{def:winding}
Let $\fp{x} = x - \lfloor x \rfloor$ 
denote the fractional part of a real
number $x$.  Let $c_{1},\ldots,c_{l}$ be the sequence of fractional
parts $c_{i} = \fp{a\, i+b}$ of an arithmetic progression, where $a$ is
assumed irrational, so the $c_{i}$ are distinct.  Let $\sigma \in
S_{l}$ be the permutation such that $\sigma (c_{1},\ldots,c_{l})$ is
increasing, i.e., $\sigma (1),\ldots,\sigma (l)$ are in the same
relative order as $c_{1},\ldots,c_{l}$.

A permutation $\sigma $ of this form is a {\em winding permutation}.
The {\em descent indicator} of $\sigma $ is the vector $(\eta
_{1},\ldots,\eta _{l-1})$ defined by
\begin{equation}\label{e:descent-indicator}
\eta _{i} = \begin{cases}
1& \text{if $\sigma (i)>\sigma (i+1)$},\\
0& \text{if $\sigma (i)<\sigma (i+1)$}.
\end{cases}
\end{equation}
The {\em head} and {\em tail} of the winding permutation $\sigma $ are
the respective permutations $\tau ,\, \theta \in S_{l-1}$ such that
$\tau (1),\ldots,\tau (l-1)$ are in the same relative order as $\sigma
(1),\ldots,\sigma (l-1)$, and $\theta (1),\ldots,\theta (l-1)$ are in
the same relative order as $\sigma (2),\ldots,\sigma (l)$.
\end{defn}

Adding an integer to $a$ in the above definition doesn't change the
$c_{i}$, so we can assume that $0<a<1$.  In that case the descent
indicator of $\sigma $ is characterized by
\begin{equation}\label{e:winding-descents}
\begin{aligned}
\eta _{i} = 1\quad \Leftrightarrow \quad c_{i}>c_{i+1}\quad
\Leftrightarrow \quad & c_{i+1} = c_{i}+a-1,\\
\eta _{i} = 0\quad \Leftrightarrow \quad c_{i}<c_{i+1}\quad
\Leftrightarrow\quad & c_{i+1} = c_{i}+a.\\
\end{aligned}
\end{equation}

\begin{prop}\label{prop:winding}
Let $\sigma \in S_{l}$ be a winding permutation, with descent
indicator $\eta $, and head and tail $\tau ,\, \theta \in S_{l-1}$.
For every $\lambda \in \ZZ ^{l-1}$ we have identities 
\begin{gather}\label{e:winding-1}
E^{\theta ^{-1}}_{\lambda }(x;q) = x^{\eta }\, E^{\tau ^{-1}}_{\lambda
-\eta }(x;q),\\
\label{e:winding-2}
F^{\theta ^{-1}}_{\lambda }(x;q) = x^{\eta }\, F^{\tau ^{-1}}_{\lambda
-\eta }(x;q)
\end{gather}
of Laurent polynomials in $x_{1},\ldots,x_{l-1}$.
\end{prop}

The proof uses the following lemma.

\begin{lemma}\label{lem:winding}
With $\tau ,\theta $ and $\eta $ as in Proposition \ref{prop:winding},
and for every $w\in S_{l-1}$, there is an identity of operators on
$\kk[x_{1}^{\pm 1},\ldots,x_{l-1}^{\pm 1}]$
\begin{equation}\label{e:Hecke-winding}
T_{\tau w}^{-1}\, T_{\tau }\, x^{-\eta }\, T_{\theta }^{-1}\,
T_{\theta w} = q^{e}\, x^{-w^{-1}(\eta )},
\end{equation}
for some exponent $e$ depending on $w$.
\end{lemma}

\begin{proof}
We prove \eqref{e:Hecke-winding} by induction on the length of $w$.
The base case $w=1$ is trivial.  Suppose now that $w = v s_{i}$ is a
reduced factorization.  We can write the left hand side of
\eqref{e:Hecke-winding} as
\begin{equation}\label{e:winding-induction}
T_{i}^{\varepsilon _{1}}\, T_{\tau v}^{-1}\, T_{\tau }\, x^{-\eta }\,
T_{\theta }^{-1}\, T_{\theta v}\, T_{i}^{\varepsilon _{2}},
\end{equation}
where
\begin{equation}\label{e:epsilons}
\varepsilon _{1} = \begin{cases}
+1 &	 \text{if $\tau vs_{i} < \tau  v$}\\
-1&	 \text{if $\tau vs_{i} > \tau  v$}
\end{cases}\qquad
\varepsilon _{2} = \begin{cases}
+1 &	 \text{if $\theta vs_{i} > \theta v$}\\
-1&	 \text{if $\theta vs_{i} < \theta v$}
\end{cases}.
\end{equation}
Assuming by induction that \eqref{e:Hecke-winding} holds for $v$, and
substituting this into \eqref{e:winding-induction}, we are left to
show that
\begin{equation}\label{e:TxT-special}
T_{i}^{\varepsilon _{1}}\, x^{-v^{-1}(\eta )}\, T_{i}^{\varepsilon _{2}} =
q^{e}\, x^{-s_{i} v^{-1}(\eta )}
\end{equation}
for some exponent $e$.  We now consider the possible values for
$\langle \alpha _{i}^{\vee }, -v^{-1}(\eta ) \rangle$, which is equal
to $\eta _{k}-\eta _{j}$, where $v(i) = j$, $v(i+1) = k$.

Case 1: If $\eta _{j} = \eta _{k}$, we see from \eqref{e:winding-descents}
that $c_{j+1}-c_{j} = c_{k+1}-c_{k}$, hence $c_{j+1} <
c_{k+1}\Leftrightarrow c_{j} < c_{k}$.  This implies that $\sigma
(j+1)<\sigma(k+1) \Leftrightarrow \sigma (j)<\sigma (k)$, and therefore that
$\tau v(i) < \tau v(i+1) \Leftrightarrow \theta v(i) < \theta v(i+1)$.
Hence, in this case we have $\varepsilon_{1} = -\varepsilon _{2}$.

Case 2: If $\eta _{j} = 1$ and $\eta _{k} = 0$, then from
\eqref{e:winding-descents} we get $c_{j+1} = c_{j}+a-1$ and $c_{k+1} =
c_{k}+a$.  Then $c_{k+1}-c_{j+1} = c_{k}-c_{j}+1$.  Since
$c_{k+1}-c_{j+1}$ and $c_{k}-c_{j}$ both have absolute value less than
$1$, this implies $c_{k} < c_{j}$ and $c_{k+1} > c_{j+1}$.  It follows
in the same way as in Case 1 that $\tau v(i) > \tau v(i+1)$ and
$\theta v(i) < \theta v(i+1)$.  Hence, in this case we have $\varepsilon
_{1} = \varepsilon _{2} = 1$.

Case 3: If $\eta _{j} = 0$ and $\eta _{k} = 1$ we reason as in Case 2,
but with $j$ and $k$ exchanged, to conclude that in this case we have
$\varepsilon _{1} = \varepsilon _{2} = -1$.

In each case, \eqref{e:TxT-special} now follows from the well-known
affine Hecke algebra identities
\begin{equation}\label{e:TxT-general}
\begin{aligned}
T_{i}^{-1}\, x^{\mu }\, T_{i}^{} = T_{i}\, x^{\mu }\, T_{i}^{-1} =
x^{\mu } = x^{s_{i}\mu }
& \qquad 	\text{if $ \langle \alpha _{i}^{\vee }, \mu  \rangle = 0$},\\
T_{i}\, x^{\mu }\, T_{i} = q\, x^{s_{i}\mu }
& \qquad 	\text{if $ \langle \alpha _{i}^{\vee }, \mu  \rangle = -1$},\\
T_{i}^{-1}\, x^{\mu }\, T_{i}^{-1} = q^{-1} x^{s_{i}\mu }
& \qquad \text{if $ \langle \alpha _{i}^{\vee }, \mu \rangle = 1$},
\end{aligned}
\end{equation}
which can be verified directly from the definition of $T_{i}$.
\end{proof}

\begin{proof}[Proof of Proposition \ref{prop:winding}] Let $w_{0}\in
S_{l}$ and $w_{0}'\in S_{l-1}$ be the longest permutations.  Then
$w_{0}'\tau $, $w_{0}'\theta $ are the head and tail of the winding
permutation $w_{0}\sigma $, and the descent indicator of $w_{0}\sigma
$ is $\eta'_{i} = 1-\eta _{i}$.  Using these facts and the definition
\eqref{e:F-twist} of $F^{\pi }_{\lambda }(x;q)$, one can check that
\eqref{e:winding-1} implies \eqref{e:winding-2}.

To prove \eqref{e:winding-1}, we begin by observing that for any given
$\lambda $ there exists $w\in S_{l}$ such that both $w^{-1}(\lambda )$
and $w^{-1}(\lambda -\eta )$ are dominant.  To see this, first choose
any $v$ such that $v(\lambda )=\lambda _{+}$ is dominant.  Since $\eta
$ is $\{0,1 \}$-valued, the weight $\mu =v(\lambda -\eta ) = \lambda
_{+}-v(\eta )$ has the property that for all $i<j$, if $\mu _{i}<\mu
_{j}$ then $(\lambda _{+})_{i} = (\lambda _{+})_{j}$.  Hence, there is
a permutation $u$ that fixes $\lambda _{+}$ and sorts $\mu $ into
weakly decreasing order, so $u(\mu ) = uv(\lambda -\eta )$ is
dominant.  Since $uv(\lambda ) = \lambda _{+}$ is also dominant,
$w^{-1} = uv$ works.

Now, Lemma \ref{lem:winding} implies
\begin{equation}\label{e:winding-expanded-2}
T_{\tau w}^{-1}\, T_{\tau }\, x^{-\eta }\, T_{\theta }^{-1}\,
T_{\theta w} (x^{w^{-1}(\lambda) }) \sim
x^{-w^{-1}(\eta )} x^{w^{-1}(\lambda )},
\end{equation}
where $\sim$ signifies that the expressions are equal up to a $q$
power factor.  Equivalently,
\begin{equation}\label{e:winding-expanded}
T_{\theta }^{-1}\, T_{\theta w}\,
x^{w^{-1}(\lambda) } \sim x^{\eta }\, T_{\tau }^{-1}\,
T_{\tau w}\, x^{w^{-1}(\lambda - \eta )}.
\end{equation}
Writing out the definitions of $E^{\theta ^{-1}}_{\lambda }$ and
$E^{\tau ^{-1}}_{\lambda - \eta }$, while ignoring $q$ power factors,
and using the fact that $\lambda _{+} = w^{-1}(\lambda )$ and
$(\lambda -\eta )_{+} = w^{-1}(\lambda -\eta )$ for this $w$,
\eqref{e:winding-expanded} implies that \eqref{e:winding-1} holds up
to a scalar factor $q^{e}$.  But we know that the $x^{\lambda}$ term
on each side has coefficient 1, so \eqref{e:winding-1} holds exactly.
\end{proof}

\subsection{Stable shuffle theorem}

We now prove an identity of formal power series with coefficients in
$\GL _{l}$ characters, that is, symmetric Laurent polynomials in
variables $x_{1},\ldots,x_{l}$.  When truncated to the polynomial
part, this identity will reduce to our shuffle theorem for paths under
a line (Theorem \ref{thm:main-G}).

\begin{thm}\label{thm:main-L}
Let $p ,s$ be real numbers with $p$ positive and irrational.  For $i =
1,\ldots,l$, let
\[
b_{i} = \lfloor s-p(i-1) \rfloor - \lfloor s-pi \rfloor.
\]
Let $c_{i} = \fp{s-p(i-1)}$, and let $\sigma \in S_{l}$ be the
permutation such that $\sigma (1),\ldots,\sigma (l)$ are in the same
relative order as $c_{l},\ldots,c_{1}$, i.e., $\sigma
(c_{l},\ldots,c_{1})$ is increasing.
For any non-negative integers $u,v$ we have the identity of formal
power series in $t$
\begin{equation}\label{e:main-L}
\Hcat _{b_{1}+u,b_{2},\ldots,b_{l-1},b_{l}-v} = \sum
_{a_{1},\ldots,a_{l-1}\geq 0} t^{|\aA |} \Lcal ^{\sigma
}_{((b_{l},\ldots,b_{1}) + (-v,a_{l-1},\ldots,a_{1})) /
(a_{l-1},\ldots,a_{1},-u)}(x;q),
\end{equation}
where $\Hcat _{\bb }$ is given by Definition
\ref{def:negut-catalanimal}.

\end{thm}

\begin{remark}\label{rem:extended-path}
If $\delta $ is the highest south-east lattice path weakly below the
line $y+px = s$, starting at $(0,\lfloor s \rfloor)$ and extending
forever (not stopping at the $x$ axis), then $b_{i}$ is the number of
south steps in $\delta $ along the line $x = i-1$, and $c_{i}$ is the
gap along $x=i-1$ between the given line and the highest point of
$\delta $ beneath it.
\end{remark}

\begin{proof}[Proof of Theorem \ref{thm:main-L}]
We will prove that for $\bb $, $\sigma $ as in the hypothesis of the
theorem, we have the stronger `unstraightened' identity
\begin{multline}\label{e:main-L-unstraightened}
x_{1}^{u}x_{l}^{-v}x^{\bb }\frac{\prod _{i+1<j}(1 -
q\, t\, x_{i}/x_{j})}{\prod _{i<j}(1 - t\, x_{i}/x_{j})}\\
 = \sum _{a_{1},\ldots,a_{l-1}\geq 0} t^{|\aA |} w_{0} \bigl(F^{\sigma ^{-1}
}_{(b_{l},\ldots,b_{1})+(-v,a_{l-1},\ldots,a_{1})}(x;q)
\overline{E^{\sigma ^{-1}}_{(a_{l-1},\ldots,a_{1},-u)}(x;q)} \bigr).
\end{multline}
By Proposition \ref{prop:L-formula}, applying the Hall-Littlewood
raising operator $\Hbold _{q}$ to both sides of
\eqref{e:main-L-unstraightened} yields \eqref{e:main-L}.

By construction, the $b_{i}$ take only values $\lfloor p \rfloor$ or
$\lceil p \rceil$, and since $b_{i}+c_{i} - c_{i+1}= p$, we have
\begin{equation}\label{e:b-c-sigma}
b_{i} = \lfloor p \rfloor\quad \Leftrightarrow \quad c_{i}>c_{i+1}
\quad \Leftrightarrow \quad \sigma (l-i) < \sigma (l-i+1).
\end{equation}
In particular, $b_{l}\leq b_{l-1}+1$, hence $b_{l}-v\leq
b_{l-1}+a_{l-1}+1$, and if equality holds, then $b_{l-1} = \lfloor p
\rfloor$, so $\sigma (1)<\sigma (2)$.  Using Lemma
\ref{lem:factorization}, and recalling that the definition
\eqref{e:F-twist} of $F^{\sigma }_{\lambda }(x;q)$ is $
\overline{E^{\sigma w_{0}}_{-\lambda }(x;q)}$, we have
\begin{gather}\label{e:factor-E}
E^{\sigma ^{-1}}_{(a_{l-1},\ldots,a_{1},-u)}(x;q) = x_{l}^{-u} E^{\tau
^{-1}}_{(a_{l-1},
\ldots, a_{1})}(x_{1},\ldots,x_{l-1};q)\\
\label{e:factor-F}
F^{\sigma ^{-1}}_{(b_{l},\ldots,b_{1})+(-v,a_{l-1},\ldots,a_{1})}(x;q) =
x_{1}^{b_{l}-v} F^{\theta ^{-1}}_{(b_{l-1},\ldots,b_{1})+(a_{l-1},\ldots,a_{1})}(x_{2},\ldots,x_{l};q),
\end{gather}
where $\tau $, $\theta $ are the head and tail of $\sigma $, as in
Proposition \ref{prop:winding}.  Note that $\sigma $ is a winding
permutation.  From \eqref{e:b-c-sigma} we also see that
$(b_{l-1},\ldots,b_{1}) = \eta + \lfloor p \rfloor \cdot
(1,\ldots,1)$, where $\eta $ is the descent indicator of $\sigma $.
Adding a constant vector $k\cdot (1,\ldots,1)$ to the index $\lambda $
multiplies any $F^{\pi }_{\lambda }(x;q)$ by $(\prod _{i}x_{i})^{k}$.
Using this and Proposition \ref{prop:winding}, we can replace
\eqref{e:factor-F} with
\begin{equation}\label{e:factor-F-2}
F^{\sigma ^{-1}}_{(b_{l},\ldots,b_{1})+(-v,a_{l-1},\ldots,a_{1})}(x;q)
= x_{1}^{b_{l}-v}x_{2}^{b_{l-1}}\cdots x_{l}^{b_{1}} F^{\tau ^{-1}
}_{(a_{l-1},\ldots,a_{1})}(x_{2},\ldots,x_{l};q)
\end{equation}
Using the Cauchy identity \eqref{e:Cauchy} from Theorem
\ref{thm:Cauchy} in $l-1$ variables, with twisting permutation $\tau
^{-1}$, and substituting $x_{i}^{-1}$ for $x_{i}$, we obtain
\begin{equation}\label{e:Cauchy-applied}
\frac{\prod _{i<j} (1 - q\, t\, y_{j}/x_{i})}{\prod _{i\leq j} (1 -
t\, y_{j}/x_{i})} = \sum _{\aA \geq 0} t^{|\aA |}\, F^{\tau
^{-1}}_{\aA }(y_{1},\ldots,y_{l-1};q)\overline{E^{\tau ^{-1}}_{\aA
}(x_{1},\ldots,x_{l-1};q)}.
\end{equation}
Setting $y_{i} = x_{i+1}$ in \eqref{e:Cauchy-applied} and multiplying
both sides by $w_{0}(x_{1}^{u}x_{l}^{-v}x^{\bb })$, then using
\eqref{e:factor-E} and \eqref{e:factor-F-2}, and finally applying
$w_{0}$ to both sides, yields \eqref{e:main-L-unstraightened}.
\end{proof}

\subsection{LLT data and the \texorpdfstring{$\dinv $} {dinv}
statistic}
\label{ss:dinv-data}

Our next goal is to deduce the combinatorial version of our shuffle
theorem---that is, the identity \eqref{e:main-G-pre} previewed in the
introduction and restated as \eqref{e:main-G}, below---from Theorem
\ref{thm:main-L}.  To do this we first need to define the data that
will serve to attach LLT polynomials to lattice paths, and relate
these to the combinatorial statistic $\dinv _{p}(\lambda )$.

We will be concerned with lattice paths $\lambda $ lying weakly below
the line segment
\begin{equation}\label{e:line}
y + p\, x= s\qquad (p = s/r)
\end{equation}
between arbitrary points $(0,s)$ and $(r,0)$ on the positive $y$ and
$x$ axes.

We always assume that the slope $-p$ of the line is irrational.
Clearly it is possible to perturb any line slightly so as make its
slope irrational, without changing the set of lattice points, and
therefore also the set of lattice paths, that lie below the line.  All
dependence on $p$ in the combinatorial constructions to follow comes
from comparisons between $p$ and various rational numbers.  By taking
$p$ to be irrational, we avoid the need to resolve ambiguities that
would result from equality occurring in the comparisons.

\begin{defn}\label{def:dinv}
Let $\lambda $ be a south-east lattice path in the first quadrant with
endpoints on the axes.  Let $Y$ be the Young diagram enclosed by the
positive axes and $\lambda $.  The {\em arm} and {\em leg} of a box
$y\in Y$ are, as usual, the number of boxes in $Y$ strictly east of
$y$ and strictly north of $y$, respectively.  Given a positive
irrational number $p$, we define $\dinv _{p}(\lambda )$ to be the
number of boxes in $Y$ whose arm $a$ and leg $\ell$ satisfy
\begin{equation}\label{e:p-balanced-2}
\frac{\ell}{a+1}<p<\frac{\ell+1}{a},
\end{equation}
where we interpret $(\ell+1)/a$ as $+\infty $ if $a = 0$.
\end{defn}

Geometrically, condition \eqref{e:p-balanced-2} means that some line
of slope $-p$ crosses both the east step in $\lambda $ at the top of
the leg and the south step at the end of the arm, as shown in Figure
\ref{fig:balanced-hook}.  Since $p$ is irrational, such a line can
always be assumed to pass through the interiors of the two steps.

\begin{figure}
\begin{tikzpicture}[scale=.5]
\draw [help lines, ->] (0,0) -- (0,5.5);
\draw [help lines, ->] (0,0) -- (8.5,0);
\filldraw [fill = black!10, draw=black] (1,4) -- (1,1) -- (5,1) -- (5,2)
-- (2,2) -- (2,4) -- cycle;
\draw [very thick] (0,5) -- (0,4) -- (3,4) -- (3,3) -- (5,3) -- (5,1) -- (7,1)
-- (7,0) -- (8,0);
\draw (.5,4.5) -- (5.8,1.2);
\end{tikzpicture}
\caption{\label{fig:balanced-hook}}
\end{figure}

To each lattice path weakly below the line \eqref{e:line} we now attach a
tuple of one-row skew shapes $\beta /\alpha $ and a permutation
$\sigma $.  The index $\nu (\lambda )$ of the LLT polynomial in
\eqref{e:main-G-pre} and \eqref{e:main-G} will be defined in terms of
these data.

\begin{defn}\label{def:LLT-data}
Let $\lambda $ be a south-east lattice path from $(0,\lfloor s
\rfloor)$ to $(l-1,0)$ which is weakly below the line $y+p\, x = s$ in
\eqref{e:line}, where $l-1\leq r$ and $p$ is irrational.  For $i =
1,\ldots,l$, let
\begin{equation}\label{e:top-points}
d_{i} = \lfloor s-p(i-1) \rfloor
\end{equation}
be the $y$ coordinate of the highest lattice point weakly below the
given line at $x=i-1$.  Let
\begin{equation}\label{e:alpha-beta}
\alpha = (\alpha _{l},\ldots,\alpha _{1}),\quad \beta = (\beta
_{l},\ldots,\beta _{1})
\end{equation}
be the vectors of integers $0\leq \alpha _{i}\leq \beta _{i}$, written
in reverse order, such that the south steps in $\lambda $ on the line
$x = i-1$ go from $y = d_{i}-\alpha _{i}$ to $y = d_{i}-\beta _{i}$.
Let
\begin{equation}\label{e:gaps}
c_{i}  = s-p(i-1)-d_{i} = \fp{ s-p(i-1) }
\end{equation}
be the gap between the given line and the highest lattice point weakly
below it along the line $x = i-1$.
Let $\sigma \in S_{l}$ be the permutation with $\sigma
(1),\ldots,\sigma (l)$ in the same relative order as
$c_{l},\ldots,c_{1}$, i.e., such that $\sigma (c_{l},\ldots,c_{1})$ is
increasing.  The vectors $\alpha $ and $\beta $ and the permutation
$\sigma $ are the {\em LLT data} associated with $\lambda $ and the
given line.
\end{defn}

\begin{example}\label{ex:LLT-data}
The first diagram in Figure \ref{fig:lambda-prime} shows a line $y +
px = s$ with $p\approx 1.36$, $s\approx 9.27$, and a path $\lambda $
below it from $(0,\lfloor s \rfloor) = (0,9)$ to $(l-1,0) = (6,0)$
with $l = 7$.

In this example, the $y$ coordinates of the highest lattice points
below the line at $x = 0,\ldots,6$ are $(d_{1},\ldots,d_{7}) =
(9,7,6,5,3,2,1)$.  The runs of south steps in $\lambda $ go from $y$
coordinates $(9,6,6,3,1,1,0)$ to $(6,6,3,1,1,0,0)$.  Subtracting these
from the $d_{i}$ and listing them in reverse order gives
\begin{equation}\label{e:alpha-beta-example}
\alpha = (1,1,2,2,0,1,0),\quad \beta = (1,2,2,4,3,1,3).
\end{equation}
The gaps, in reverse order, are $(c_{7},\ldots,c_{1})\approx (.11,
.47, .83, .19, .55, .91, .27)$, giving
\begin{equation}\label{e:sigma-example}
\sigma = (1,4,6,2,5,7,3).
\end{equation}
\end{example}

\begin{prop}\label{prop:h=dinv}
Given the line \eqref{e:line} and a lattice path $\lambda $ weakly
below it satisfying the conditions in Definition \ref{def:LLT-data},
let $\alpha ,\beta, \sigma $ be the associated LLT data. Then
\begin{equation}\label{e:h=dinv}
\dinv _{p}(\lambda ) = h_{\sigma }(\beta /\alpha ),
\end{equation}
where $\dinv _{p}(\lambda )$ is given by Definition \ref{def:dinv} and
$h_{\sigma }(\beta /\alpha )$ is as in Corollary~\ref{cor:G-versus-L}.
\end{prop}

\begin{proof}
Let $\lambda '$ be the image of $\lambda $ under the transformation in
the plane that sends $(x,y)$ to $(x,y+px)$. Then $\lambda '$ is a path
composed of unit south steps and sloped steps $(1,p)$ (transforms of
east steps), which starts at $(0,\lfloor s \rfloor$) and stays weakly
below the horizontal line $y = s$ (transform of the line $y+p\, x=s$).

The south steps in $\lambda '$ on the line $x = i-1$ run from $y = s -
(c_{i} + \alpha _{i})$ to $y = s - (c_{i} + \beta _{i})$.  This means
that if we offset the $i$-th row $(\beta _{l+1-i})/(\alpha _{l+1-i})$
in the tuple of one-row skew diagrams $\beta /\alpha $ by $c_{l+1-i}$,
then the $x$ coordinate on each box of $\beta /\alpha $ covers the
same unit interval as does the distance below the line $y=s$ on the
south step in $\lambda '$ corresponding to that box.  See Figure
\ref{fig:lambda-prime} for an example.

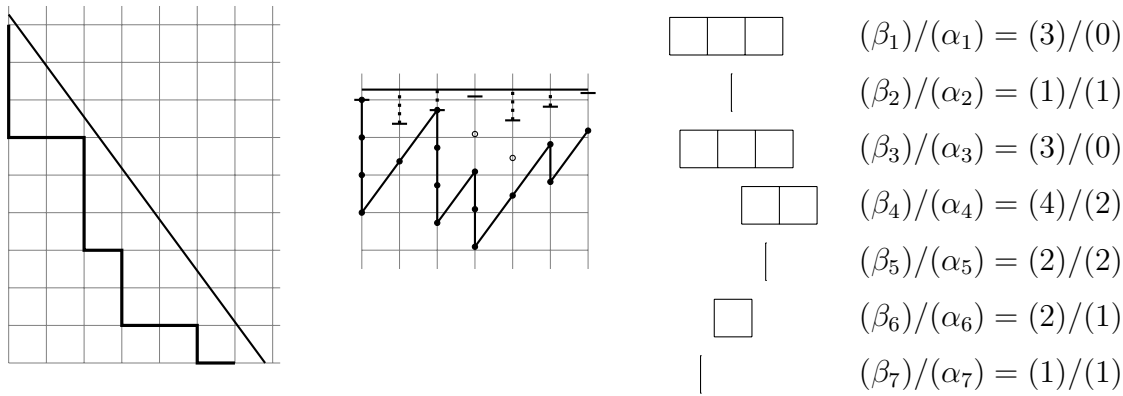
\begin{figure}
\begin{tikzpicture}[scale=.5, baseline=-.5cm]
\draw[help lines] (0,0) grid (7.2,9.5);
\draw [thick] (0,9.2727) -- (6.8,0);
\draw [very thick] (0,9) -- (0,6) -- (2,6) -- (2,3) -- (3,3) -- (3,1)
-- (5,1) -- (5,0) -- (6,0);
\end{tikzpicture}
\qquad
\begin{tikzpicture}[scale=.5, baseline=.5cm]
\draw[help lines] (0,4.5) grid (6,9.7);
\draw [thick] (0,9.2727) -- (6,9.2727);
\draw [thick] (0,9) -- (0,6) -- (2,8.7272) -- (2,5.7272) -- (3,7.0909)
-- (3,5.0909) -- (5,7.8181) -- (5,6.8181) -- (6,8.1818);
\draw [fill] (0,9) circle [radius=.07];
\draw [fill] (0,8) circle [radius=.07];
\draw [fill] (0,7) circle [radius=.07];
\draw [fill] (0,6) circle [radius=.07];
\draw [fill] (1,7.3636) circle [radius=.07];
\draw [fill] (2,8.7272) circle [radius=.07];
\draw [fill] (2,7.7272) circle [radius=.07];
\draw [fill] (2,6.7272) circle [radius=.07];
\draw [fill] (2,5.7272) circle [radius=.07];
\draw (3,8.0909) circle [radius=.07];
\draw [fill] (3,7.0909) circle [radius=.07];
\draw [fill] (3,6.0909) circle [radius=.07];
\draw [fill] (3,5.0909) circle [radius=.07];
\draw (4,7.4545) circle [radius=.07];
\draw [fill] (4,6.4545) circle [radius=.07];
\draw [fill] (5,7.8181) circle [radius=.07];
\draw [fill] (5,6.8181) circle [radius=.07];
\draw [fill] (6,8.1818) circle [radius=.07];
\draw [thick] (-.2,9) -- (.2,9);
\draw [thick] (.8,8.3636) -- (1.2,8.3636);
\draw [very thick,dotted] (1,8.3636) -- (1,9.2727);
\draw [thick] (1.8,8.7272) -- (2.2,8.7272);
\draw [very thick,dotted] (2,8.7272) -- (2,9.2727);
\draw [thick] (2.8,9.0909) -- (3.2,9.0909);
\draw [thick] (3.8,8.4545) -- (4.2,8.4545);
\draw [very thick,dotted] (4,8.4545) -- (4,9.2727);
\draw [thick] (4.8,8.8181) -- (5.2,8.8181);
\draw [very thick,dotted] (5,8.8181) -- (5,9.2727);
\draw [thick] (5.8,9.1818) -- (6.2,9.1818);
\end{tikzpicture}
\qquad
\begin{tikzpicture}[scale=.5]
\draw [shift={(.0909,0)}] (1,1) grid (1,0);
\node [right] at (5,.5) {$(\beta _{7})/(\alpha _{7}) = (1)/(1)$};
\draw [shift={(.4545,1.5)}] (2,1) grid (1,0);
\node [right] at (5,2) {$(\beta _{6})/(\alpha _{6}) = (2)/(1)$};
\draw [shift={(.8181,3)}] (2,1) grid (2,0);
\node [right] at (5,3.5) {$(\beta _{5})/(\alpha _{5}) = (2)/(2)$};
\draw [shift={(.1818,4.5)}](4,1) grid (2,0);
\node [right] at (5,5) {$(\beta _{4})/(\alpha _{4}) = (4)/(2)$};
\draw [shift={(.5454,6)}] (3,1) grid (0,0);
\node [right] at (5,6.5) {$(\beta _{3})/(\alpha _{3}) = (3)/(0)$};
\draw [shift={(.9090,7.5)}](1,1) grid (1,0);
\node [right] at (5,8) {$(\beta _{2})/(\alpha _{2}) = (1)/(1)$};
\draw [shift={(.2727,9)}] (3,1) grid (0,0);
\node [right] at (5,9.5) {$(\beta _{1})/(\alpha _{1}) = (3)/(0)$};
\end{tikzpicture}
\caption{\label{fig:lambda-prime}
(i) A path $\lambda $ under $y+px=s$ with $p\approx 1.36$, $s\approx
9.27$, $l=7$. (ii) Transformed path $\lambda '$ under $y = s$, with
gaps $c_{i}$ marked.  (iii) Bottom to top: tuple of rows $(\beta
_{7},\ldots,\beta _{1}) /(\alpha_{7},\ldots,\alpha _{1}) $ offset by
$(c_{7},\ldots,c_{1})$.}
\end{figure}

Since $0<c_{i}<1$ and the numbers $c_{l},\ldots,c_{1}$ are in the same
relative order as $\sigma (1),\ldots,\sigma (l)$, the description of
$h_{\sigma }(\beta /\alpha )$ in Remark \ref{rem:triples} still
applies if we offset row $i$ by $c_{l+1-i}$ instead of $\epsilon
\sigma (i)$.  Mapping this onto $\lambda '$, we see that $h_{\sigma
}(\beta /\alpha )$ is the number of horizontal alignments between any
endpoint of a step in $\lambda '$ and the interior of a south step
occurring later in $\lambda '$.  To put this another way, for each
south step $S$ in $\lambda '$, let $B_{S}$ denote the interior of the
horizontal band of height 1 to the left of $S$ in the plane.  Then
$h_{\sigma }(\beta /\alpha )$ is the number of pairs consisting of a
south step $S$ and a point $P\in B_{S}$ which is an endpoint of a step
in $\lambda '$.

For comparison, $\dinv _{p}(\lambda )$ is the number of pairs
consisting of a south step $S$ in $\lambda '$ and a sloped step which
meets $B_{S}$.  To complete the proof it suffices to show that each
band $B_{S}$ contains the same number of step endpoints $P$ as the
number of sloped steps that meet $B_{S}$.  In fact, we make the
following stronger claim: {\em within each band $B_{S}$, step
endpoints alternate from left to right with fragments of sloped steps,
starting with a step endpoint and ending with a sloped step fragment.}

To see this, consider a connected component $C$ of $\lambda '\cap
B_{S}$.  Each component $C$ either enters $B_{S}$ from above along a
south step or from below along a sloped step, and exits $B_{S}$ either
at the top along a sloped step or at the bottom along a south step,
except in two degenerate situations.  One of these occurs if $C$
contains the starting point $(0,\lfloor s \rfloor)$ of $\lambda '$.
In this case we regard $C$ as entering $B_{S}$ from above.  The other
is if $C$ contains a sloped step that adjoins $S$.  Then we regard $C$
as exiting $B_{S}$ at the top.

Each component $C$ thus belongs to one of four cases shown in Figure
\ref{fig:band-components}.  Note that since $B_{S}$ has height 1, it
cannot contain a full south step of $\lambda '$.  In Figure
\ref{fig:band-components} we have chosen $p<1$ in order to illustrate
the possibility that $B_{S}$ might contain full sloped steps of
$\lambda '$.  If $p>1$, then $B_{S}$ can only meet sloped steps in
proper fragments.

\begin{figure}
\begin{tikzpicture}
\filldraw [dashed, fill=black!10] (0,1) -- (3,1) -- (3,0) -- (0,0);
\draw [very thick] (3,0) -- (3,1);
\node [right] at (3,.5) {$S$};
\draw [thick] (.5,1.3) -- (.5,.3) -- (1.5, .75) -- (2.5, 1.2);
\draw [fill] (.5,.3) circle [radius=.05];
\draw [fill] (1.5, .75) circle [radius=.05];

\filldraw [dashed, fill=black!10] (4,1) -- (7,1) -- (7,0) -- (4,0);
\draw [very thick] (7,0) -- (7,1);
\node [right] at (7,.5) {$S$};
\draw [thick] (5,1.3) -- (5,.3) -- (6, .75) -- (6, -.25);
\draw [fill] (5,.3) circle [radius=.05];
\draw [fill] (6,.75) circle [radius=.05];

\filldraw [dashed, fill=black!10] (8,1) -- (11,1) -- (11,0) -- (8,0);
\draw [very thick] (11,0) -- (11,1);
\node [right] at (11,.5) {$S$};
\draw [thick] (8,-.15) -- (9,.3) -- (10, .75) -- (11, 1.2);
\draw [fill] (9,.3) circle [radius=.05];
\draw [fill] (10, .75) circle [radius=.05];

\filldraw [dashed, fill=black!10] (12,1) -- (15,1) -- (15,0) -- (12,0);
\draw [very thick] (15,0) -- (15,1);
\node [right] at (15,.5) {$S$};
\draw [thick] (12.5,-.15) -- (13.5,.3) -- (14.5, .75) -- (14.5, -.25);
\draw [fill] (13.5,.3) circle [radius=.05];
\draw [fill] (14.5, .75) circle [radius=.05];
\end{tikzpicture}
\caption{\label{fig:band-components}}
\end{figure}

On each component $C$, step endpoints clearly alternate with sloped
step fragments, starting with a step endpoint if $C$ enters from
above, or with a sloped step fragment if $C$ enters from below, and
ending with a step endpoint if $C$ exits at the bottom, or with a
sloped step fragment if $C$ exits at the top.  Since the distance from
the line $y=s$ to the starting point of $\lambda '$ is less than $1$,
the leftmost component $C$ of $\lambda '\cap B_{S}$ always enters
$B_{S}$ from the top.  Each subsequent component from left to right
must enter $B_{S}$ from the same side (top or bottom) that the
previous component exited.  This implies the claim stated above.
\end{proof}

\subsection{Shuffle theorem for paths under a line}
\label{ss:main-G}

We now prove the identity previewed as \eqref{e:main-G-pre} in the
introduction.

\begin{thm}\label{thm:main-G}
Let $r,s$ be positive real numbers with $p = s/r$ irrational.  We
have the identity
\begin{equation}\label{e:main-G}
D_{b_{1},\ldots,b_{l}} \cdot 1 = \sum _{\lambda } t^{a(\lambda )}
q^{\dinv _{p}(\lambda )} \omega (\Gcal_{\nu (\lambda )}(X;q^{-1})),
\end{equation}
where the sum is over lattice paths $\lambda $ from $(0,\lfloor s
\rfloor)$ to $(\lfloor r \rfloor,0)$ lying weakly below the line
\eqref{e:line} through $(0,s)$ and $(r,0)$, and the other pieces of
\eqref{e:main-G} are defined as follows.

The integer $a(\lambda )$ is the number of lattice squares enclosed
between $\lambda $ and $\delta $, where $\delta$ is the highest path from
$(0,\lfloor s\rfloor)$ to $(\lfloor r \rfloor,0)$ weakly below the
given line. The index $b_{i}$ is the number of south steps in $\delta
$ along the line $x = i-1$, for $i=1,\ldots,l$, where $l = \lfloor r
\rfloor+1$.  The integer $\dinv _{p}(\lambda )$ is given by Definition
\ref{def:dinv}.

The LLT polynomial $\Gcal_{\nu (\lambda )}(X;q)$ is indexed by the
tuple of one-row skew shapes $\nu (\lambda ) = \sigma (\beta /\alpha
)$, where $\alpha ,\beta $ and $\sigma $ are the LLT data associated
to $\lambda $ in Definition \ref{def:LLT-data}.  More explicitly,
$\sigma (\beta /\alpha ) = (\beta _{w_{0}\sigma^{-1}(1)},\ldots,\beta
_{w_{0}\sigma^{-1}(l)})/(\alpha_{w_{0}\sigma^{-1}(1)},\ldots,\alpha
_{w_{0}\sigma^{-1}(l)})$, where $\alpha =(\alpha _{l},\ldots,\alpha
_{1})$ and $\beta =(\beta _{l},\ldots,\beta _{1})$.

The operator $D_{b_{1},\ldots,b_{l}}$ on the left hand side is a Negut
element in $\Ecal $, as defined in \S \ref{ss:Negut-elements}, so that
$D_{b_{1},\ldots,b_{l}} \cdot 1$ satisfies \eqref{e:Db-Hb}.
\end{thm}

\begin{proof}
We prove the equivalent identity
\begin{equation}\label{e:omega-main}
\omega (D_{b_{1},\ldots,b_{l}} \cdot 1) = \sum _{\lambda } t^{a(\lambda )}
q^{\dinv _{p}(\lambda )} \Gcal_{\nu (\lambda )}(X;q^{-1}).
\end{equation}
By Corollary \ref{cor:Db-Hb} and Lemma \ref{lem:finite-variables},
both sides of \eqref{e:omega-main} involve only Schur functions
$s_{\lambda }(X)$ indexed by partitions such that $\ell(\lambda )\leq l$.
It therefore suffices to prove that \eqref{e:omega-main} holds when
evaluated in $l$ variables $x_{1},\ldots , x_{l}$.  After doing this
and using the formula \eqref{e:Db-Hb} from Corollary \ref{cor:Db-Hb},
the desired identity becomes
\begin{equation}\label{e:main-G-Hb}
(\Hcat _{\bb })_{\pol} = \sum _{\lambda } t^{a(\lambda )}\, q^{\dinv
_{p}(\lambda )}\, \Gcal _{\nu (\lambda )}(x_{1},\ldots,x_{l}; q^{-1}).
\end{equation}
This is the same identity \eqref{e:main-G-Hb-pre} that was mentioned
in the introduction to \S \ref{s:schiffmann}.  We now prove it using
Theorem \ref{thm:main-L}.

Let $b'_{i} = \lfloor s-p(i-1) \rfloor - \lfloor s-pi \rfloor$.  As in
Remark \ref{rem:extended-path}, this is the number of south steps
along $x = i-1$ in the highest south-east path $\delta '$ under our
given line, where $\delta '$ starts at $(0, \lfloor s \rfloor)$ and
extends forever.  For $i<l$ we have $b'_{i} = b_{i}$.  On the line
$x=l-1=\lfloor r \rfloor$, however, the path $\delta $ stops at
$(l-1,0)$, while $\delta '$ may extend below the $x$-axis, giving
$b_{l}\leq b'_{l}$.

We now apply Theorem \ref{thm:main-L} with $b'_{i}$ in place of
$b_{i}$, $u=0$, and $v = b'_{l} - b_{l}$, and then take the polynomial
part on both sides of \eqref{e:main-L}.  This gives the same left hand
side as in \eqref{e:main-G-Hb}.  On the right hand side, by
Corollary~\ref{cor:G-versus-L}, only those terms survive for which
the index $\aA $ satisfies $(a_{l-1},\ldots,a_{1},0)\leq
(b_{l},\ldots,b_{1})+(0, a_{l-1},\ldots,a_{1})$ in each coordinate,
that is, for which
\begin{equation}\label{e:area-columns}
a_{l-1}\leq b_{l}\quad \text{and}\quad a_{i}\leq a_{i+1} + b_{i+1}
\text{ for } i=1,\ldots,l-2.
\end{equation}

Now, \eqref{e:area-columns} is precisely the condition for there to
exist a (unique) lattice path $\lambda $ from $(0,\lfloor s \rfloor)$
to $(\lfloor r \rfloor,0)$ such that $a_{i}$ is the number of lattice
squares in the $i$-th column (defined by $x\in [i-1,i]$) of the region
between $\lambda $ and the highest path $\delta $.  Moreover, when
\eqref{e:area-columns} holds, the LLT data for $\lambda $, as in
Definition \ref{def:LLT-data}, are given by
\begin{equation}\label{e:particular-LLT-data}
\begin{aligned}
\beta &	 = (b_{l},\ldots,b_{1})+(0, a_{l-1},\ldots,a_{1}),\\
\alpha & = (a_{l-1},\ldots,a_{1},0),
\end{aligned}
\end{equation}
and $\sigma \in S_{l}$ such that $\sigma (1),\ldots,\sigma (l)$ are in
the same relative order as $c_{l},\ldots,c_{1}$, where $c_{i} =
\fp{s-p(i-1)}$, as in Theorem \ref{thm:main-L}.  Hence, by 
Corollary~\ref{cor:G-versus-L} and Proposition~\ref{prop:h=dinv},
we have
\begin{equation}\label{e:L=G}
\Lcal ^{\sigma }_{((b_{l},\ldots,b_{1})+(0,
a_{l-1},\ldots,a_{1}))/(a_{l-1},\ldots,a_{1},0)}(x;q)_{\pol}
= q^{\dinv
_{p}(\lambda )} \Gcal_{\nu (\lambda )}(x;q^{-1}).
\end{equation}
When \eqref{e:area-columns} holds we clearly also have $a(\lambda ) =
|\aA |$.  This shows that the polynomial part of the right hand side
in \eqref{e:main-L} is the same as the right hand side of
\eqref{e:main-G-Hb}.
\end{proof}

\begin{remark}\label{rem:y-extension}
The preceding argument also goes through with $u>0$ in Theorem
\ref{thm:main-L} to give a slightly more general version of Theorem
\ref{thm:main-G} in which the sum is over lattice paths $\lambda $
that start at a higher point $(0,n)$ on the $y$ axis, with $n =
\lfloor s \rfloor + u$, go directly south to $(0,\lfloor s \rfloor)$,
and then continue below the given line to $(l-1,0)$ as before.

The corresponding modifications to \eqref{e:main-G} are {\em (i)} the
index $b_{1}$ on the left hand side is the number of south steps in
$\lambda $ on the $y$ axis including the extension to $(0,n)$, and
{\em (ii)} the row in $\nu (\lambda )$ corresponding to south steps in
$\lambda $ on the $y$ axis is also extended accordingly.
\end{remark}

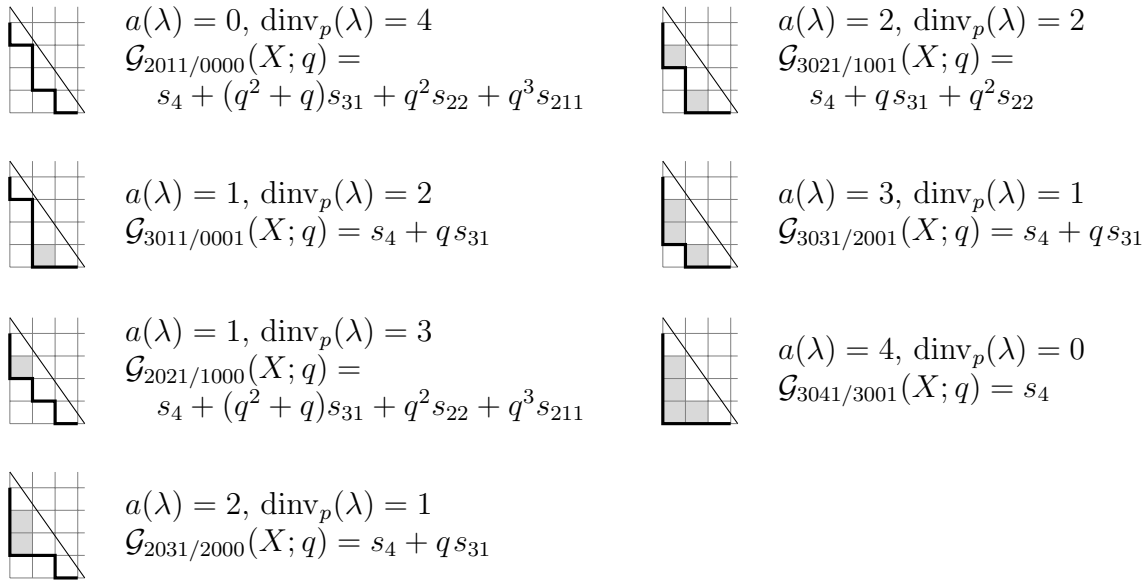
\begin{figure}
\[
\begin{array}{ll@{\qquad }ll}
\begin{tikzpicture}[scale=.3,baseline=.6cm]
\draw [help lines] (0,0) grid (331/100,47/10);
\draw (0,47/10) -- (331/100,0);
\draw [very thick] (0,4) -- (0,3) -- (1,3) -- (1,1) -- (2,1) -- (2,0) -- (3,0);
\end{tikzpicture}
&
\begin{array}{l}
a(\lambda ) = 0,\, \dinv _{p} (\lambda ) = 4\\
\Gcal_{ 2011 / 0000}(X;q) =\\
\quad s_{4} + (q^2 + q)s_{31} + q^2s_{22} + q^3s_{211}
\end{array}
&
\begin{tikzpicture}[scale=.3,baseline=.6cm]
\draw[draw=none, fill=black!15]  (0,2) rectangle (1,3);
\draw[draw=none, fill=black!15]  (1,0) rectangle (2,1);
\draw[help lines] (0,0) grid (331/100,47/10);
\draw (0,47/10) -- (331/100,0);
\draw[very thick] (0,4) -- (0,2) -- (1,2) -- (1,0) -- (2,0) -- (2,0) -- (3,0);
\end{tikzpicture}
&
\begin{array}{l}
a(\lambda ) = 2,\, \dinv _{p} (\lambda ) = 2\\
\Gcal_{ 3021 / 1001}(X;q) = \\
\quad s_{4} + q\hspace{.3mm} s_{31} + q^2s_{22}
\end{array}\\[6ex]

\begin{tikzpicture}[scale=.3,baseline=.6cm]
\draw[draw=none, fill=black!15]  (1,0) rectangle (2,1);
\draw [help lines] (0,0) grid (331/100,47/10);
\draw (0,47/10) -- (331/100,0);
\draw [very thick]  (0,4) -- (0,3) -- (1,3) -- (1,0) -- (2,0) -- (2,0) -- (3,0);
\end{tikzpicture}
&
\begin{array}{l}
a(\lambda ) = 1,\, \dinv _{p} (\lambda ) = 2\\
\Gcal_{ 3011 / 0001}(X;q) =
s_{4} + q\hspace{.3mm} s_{31}
\end{array}
&
\begin{tikzpicture}[scale=.3,baseline=.6cm]
\draw[draw=none, fill=black!15]  (0,1) rectangle (1,3);
\draw[draw=none, fill=black!15]  (1,0) rectangle (2,1);
\draw[help lines] (0,0) grid (331/100,47/10);
\draw (0,47/10) -- (331/100,0);
\draw[very thick] (0,4) -- (0,1) -- (1,1) -- (1,0) -- (2,0) -- (2,0) -- (3,0);
\end{tikzpicture}
&
\begin{array}{l}
a(\lambda ) = 3,\, \dinv _{p} (\lambda ) = 1\\
\Gcal_{ 3031 / 2001}(X;q) =
s_{4} + q\hspace{.3mm} s_{31} 
\end{array}\\[6ex]

\begin{tikzpicture}[scale=.3,baseline=.6cm]
\draw[draw=none, fill=black!15]  (0,2) rectangle (1,3);
\draw [help lines] (0,0) grid (331/100,47/10);
\draw (0,47/10) -- (331/100,0);
\draw [very thick]  (0,4) -- (0,2) -- (1,2) -- (1,1) -- (2,1) -- (2,0) -- (3,0);
\end{tikzpicture}
&
\begin{array}{l}
a(\lambda ) = 1,\, \dinv _{p} (\lambda ) = 3\\
\Gcal_{ 2021 / 1000}(X;q) = \\
\quad s_{4} + (q^2 + q)s_{31} + q^2s_{22} + q^3s_{211} 
\end{array}
&
\begin{tikzpicture}[scale=.3,baseline=.6cm]
\draw[draw=none, fill=black!15]  (0,0) rectangle (1,3);
\draw[draw=none, fill=black!15]  (1,0) rectangle (2,1);
\draw[help lines] (0,0) grid (331/100,47/10);
\draw (0,47/10) -- (331/100,0);
\draw[very thick] (0,4) -- (0,0) -- (1,0) -- (1,0) -- (2,0) -- (2,0) -- (3,0);
\end{tikzpicture}
&
\begin{array}{l}
a(\lambda ) = 4,\, \dinv _{p} (\lambda ) = 0\\
\Gcal_{ 3041 / 3001}(X;q) = s_{4} 
\end{array}\\[6ex]

\begin{tikzpicture}[scale=.3,baseline=.6cm]
\draw[draw=none, fill=black!15]  (0,1) rectangle (1,3);
\draw [help lines] (0,0) grid (331/100,47/10);
\draw (0,47/10) -- (331/100,0);
\draw [very thick] (0,4) -- (0,1) -- (1,1) -- (1,1) -- (2,1) -- (2,0) -- (3,0);
\end{tikzpicture}
&
\begin{array}{l}
a(\lambda ) = 2,\, \dinv _{p} (\lambda ) = 1\\
\Gcal_{ 2031 / 2000}(X;q) = s_{4} + q\hspace{.3mm} s_{31} 
\end{array}
\end{array}
\]
\caption{\label{fig:main-G}
An illustration of Theorem~\ref{thm:main-G} as described in
Example~\ref{ex:figure6}.}
\end{figure}

\begin{example}\label{ex:figure6}
Figure \ref{fig:main-G} illustrates Theorem~\ref{thm:main-G}
for $s \approx 4.7$, $r \approx 3.31$.  We have
$(c_4, c_3,c_2,c_1) \approx ( .44,.86, .28,.70)$ and $\sigma =
(1,2,3,4) \mapsto (2,4,1,3)$.  The paths $\lambda$ are shown along
with the corresponding statistics and LLT polynomials
$\Gcal_{\nu(\lambda)}(X;q) =
\Gcal_{\sigma(\beta)/\sigma(\alpha)}(X;q)$.  The highest path $\delta$
is the one at the top left in the figure and $(b_1,b_2,b_3,b_4) = (1,2,1,0)$.  The
left side of \eqref{e:omega-main}, evaluated in $l = 4$ variables, is
then
\begin{multline}\label{eq:figure6-ex}
\omega (D_{1,2,1,0} \cdot 1)(x) = \Hcat _{(1,2,1,0)}(x)_{\pol} =\\
\sigmabold \big(\frac{ x_1x_2^2x_3(1 - q\, t\, x_{1} /x_3)(1 - q\, t\,
x_{2} /x_4)(1 - q\, t\, x_{1} /x_4)} {\prod_{1\le i < j \le 4}(1 - q\,
x_{i}/ x_{j})(1 - t\,x_{i}/ x_{j})} \big)_{\pol}\,.
\end{multline}
To see that
\eqref{e:omega-main} holds at $t = 0$, observe that the right hand side of~\eqref{eq:figure6-ex} becomes the Hall-Littlewood polynomial
$\Hbold_q(x_1x_2^2x_3)_{\pol} = q H_{2110}(x;q)$, which agrees with
the area 0 contribution $q^4 \Gcal_{2011 / 0000}(x;q^{-1})$.
\end{example}

\begin{defn}
\label{def:qt cat number} For $\bb \in \ZZ^l$, the \emph{generalized
$q,t$-Catalan number $C_\bb(q,t)$} is the coefficient of the single
row Schur function $s_{(|\bb|)}(X)$ in $\omega (D_\bb \cdot 1)$.
\end{defn}
When $\mathbf{b}= 1^l$, $C_\bb(q,t)$ is the $q,t$-Catalan number
introduced by Garsia and the second author \cite{GaHa96}.  The
generalized $q,t$-Catalan numbers have been studied in
\cite{GHRS}---see \S \ref{ss:Negut-conjecture}.

\begin{cor}
With $\mathbf{b} = (b_1,\dots, b_l)$, $r$, $s$, and $p = s/r$ as in
Theorem \ref{thm:main-G},
\begin{align}
C_\bb(q,t) = \sum _{\lambda } t^{a(\lambda )}
q^{\dinv _{p}(\lambda )},
\end{align}
where the sum is over lattice paths $\lambda $ from $(0,\lfloor
s\rfloor)$ to $(\lfloor r \rfloor,0)$ lying weakly below the line
through $(0,s)$ and $(r,0)$.
\end{cor}

\section{Relation to previous shuffle theorems}
\setcounter{subsection}{1}

Theorem \ref{thm:main-G} is formulated a little differently than the
classical and $(km,kn)$ shuffle theorems in \cite{BeGaSeXi16,
HaHaLoReUl05}, {although these also have an algebraic side and a
combinatorial side resembling ours.}  We now explain how to recover
them from our version by transforming each side of \eqref{e:main-G}
into its counterpart in the $(km,kn)$ and classical shuffle
conjectures.

For the $(km,kn)$ shuffle conjecture, we take the line {$y+px = s$} in
\eqref{e:line} to be a perturbation of the line from $(0,kn)$ to
$(km,0)$, with $s = kn$ and $r$ slightly larger than $km$.  Our
perturbed line has the same lattice points and paths under it as the
line from $(0,kn)$ to $(km,0)$, but it now has slope $-p$, where $p =
n/m-\epsilon $ for a small $\epsilon >0$.  The classical shuffle
conjectures in \cite{HaHaLoReUl05} are the special cases of the
$(km,kn)$ conjecture with $n=1$.  For these we perturb the line from
$(0,k)$ to $(km,0)$ in the same way.  Note that for our chosen line we
have $l=km+1$, and every lattice path $\lambda $ under it has
$b_{l}=0$ south steps at $x=km$.

The classical shuffle conjecture was formulated in \cite[Conjecture
6.2.2]{HaHaLoReUl05} as the identity
\begin{equation}\label{e:HHLRU-RHS}
\nabla^{m} e_{k} = \sum _{\lambda } 
\sum _{P\in \SSYT {((\lambda+(1^k))/\lambda)}}
t^{a(\lambda )} q^{\dinv _{m}(P)} x^{P}\,,
\end{equation}
where the sum is over lattice paths $\lambda$ below the bounding line,
and {$\dinv_m(P)$ is a statistic defined in \cite{HaHaLoReUl05},
attached to each labelling $P$ of the south steps in $\lambda $ by
non-negative integers} strictly increasing from south to north along
each vertical run.  The left hand side of \eqref{e:HHLRU-RHS} is
$e_{k}[-M X^{m,1}] \cdot 1$ by Corollary \ref{cor:nabla-ek}.  This
agrees with the left hand side $D_{b_{1},\ldots,b_{l}}\cdot 1$
of \eqref{e:main-G} by Corollary \ref{cor:ek-vs-Db}.

It was noted and used in \cite{HaHaLoReUl05} that the combinatorial
side of \eqref{e:HHLRU-RHS} can be phrased in terms of LLT
polynomials, but to explicitly match with our formulation requires
that we transform the right hand side of ~\eqref{e:main-G} as follows.
For the given $\nu(\lambda)=\sigma(\beta/\alpha)$, apply Proposition
\ref{prop:omega-G} to replace $\omega \Gcal_{\nu (\lambda)}(X;q^{-1})$
with $q^{-I(\nu (\lambda )^{R})} \Gcal_{\nu (\lambda)^{R}}(X;q)$.
Then writing out $\Gcal_{\nu(\lambda)^R}(X;q)$ term by term with
tableaux on the tuple $\nu(\lambda )^{R}$ of one-column diagrams using
Definition~\ref{def:G-nu} gives
\begin{equation}\label{e:main-G-cor}
D_{b_{1},\ldots,b_{l}} \cdot 1 = \sum _{\lambda } \sum _{T\in \SSYT
(\nu (\lambda )^{R})} t^{a(\lambda )} q^{\dinv _{p}(\lambda ) -I(\nu
(\lambda )^{R})+\inv(T)} x^{T}\,,
\end{equation}
where $\inv(T)$ is defined in \S\ref{ss:G-nu} to be the number of
attacking inversions.

By the construction, boxes in each column of $\nu (\lambda )^{R}$,
from top to bottom, correspond to south steps $u$ in a vertical run in
$\lambda $, from north to south.  Semistandard tableaux $T\in \SSYT
(\nu (\lambda )^{R})$ therefore biject with labellings $P_T\colon
\{\text{south steps in $\lambda $} \}\rightarrow \NN $ such that $P_T$
is strictly increasing from south to north on each vertical run in
$\lambda $; namely, $P_T\in\SSYT((\lambda+(1^k))/\lambda)$.
Changing~\eqref{e:main-G-cor} to instead sum over labellings, we can
match the right hand sides of ~\eqref{e:main-G-cor} and
\eqref{e:HHLRU-RHS} by showing that for $p=1/m-\epsilon$,
\begin{equation}\label{e:dinv-m}
\dinv _{m}(P_T) = \dinv _{p}(\lambda ) - I(\nu (\lambda )^{R}) +\inv(T).
\end{equation}

For any super tableau $T$, \cite[Corollary 6.4.2]{HaHaLoReUl05}
implies that $\dinv _{m}(P_T) = e_{\lambda }+\inv(T)$ for an offset
$e_{\lambda }$ not depending on $T$.  For the tableau $T_0$ with all
entries $\bar 1$, \cite[Lemma 6.3.3]{HaHaLoReUl05} gives that
$\dinv_m(P_{T_0})= b_m(\lambda)$, where we note that $b_m(\lambda)$
defined in \cite[(100)]{HaHaLoReUl05} is simply $\dinv_p(\lambda)$
with $p=1/m-\epsilon$.  Therefore, $e_\lambda = \dinv
_{m}(P_{T_0})-\inv(T_0) = \dinv_p(\lambda) - I(\nu(\lambda)^R)$
by~\eqref{e:inv2Tnot}.

In fact, there is a direct correspondence between
the combinatorics of $\dinv_{m}(P)$ for paths, as defined in 
\cite{HaHaLoReUl05}, and that of triples in negative tableaux on a tuple 
of one-row shapes, as considered in \S \ref{ss:LLT-tableaux}.

\begin{prop}\label{prop:dinv-m=h}
Let $\lambda$ be a lattice path from $(0,k)$ to $(km,0)$, lying 
weakly below the bounding line $y + p\, x = k$ with $p = 1/m-\epsilon$.
Let $\alpha$, $\beta$, $\sigma$ be the LLT data associated to
$\lambda$ for this $p$.
There is a weight-preserving bijection  from labellings 
$P\in\SSYT((\lambda+(1^k))/\lambda)$ to negative
tableaux $ S\in \SSYT_{-} (\beta/\alpha)$ such that
\begin{equation}\label{e:dinv-m=h}
\dinv_m(P)=h_\sigma(S).
\end{equation}
\end{prop}

\begin{proof}
labelling $P=P_T\in \SSYT((\lambda+(1^k))/\lambda)$ corresponds naturally
to a semistandard tableau $T\in\SSYT(\nu(\lambda)^R)$. Their statistics
are related by~\eqref{e:dinv-m}, into which we can
substitute \(\dinv_p(\lambda) = h_\sigma(\beta/\alpha)\) by
Proposition~\ref{prop:h=dinv}. 
The bijection $T\mapsto T^{R}$ in the proof of
Proposition~\ref{prop:omega-G} satisfies $\inv(T)-I(\nu(\lambda)^R) =
-\inv(T^R)$.
Hence, $\dinv _{m}(P_T) = h_\sigma(\beta/\alpha)-\inv(T^R)$. 
To complete the bijection, take
$S=T^R\circ\sigma$.
Then \(h_\sigma(\beta/\alpha)-\inv(T^R) = h_\sigma(S)\) by
Lemma~\ref{l:triples2dinv}, proving \eqref{e:dinv-m=h}.
\end{proof}

See Figure~\ref{fig:PF} for an example with $m=1$ and $p=1-\epsilon$.
Note that these values give $\sigma = w_0$ in the LLT data.

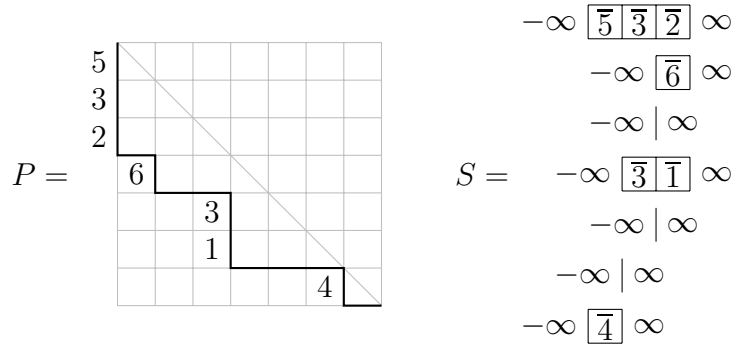
\begin{figure}
\[
\begin{tikzpicture}[scale = 0.5,baseline=1.75cm]
\draw [black!25] (0,0) grid (7,7);
\draw [black!25] (0,7)--(7,0);
\draw [thick] (0,7)--(0,4)--(1,4)--(1,3)--(3,3)--(3,1)--(6,1)--(6,0)--(7,0);
\node [left] at (0,6.5) {$5$};
\node [left] at (0,5.5) {$3$};
\node [left] at (0,4.5) {$2$};
\node [left] at (1,3.5) {$6$};
\node [left] at (3,2.5) {$3$};
\node [left] at (3,1.5) {$1$};
\node [left] at (6,0.5) {$4$};
\node [left] at (-1,3.5) {$P = $};
\end{tikzpicture}
\qquad 
\begin{tikzpicture}[scale=.45,baseline=2.25cm]
\draw (0,0) grid (1,1);
\node [left] at (0,.5) {$-\infty $};
\node at (.5,.5) {$\overline{4}$};
\node [right] at (1,.5) {$\infty $};

\draw (1,1.5) -- (1,2.5);
\node [left] at (1,2) {$-\infty $};
\node [right] at (1,2) {$\infty $};

\draw (2,3) -- (2,4);
\node [left] at (2,3.5) {$-\infty $};
\node [right] at (2,3.5) {$\infty $};

\draw [yshift=4.5cm] (1,0) grid (3,1);
\node [left] at (1,5) {$-\infty $};
\node at (1.5,5) {$\overline{3}$};
\node at (2.5,5) {$\overline{1}$};
\node [right] at (3,5) {$\infty $};

\draw (2,6) -- (2,7);
\node [left] at (2,6.5) {$-\infty $};
\node [right] at (2,6.5) {$\infty $};

\draw [yshift=7.5cm] (2,0) grid (3,1);
\node [left] at (2,8) {$-\infty $};
\node at (2.5,8) {$\overline{6}$};
\node [right] at (3,8) {$\infty $};

\draw [yshift=9cm] (0,0) grid (3,1);
\node [left] at (0,9.5) {$-\infty $};
\node at (.5,9.5) {$\overline{5}$};
\node at (1.5,9.5) {$\overline{3}$};
\node at (2.5,9.5) {$\overline{2}$};
\node [right] at (3,9.5) {$\infty $};

\node [left] at (-2,5) {$S = $};
\end{tikzpicture}
\]
\caption{\label{fig:PF}
Example of the bijection $P=P_T\leftrightarrow T\leftrightarrow
T^R\leftrightarrow S=T^R\circ \sigma$ in
Proposition~\ref{prop:dinv-m=h}, with $m=1$, $p=1-\epsilon$, $\sigma =
w_0$.  Letters in $S$ are ordered $\overline{1} > \overline{2} >
\cdots$.  We see $\dinv_1(P) = h_{w_0}(S) = 6$.}
\end{figure}

Next we turn to the
non-compositional $(km,kn)$ shuffle conjecture
from~\cite{BeGaSeXi16}.  Its symmetric function side is precisely
the Schiffmann algebra operator expression that we denote here by $e_{k}[-M
X^{m,n}]\cdot 1$.  By Corollary \ref{cor:ek-vs-Db}, this agrees with
the left hand side $D_{b_{1},\ldots,b_{l}}\cdot 1$ of
\eqref{e:main-G}.

The combinatorial
side of the  $(km,kn)$ shuffle conjecture can be written as in
\cite[\S 7]{BeGaSeXi16}, using notation defined there, as
\begin{equation}\label{e:BGSX-RHS}
\sum _{u} \sum _{\pi \in \operatorname{Park}(u)}
t^{\operatorname{area}(u)}\, q^{\dinv (u) + \operatorname{tdinv}(\pi ) -
\operatorname{maxtdinv(u )}} F_{\operatorname{ides}(\pi )}(x).
\end{equation}
Here $u$ encodes a north-east lattice path lying above the line from
$(0,0)$ to $(km,kn)$, $\operatorname{Park}(u)$ encodes the set of
standard Young tableaux on a tuple of columns corresponding to
vertical runs in the path encoded by $u$, and $F_{\gamma }(x)$ is a
Gessel fundamental quasi-symmetric function.

To make $u$ correspond to a lattice path $\lambda $ under the line
from $(0,kn)$ to $(km,0)$, as in \eqref{e:main-G-cor}, we need to flip
the picture over, replacing each entry $\pi (j)$ with $kn+1-\pi (j)$
so the resulting standard tableau on $\nu (\lambda )^{R}$ has columns
increasing upwards, as it should, instead of decreasing.  Using
\cite[Definition 7.1]{BeGaSeXi16} and taking account the modification
to $\pi $, we can translate notation in \eqref{e:BGSX-RHS} as follows:
$\operatorname{area}(u) = a(\lambda )$, $\operatorname{tdinv(\pi )} =
i(\pi )$, $\operatorname{maxtdinv(\pi )} = I(\nu (\lambda )^{R})$, and
$\dinv (u) = \dinv _{p}(\lambda )$, where $p = n/m - \epsilon $.

Finally, the definition of $\operatorname{ides}(\pi )$ becomes the
descent set of $\pi $ relative to the reading order on $\nu
(\lambda )^{R}$.  This implies that expanding
$F_{\operatorname{ides}(\pi )}(x)$ into monomials gives a sum with
semistandard tableaux $T$ in place of standard tableaux $\pi $ and
$x^{T}$ in place of $F_{\operatorname{ides}(\pi )}(x)$.  After these
substitutions, \eqref{e:BGSX-RHS} coincides with the right hand side
of \eqref{e:main-G-cor}.

\section{A positivity conjecture}
\label{s:positivity}

\subsection{}
\label{ss:positivity}

Theorem \ref{thm:main-G}, Corollary \ref{cor:Db-Hb} and
\cite[Proposition 5.3.1]{HaHaLoReUl05} imply that the
symmetric function
\begin{equation}\label{e:Db-again}
\omega (D_{\bb }\cdot 1) = \Hbold _{q,t}\left(\frac{x^{\bb }}{\prod
_{i} (1-q\, t\, x_{i}/x_{i+1})} \right)_{\pol }
\end{equation}
is $q,t$ Schur positive when $b_{i}$ is the number of south steps
along $x = i-1$ on the highest lattice path below a line with
endpoints on the positive $x$ and $y$ axes.  Computational evidence
leads us to conjecture that \eqref{e:Db-again} is $q,t$ Schur positive
under a more general geometric condition on $\bb $.

Let $C$ be a convex curve (meaning that the region above $C$ is
convex) in the first quadrant with endpoints $(r,0)$ and $(0,s)$ on
the positive $x$ and $y$ axes.  Let $\delta $ be the highest lattice
path from $(0,\lfloor s \rfloor)$ to $(\lfloor r \rfloor,0)$ weakly
below $C$. Let $b_{i}$ be the number of south steps in $\delta $ along
$x = i-1$ for $i = 1,\ldots,l$, where $l = \lfloor r \rfloor + 1$.
Algebraically, this means that there are real numbers $s_{0}\geq
s_{1}\geq \cdots \geq s_{l}=0$ with weakly decreasing differences
$s_{i-1}-s_{i}\geq s_{i}-s_{i+1}$, such that $b_{i} = \lfloor s_{i-1}
\rfloor -\lfloor s_{i} \rfloor$.

Note that if $\delta $ is the highest path strictly below a convex
curve $C'$, then it is also the highest lattice path weakly below a
slightly lower curve $C$, and vice versa, so it doesn't matter whether
we use `weakly below' or `strictly below' to formulate the condition
on $\delta $.

\begin{conj}\label{conj:positivity}
When $b_{i}$ is the number of south steps along $x = i-1$ in the
highest lattice path below a convex curve, as above, the symmetric
function in \eqref{e:Db-again} is a linear combination of Schur
functions with coefficients in $\NN [q,t]$.
\end{conj}

At $q=1$, the $q$-Kostka coefficients reduce to $K_{\lambda ,\mu }(1)
= K_{\lambda ,\mu } = \langle s_{\lambda },h_{\mu } \rangle$.  Hence,
the Hall-Littlewood symmetrization operator reduces to $\Hbold
_{q}(x^{\mu })_{\pol }|_{q=1} = h_{\mu }(x)$ if $\mu _{i}\geq 0$ for
all $i$, and otherwise $\Hbold _{q}(x^{\mu })_{\pol } = 0$.  At $q=1$,
the factors containing $t$ in \eqref{e:Hqt} cancel, so $\Hbold _{q,t}$
reduces to the same thing as $\Hbold _{q}$.

It follows that \eqref{e:Db-again} specializes at $q = 1$ to
\begin{equation}\label{e:Db-q=1}
\omega (D_{\bb }\cdot 1)|_{q=1} = \sum _{a_{1},\ldots,a_{l-1}\geq 0}
t^{|\aA |} h_{\bb
+(a_{1},a_{2}-a_{1},\ldots,a_{l-1}-a_{l-2},-a_{l-1})},
\end{equation}
with the convention that $h_{\mu } = 0$ if $\mu _{i}<0$ for any $i$.
As in Theorem \ref{thm:main-G}, the index $\bb
+(a_{1},a_{2}-a_{1},\ldots,a_{l-1}-a_{l-2},-a_{l-1})$ is non-negative
precisely when it is the sequence $b(\lambda )$ of lengths of south
runs in a lattice path $\lambda $ lying below the path $\delta $ whose
south runs are given by $\bb $.  Here $a_{i}$ is the number of lattice
squares in column $i$ between $\lambda $ and $\delta $, so $|\aA |$ is
the area $a(\lambda )$ enclosed between the two paths.  This gives a
combinatorial expression
\begin{equation}\label{e:Db-q=1-bis}
\omega (D_{\bb }\cdot 1)|_{q=1} = \sum _{\lambda } t^{a(\lambda )}
h_{b(\lambda )},
\end{equation}
for \eqref{e:Db-again} at $q=1$, which is positive in terms of
complete homogeneous symmetric functions $h_{\lambda }$, hence $t$
Schur positive.  We may conjecture that when the hypothesis of
Conjecture \ref{conj:positivity} holds, $\omega (D_{\bb }\cdot 1)$ is
given by some Schur positive combinatorial $q$-analog of
\eqref{e:Db-q=1-bis}, but it remains an open problem to formulate such
a conjecture precisely.

Of course, \eqref{e:Db-q=1-bis} cannot be considered evidence for
Conjecture \ref{conj:positivity}, since \eqref{e:Db-q=1-bis} holds for
any $\bb \geq 0$, whether the convexity hypothesis is satisfied or
not.

\subsection{Relation to previous conjectures}
\label{ss:Negut-conjecture}

The generalized $q,t$-Catalan numbers $C_\bb(q,t) = \langle
s_{(|\bb|)}(X), \omega (D_\bb \cdot 1) \rangle$ from Definition
\ref{def:qt cat number} coincide with the functions denoted
$F(b_2,\dots, b_l)$ in \cite{GHRS}, where several equivalent
expressions for them were obtained.  To see that $C_{\bb }(q,t) =
F(b_{2},\ldots,b_{l})$, one can compare the formula in
Proposition~\ref{prop:C(q,t)}, below, with the equation just before
(2.6) in \cite{GHRS}.  It was also shown in \cite{GHRS} that this
quantity does not depend on $b_{1}$, hence the notation $F(b_2,\dots,
b_l)$.

Conjecture \ref{conj:positivity} implies a conjecture of Negut,
announced in \cite{GHRS}, which asserts that $C_{\bb }(q,t) \in \NN
[q,t]$ when $b_{2}\geq \cdots \geq b_{l}$.  Conjecture
\ref{conj:positivity} is stronger than Negut's conjecture in two ways:
the weight $\bb $ is generalized from a partition to the highest path
below a convex curve, and the coefficient of $s_{(|\bb|)}(X)$ in
$\omega (D_\bb \cdot 1)$ is generalized to the coefficient of any
Schur function.

\begin{prop}
\label{prop:C(q,t)}
The generalized $q,t$-Catalan number $C_\bb(q,t)$ has the following
description as a series coefficient:
\begin{align}
\label{e:for GHRS}
C_\bb(q,t) =
\langle z^{-\bb} \rangle
\,
\prod_{i=1}^l \frac{1}{1-z_i^{-1}} \,
\prod_{i=1}^{l-1}\frac{1}{1-q\, t\, z_{i}/z_{i+1}}
 \, \prod _{i<j} \frac{(1-z_{i}/z_{j})(1-q\, t\,
z_{i}/z_{j})}{(1-q \, z_{i}/z_{j})(1-t \, z_{i}/z_{j})} .
\end{align}
\end{prop}
\begin{proof}
From \eqref{e:Dz-on-1} we have
\begin{align*}
\label{e:Dz-on-1 v2}
\omega (D_\bb \cdot 1)
= \langle z^{0} \rangle \,
 \frac{z^{\bb }}{\prod_{i=1}^{l-1} (1-q\, t\, z_{i}/z_{i+1})} \,
 \prod _{i<j} \frac{1-q\, t\,
z_{i}/z_{j}}{(1-q \, z_{i}/z_{j})(1-t \, z_{i}/z_{j})} \,
\Omega[\overline{Z} X] \, \prod _{i<j}(1-z_{i}/z_{j}).
\end{align*}
Specializing  $X = 1$ gives the result.
\end{proof}


\bibliographystyle{hamsplain} \bibliography{shuffle}

\end{document}